\documentclass[12pt]{amsart}
\usepackage[latin1]{inputenc}
\usepackage{amsbsy}
\usepackage{amscd} 
\usepackage{amsfonts}
\usepackage{amsgen} 
\usepackage{amsmath}
\usepackage{amsopn} 
\usepackage{amssymb}
\usepackage{amstext}
\usepackage{amsthm} 
\usepackage{amsxtra}
\usepackage{rotating}
\usepackage{tikz}
\theoremstyle{plain} 
\newtheorem{thm}{Theorem}[section]
\newtheorem{prop}[thm]{Proposition}
\newtheorem{lem}[thm]{Lemma}
\newtheorem{cor}[thm]{Corollary}
\theoremstyle{definition}
\newtheorem{defn}[thm]{Definition}
\newtheorem{rem}[thm]{Remark}
\newtheorem{ex}[thm]{Example}

\numberwithin{equation}{section}

\renewcommand{\theta}{\vartheta}
\renewcommand{\phi}{\varphi}
\renewcommand{\epsilon}{\varepsilon}
\renewcommand{\subset}{\subseteq}
\renewcommand{\supset}{\supseteq}

\newcommand{\N}{\mathbb N}
\newcommand{\Z}{\mathbb Z}

\newcommand{\CC}{\mathcal C}
\newcommand{\DD}{\mathcal D}
\newcommand{\OOO}{\mathcal O}
\newcommand{\HHH}{\mathcal H}
\newcommand{\BBB}{\mathcal B}
\newcommand{\SSS}{\mathcal S}

\usepackage{calc}
\newcounter{PartitionDepth}
\newcounter{PartitionLength}
\newcommand{\parti}[2]{
 \begin{picture}(#2,#1)
 \setcounter{PartitionDepth}{-1-#1}
 \put(#2,\thePartitionDepth){\line(0,1){#1}}
 \end{picture}}
\newcommand{\partii}[3]{
 \begin{picture}(#3,#1)
 \setcounter{PartitionLength}{#3-#2}
 \setcounter{PartitionDepth}{-1-#1}
 \put(#2,\thePartitionDepth){\line(0,1){#1}}     
 \put(#3,\thePartitionDepth){\line(0,1){#1}}
 \put(#2,\thePartitionDepth){\line(1,0){\thePartitionLength}}
 \end{picture}}
\newcommand{\partiii}[4]{
 \begin{picture}(#4,#1)
 \setcounter{PartitionLength}{#4-#2}
 \setcounter{PartitionDepth}{-1-#1}
 \put(#2,\thePartitionDepth){\line(0,1){#1}}
 \put(#3,\thePartitionDepth){\line(0,1){#1}}
 \put(#4,\thePartitionDepth){\line(0,1){#1}}
 \put(#2,\thePartitionDepth){\line(1,0){\thePartitionLength}} 
 \end{picture}}

\newcommand{\upparti}[2]{
 \begin{picture}(#2,#1)
 \setcounter{PartitionDepth}{#1}
 \put(#2,0){\line(0,1){#1}}
 \end{picture}}
\newcommand{\uppartii}[3]{
 \begin{picture}(#3,#1)
 \setcounter{PartitionLength}{#3-#2}
 \setcounter{PartitionDepth}{#1}
 \put(#2,0){\line(0,1){#1}}     
 \put(#3,0){\line(0,1){#1}}
 \put(#2,\thePartitionDepth){\line(1,0){\thePartitionLength}}
 \end{picture}}
\newcommand{\uppartiii}[4]{
 \begin{picture}(#4,#1)
 \setcounter{PartitionLength}{#4-#2}
 \setcounter{PartitionDepth}{#1}
 \put(#2,0){\line(0,1){#1}}
 \put(#3,0){\line(0,1){#1}}
 \put(#4,0){\line(0,1){#1}}
 \put(#2,\thePartitionDepth){\line(1,0){\thePartitionLength}} 
 \end{picture}}
\newcommand{\uppartiv}[5]{
 \begin{picture}(#5,#1)
 \setcounter{PartitionLength}{#5-#2}
 \setcounter{PartitionDepth}{#1}
 \put(#2,0){\line(0,1){#1}}
 \put(#3,0){\line(0,1){#1}}
 \put(#4,0){\line(0,1){#1}}
 \put(#5,0){\line(0,1){#1}}
 \put(#2,\thePartitionDepth){\line(1,0){\thePartitionLength}} 
 \end{picture}}

\newsavebox{\boxpaarpart}
   \savebox{\boxpaarpart}
   { \begin{picture}(0.7,0.5)
     \put(0,0){\line(0,1){0.4}}
     \put(0.5,0){\line(0,1){0.4}}
     \put(0,0.4){\line(1,0){0.5}}
     \end{picture}}
\newsavebox{\boxbaarpart}
   \savebox{\boxbaarpart}
   { \begin{picture}(0.7,0.5)
     \put(0,0){\line(0,1){0.4}}
     \put(0.5,0){\line(0,1){0.4}}
     \put(0,0){\line(1,0){0.5}}
     \end{picture}}
\newsavebox{\boxdreipart}
   \savebox{\boxdreipart}
   { \begin{picture}(1,0.5)
     \put(0,0){\line(0,1){0.4}}
     \put(0.4,0){\line(0,1){0.4}}
     \put(0.8,0){\line(0,1){0.4}}
     \put(0,0.4){\line(1,0){0.8}}
     \end{picture}}
\newsavebox{\boxvierpart}
   \savebox{\boxvierpart}
   { \begin{picture}(1.4,0.5)
     \put(0,0){\line(0,1){0.4}}
     \put(0.4,0){\line(0,1){0.4}}
     \put(0.8,0){\line(0,1){0.4}}
     \put(1.2,0){\line(0,1){0.4}}
     \put(0,0.4){\line(1,0){1.2}}
     \end{picture}}
\newsavebox{\boxvierpartrot}
   \savebox{\boxvierpartrot}
   { \begin{picture}(0.5,0.8)
     \put(0,-0.2){\line(0,1){0.3}}
     \put(0.4,-0.2){\line(0,1){0.3}}
     \put(0,0.1){\line(1,0){0.4}}
     \put(0,0.4){\line(0,1){0.3}}
     \put(0.4,0.4){\line(0,1){0.3}}
     \put(0,0.4){\line(1,0){0.4}}
     \put(0.2,0.1){\line(0,1){0.3}}
     \end{picture}}
\newsavebox{\boxvierpartrotdrei}
   \savebox{\boxvierpartrotdrei}
   { \begin{picture}(1,0.8)
     \put(0,-0.2){\line(0,1){0.3}}
     \put(0.4,-0.2){\line(0,1){0.8}}
     \put(0.8,-0.2){\line(0,1){0.3}}
     \put(0,0.1){\line(1,0){0.8}}
     \end{picture}}
\newsavebox{\boxcrosspart}
   \savebox{\boxcrosspart}
   { \begin{picture}(0.5,0.8)
     \put(0,-0.2){\line(1,2){0.4}}
     \put(0.4,-0.2){\line(-1,2){0.4}}
     \end{picture}}
\newsavebox{\boxhalflibpart}
   \savebox{\boxhalflibpart}
   { \begin{picture}(1,0.8)
     \put(0,-0.2){\line(1,1){0.8}}
     \put(0.8,-0.2){\line(-1,1){0.8}}
     \put(0.4,-0.2){\line(0,1){0.8}}
     \end{picture}}
\newsavebox{\boxpositioner} 
   \savebox{\boxpositioner}
   { \begin{picture}(1.4,0.5)
     \put(0,0){\line(0,1){0.3}}
     \put(0.4,0){\line(0,1){0.6}}
     \put(0.8,0){\line(0,1){0.3}}
     \put(1.2,0){\line(0,1){0.6}}
     \put(0.4,0.6){\line(1,0){0.8}}
     \end{picture}}
\newsavebox{\boxfatcross} 
   \savebox{\boxfatcross}
   { \begin{picture}(1.4,1.3)
     \put(0,-0.2){\line(0,1){0.3}}
     \put(0.4,-0.2){\line(0,1){0.3}}
     \put(0.8,-0.2){\line(0,1){0.3}}
     \put(1.2,-0.2){\line(0,1){0.3}}
     \put(0,0.1){\line(1,0){0.4}}
     \put(0.8,0.1){\line(1,0){0.4}}
     \put(0,0.9){\line(0,1){0.3}}
     \put(0.4,0.9){\line(0,1){0.3}}
     \put(0.8,0.9){\line(0,1){0.3}}
     \put(1.2,0.9){\line(0,1){0.3}}
     \put(0,0.9){\line(1,0){0.4}}
     \put(0.8,0.9){\line(1,0){0.4}}
     \put(0.2,0.1){\line(1,1){0.8}}
     \put(0.2,0.9){\line(1,-1){0.8}}
     \end{picture}}
\newsavebox{\boxprimarypart} 
   \savebox{\boxprimarypart}
   { \begin{picture}(1,1.3)
     \put(0,-0.2){\line(0,1){0.3}}
     \put(0.4,-0.2){\line(0,1){0.3}}
     \put(0.8,-0.2){\line(0,1){0.3}}
     \put(0.4,0.1){\line(1,0){0.4}}
     \put(0,0.9){\line(0,1){0.3}}
     \put(0.4,0.9){\line(0,1){0.3}}
     \put(0.8,0.9){\line(0,1){0.3}}
     \put(0,0.9){\line(1,0){0.4}}
     \put(0,0.1){\line(1,1){0.8}}
     \put(0.2,0.9){\line(1,-2){0.4}}
     \end{picture}}
\newcommand{\paarpart}{\usebox{\boxpaarpart}}

\newcommand{\vierpart}{\usebox{\boxvierpart}}

\newcommand{\legpart}{\usebox{\boxpositioner}}

\newcommand{\idpart}{|}
\newcommand{\singleton}{\uparrow}

\newcommand{\downsubset}{\begin{turn}{270}$\subset$\end{turn}}

\newcommand{\twocol}{\circ\bullet}

\newsavebox{\boxidpartww}
   \savebox{\boxidpartww}
   { \begin{picture}(0.5,1.2)
     \put(0.2,0.15){\line(0,1){0.7}}
     \put(0,-0.2){$\circ$}
     \put(0,0.8){$\circ$}
     \end{picture}}
\newsavebox{\boxidpartbw}
   \savebox{\boxidpartbw}
   { \begin{picture}(0.5,1.2)
     \put(0.2,0.15){\line(0,1){0.7}}
     \put(0,-0.2){$\circ$}
     \put(0,0.8){$\bullet$}
     \end{picture}}
\newsavebox{\boxidpartwb}
   \savebox{\boxidpartwb}
   { \begin{picture}(0.5,1.2)
     \put(0.2,0.15){\line(0,1){0.7}}
     \put(0,-0.2){$\bullet$}
     \put(0,0.8){$\circ$}
     \end{picture}}
\newsavebox{\boxidpartbb}
   \savebox{\boxidpartbb}
   { \begin{picture}(0.5,1.2)
     \put(0.2,0.15){\line(0,1){0.7}}
     \put(0,-0.2){$\bullet$}
     \put(0,0.8){$\bullet$}
     \end{picture}}

\newsavebox{\boxidpartsingletonbb}
   \savebox{\boxidpartsingletonbb}
   { \begin{picture}(0.5,1.2)
     \put(0.2,0.15){\line(0,1){0.2}}
     \put(0.2,0.65){\line(0,1){0.2}}
     \put(0,-0.2){$\bullet$}
     \put(0,0.8){$\bullet$}
     \end{picture}}
\newsavebox{\boxidpartsingletonww}
   \savebox{\boxidpartsingletonww}
   { \begin{picture}(0.5,1.2)
     \put(0.2,0.15){\line(0,1){0.2}}
     \put(0.2,0.65){\line(0,1){0.2}}
     \put(0,-0.2){$\circ$}
     \put(0,0.8){$\circ$}
     \end{picture}}
\newsavebox{\boxidpartsingletonbw}
   \savebox{\boxidpartsingletonbw}
   { \begin{picture}(0.5,1.2)
     \put(0.2,0.15){\line(0,1){0.2}}
     \put(0.2,0.65){\line(0,1){0.2}}
     \put(0,-0.2){$\circ$}
     \put(0,0.8){$\bullet$}
     \end{picture}}
\newsavebox{\boxidpartsingletonwb}
   \savebox{\boxidpartsingletonwb}
   { \begin{picture}(0.5,1.2)
     \put(0.2,0.15){\line(0,1){0.2}}
     \put(0.2,0.65){\line(0,1){0.2}}
     \put(0,-0.2){$\bullet$}
     \put(0,0.8){$\circ$}
     \end{picture}}
     
\newsavebox{\boxpaarpartbb}
   \savebox{\boxpaarpartbb}
   { \begin{picture}(1,1)
     \put(0.2,0.2){\line(0,1){0.4}}
     \put(0.7,0.2){\line(0,1){0.4}}
     \put(0.2,0.6){\line(1,0){0.5}}
     \put(0.5,-0.2){$\bullet$}
     \put(0,-0.2){$\bullet$}
     \end{picture}}
\newsavebox{\boxpaarpartww}
   \savebox{\boxpaarpartww}
   { \begin{picture}(1,1)
     \put(0.2,0.2){\line(0,1){0.4}}
     \put(0.7,0.2){\line(0,1){0.4}}
     \put(0.2,0.6){\line(1,0){0.5}}
     \put(0.5,-0.2){$\circ$}
     \put(0,-0.2){$\circ$}
     \end{picture}}
\newsavebox{\boxpaarpartbw}
   \savebox{\boxpaarpartbw}
   { \begin{picture}(1,1)
     \put(0.2,0.2){\line(0,1){0.4}}
     \put(0.7,0.2){\line(0,1){0.4}}
     \put(0.2,0.6){\line(1,0){0.5}}
     \put(0.5,-0.2){$\circ$}
     \put(0,-0.2){$\bullet$}
     \end{picture}}
\newsavebox{\boxpaarpartwb}
   \savebox{\boxpaarpartwb}
   { \begin{picture}(1,1)
     \put(0.2,0.2){\line(0,1){0.4}}
     \put(0.7,0.2){\line(0,1){0.4}}
     \put(0.2,0.6){\line(1,0){0.5}}
     \put(0.5,-0.2){$\bullet$}
     \put(0,-0.2){$\circ$}
     \end{picture}}
\newsavebox{\boxbaarpartbb}
   \savebox{\boxbaarpartbb}
   { \begin{picture}(1,1)
     \put(0.2,-0.1){\line(0,1){0.4}}
     \put(0.7,-0.1){\line(0,1){0.4}}
     \put(0.2,-0.1){\line(1,0){0.5}}
     \put(0.5,0.3){$\bullet$}
     \put(0,0.3){$\bullet$}
     \end{picture}}
\newsavebox{\boxbaarpartww}
   \savebox{\boxbaarpartww}
   { \begin{picture}(1,1)
     \put(0.2,-0.1){\line(0,1){0.4}}
     \put(0.7,-0.1){\line(0,1){0.4}}
     \put(0.2,-0.1){\line(1,0){0.5}}
     \put(0.5,0.3){$\circ$}
     \put(0,0.3){$\circ$}
     \end{picture}}
\newsavebox{\boxbaarpartbw}
   \savebox{\boxbaarpartbw}
   { \begin{picture}(1,1)
     \put(0.2,-0.1){\line(0,1){0.4}}
     \put(0.7,-0.1){\line(0,1){0.4}}
     \put(0.2,-0.1){\line(1,0){0.5}}
     \put(0.5,0.3){$\circ$}
     \put(0,0.3){$\bullet$}
     \end{picture}}
\newsavebox{\boxbaarpartwb}
   \savebox{\boxbaarpartwb}
   { \begin{picture}(1,1)
     \put(0.2,-0.1){\line(0,1){0.4}}
     \put(0.7,-0.1){\line(0,1){0.4}}
     \put(0.2,-0.1){\line(1,0){0.5}}
     \put(0.5,0.3){$\bullet$}
     \put(0,0.3){$\circ$}
     \end{picture}}

\newsavebox{\boxcutpaarpartbb}
   \savebox{\boxcutpaarpartbb}
   { \begin{picture}(1,1)
     \put(0,0.2){\line(0,1){0.4}}
     \put(0.8,0.2){\line(0,1){0.4}}
     \put(0,0.6){\line(1,0){0.2}}
     \put(0.6,0.6){\line(1,0){0.2}}
     \put(0.6,-0.2){$\bullet$}
     \put(-0.2,-0.2){$\bullet$}
     \end{picture}}
\newsavebox{\boxcutpaarpartww}
   \savebox{\boxcutpaarpartww}
   { \begin{picture}(1,1)
     \put(0,0.2){\line(0,1){0.4}}
     \put(0.8,0.2){\line(0,1){0.4}}
     \put(0,0.6){\line(1,0){0.2}}
     \put(0.6,0.6){\line(1,0){0.2}}
     \put(0.6,-0.2){$\circ$}
     \put(-0.2,-0.2){$\circ$}
     \end{picture}}
\newsavebox{\boxcutpaarpartbw}
   \savebox{\boxcutpaarpartbw}
   { \begin{picture}(1,1)
     \put(0,0.2){\line(0,1){0.4}}
     \put(0.8,0.2){\line(0,1){0.4}}
     \put(0,0.6){\line(1,0){0.2}}
     \put(0.6,0.6){\line(1,0){0.2}}
     \put(0.6,-0.2){$\circ$}
     \put(-0.2,-0.2){$\bullet$}
     \end{picture}}
\newsavebox{\boxcutpaarpartwb}
   \savebox{\boxcutpaarpartwb}
   { \begin{picture}(1,1)
     \put(0,0.2){\line(0,1){0.4}}
     \put(0.8,0.2){\line(0,1){0.4}}
     \put(0,0.6){\line(1,0){0.2}}
     \put(0.6,0.6){\line(1,0){0.2}}
     \put(0.6,-0.2){$\bullet$}
     \put(-0.2,-0.2){$\circ$}
     \end{picture}}

\newsavebox{\boxsingletonw}
   \savebox{\boxsingletonw}
   { \begin{picture}(0.5, 1)
     \put(0,0.3){$\uparrow$}
     \put(0,-0.2){$\circ$}
     \end{picture}}
\newsavebox{\boxsingletonb}
   \savebox{\boxsingletonb}
   { \begin{picture}(0.5, 1)
     \put(0,0.3){$\uparrow$}
     \put(0,-0.2){$\bullet$}
     \end{picture}}
\newsavebox{\boxdownsingletonw}
   \savebox{\boxdownsingletonw}
   { \begin{picture}(0.5, 1)
     \put(0,0){$\downarrow$}
     \put(0,0.6){$\circ$}
     \end{picture}}
\newsavebox{\boxdownsingletonb}
   \savebox{\boxdownsingletonb}
   { \begin{picture}(0.5, 1)
     \put(0,0){$\downarrow$}
     \put(0,0.6){$\bullet$}
     \end{picture}}

\newsavebox{\boxvierpartwbwb}
   \savebox{\boxvierpartwbwb}
   { \begin{picture}(1.7,1)
     \put(0.2,0.2){\line(0,1){0.4}}
     \put(0.6,0.2){\line(0,1){0.4}}
     \put(1,0.2){\line(0,1){0.4}}
     \put(1.4,0.2){\line(0,1){0.4}}
     \put(0.2,0.6){\line(1,0){1.2}}
     \put(0,-0.2){$\circ$}
     \put(0.4,-0.2){$\bullet$}
     \put(0.8,-0.2){$\circ$}
     \put(1.2,-0.2){$\bullet$}
     \end{picture}}
\newsavebox{\boxvierpartbwbw}
   \savebox{\boxvierpartbwbw}
   { \begin{picture}(1.7,1)
     \put(0.2,0.2){\line(0,1){0.4}}
     \put(0.6,0.2){\line(0,1){0.4}}
     \put(1,0.2){\line(0,1){0.4}}
     \put(1.4,0.2){\line(0,1){0.4}}
     \put(0.2,0.6){\line(1,0){1.2}}
     \put(0,-0.2){$\bullet$}
     \put(0.4,-0.2){$\circ$}
     \put(0.8,-0.2){$\bullet$}
     \put(1.2,-0.2){$\circ$}
     \end{picture}}
\newsavebox{\boxvierpartwwbb}
   \savebox{\boxvierpartwwbb}
   { \begin{picture}(1.7,1)
     \put(0.2,0.2){\line(0,1){0.4}}
     \put(0.6,0.2){\line(0,1){0.4}}
     \put(1,0.2){\line(0,1){0.4}}
     \put(1.4,0.2){\line(0,1){0.4}}
     \put(0.2,0.6){\line(1,0){1.2}}
     \put(0,-0.2){$\circ$}
     \put(0.4,-0.2){$\circ$}
     \put(0.8,-0.2){$\bullet$}
     \put(1.2,-0.2){$\bullet$}
     \end{picture}}
\newsavebox{\boxvierpartwbbw}
   \savebox{\boxvierpartwbbw}
   { \begin{picture}(1.7,1)
     \put(0.2,0.2){\line(0,1){0.4}}
     \put(0.6,0.2){\line(0,1){0.4}}
     \put(1,0.2){\line(0,1){0.4}}
     \put(1.4,0.2){\line(0,1){0.4}}
     \put(0.2,0.6){\line(1,0){1.2}}
     \put(0,-0.2){$\circ$}
     \put(0.4,-0.2){$\bullet$}
     \put(0.8,-0.2){$\bullet$}
     \put(1.2,-0.2){$\circ$}
     \end{picture}}
\newsavebox{\boxvierpartbwwb}
   \savebox{\boxvierpartbwwb}
   { \begin{picture}(1.7,1)
     \put(0.2,0.2){\line(0,1){0.4}}
     \put(0.6,0.2){\line(0,1){0.4}}
     \put(1,0.2){\line(0,1){0.4}}
     \put(1.4,0.2){\line(0,1){0.4}}
     \put(0.2,0.6){\line(1,0){1.2}}
     \put(0,-0.2){$\bullet$}
     \put(0.4,-0.2){$\circ$}
     \put(0.8,-0.2){$\circ$}
     \put(1.2,-0.2){$\bullet$}
     \end{picture}}
\newsavebox{\boxvierpartrotwbwb}
   \savebox{\boxvierpartrotwbwb}
   { \begin{picture}(1,1.2)
     \put(0.2,0.15){\line(0,1){0.2}}
     \put(0.2,0.65){\line(0,1){0.2}}
     \put(0.2,0.65){\line(1,0){0.5}}
     \put(0.2,0.35){\line(1,0){0.5}}
     \put(0.45,0.35){\line(0,1){0.3}}
     \put(0.7,0.15){\line(0,1){0.2}}
     \put(0.7,0.65){\line(0,1){0.2}}
     \put(0,0.8){$\circ$}
     \put(0.5,0.8){$\bullet$}
     \put(0,-0.2){$\circ$}
     \put(0.5,-0.2){$\bullet$}
     \end{picture}}
\newsavebox{\boxvierpartrotbwwb}
   \savebox{\boxvierpartrotbwwb}
   { \begin{picture}(1,1.2)
     \put(0.2,0.15){\line(0,1){0.2}}
     \put(0.2,0.65){\line(0,1){0.2}}
     \put(0.2,0.65){\line(1,0){0.5}}
     \put(0.2,0.35){\line(1,0){0.5}}
     \put(0.45,0.35){\line(0,1){0.3}}
     \put(0.7,0.15){\line(0,1){0.2}}
     \put(0.7,0.65){\line(0,1){0.2}}
     \put(0,0.8){$\bullet$}
     \put(0.5,0.8){$\circ$}
     \put(0,-0.2){$\circ$}
     \put(0.5,-0.2){$\bullet$}
     \end{picture}}
\newsavebox{\boxvierpartrotwbbw}
   \savebox{\boxvierpartrotwbbw}
   { \begin{picture}(1,1.2)
     \put(0.2,0.15){\line(0,1){0.2}}
     \put(0.2,0.65){\line(0,1){0.2}}
     \put(0.2,0.65){\line(1,0){0.5}}
     \put(0.2,0.35){\line(1,0){0.5}}
     \put(0.45,0.35){\line(0,1){0.3}}
     \put(0.7,0.15){\line(0,1){0.2}}
     \put(0.7,0.65){\line(0,1){0.2}}
     \put(0,0.8){$\circ$}
     \put(0.5,0.8){$\bullet$}
     \put(0,-0.2){$\bullet$}
     \put(0.5,-0.2){$\circ$}
     \end{picture}}
\newsavebox{\boxvierpartrotbwbw}
   \savebox{\boxvierpartrotbwbw}
   { \begin{picture}(1,1.2)
     \put(0.2,0.15){\line(0,1){0.2}}
     \put(0.2,0.65){\line(0,1){0.2}}
     \put(0.2,0.65){\line(1,0){0.5}}
     \put(0.2,0.35){\line(1,0){0.5}}
     \put(0.45,0.35){\line(0,1){0.3}}
     \put(0.7,0.15){\line(0,1){0.2}}
     \put(0.7,0.65){\line(0,1){0.2}}
     \put(0,0.8){$\bullet$}
     \put(0.5,0.8){$\circ$}
     \put(0,-0.2){$\bullet$}
     \put(0.5,-0.2){$\circ$}
     \end{picture}}
\newsavebox{\boxvierpartrotwwww}
   \savebox{\boxvierpartrotwwww}
   { \begin{picture}(1,1.2)
     \put(0.2,0.15){\line(0,1){0.2}}
     \put(0.2,0.65){\line(0,1){0.2}}
     \put(0.2,0.65){\line(1,0){0.5}}
     \put(0.2,0.35){\line(1,0){0.5}}
     \put(0.45,0.35){\line(0,1){0.3}}
     \put(0.7,0.15){\line(0,1){0.2}}
     \put(0.7,0.65){\line(0,1){0.2}}
     \put(0,0.8){$\circ$}
     \put(0.5,0.8){$\circ$}
     \put(0,-0.2){$\circ$}
     \put(0.5,-0.2){$\circ$}
     \end{picture}}
\newsavebox{\boxvierpartrotbbbb}
   \savebox{\boxvierpartrotbbbb}
   { \begin{picture}(1,1.2)
     \put(0.2,0.15){\line(0,1){0.2}}
     \put(0.2,0.65){\line(0,1){0.2}}
     \put(0.2,0.65){\line(1,0){0.5}}
     \put(0.2,0.35){\line(1,0){0.5}}
     \put(0.45,0.35){\line(0,1){0.3}}
     \put(0.7,0.15){\line(0,1){0.2}}
     \put(0.7,0.65){\line(0,1){0.2}}
     \put(0,0.8){$\bullet$}
     \put(0.5,0.8){$\bullet$}
     \put(0,-0.2){$\bullet$}
     \put(0.5,-0.2){$\bullet$}
     \end{picture}}
\newsavebox{\boxvierpartwrotwwb}
   \savebox{\boxvierpartwrotwwb}
   { \begin{picture}(1.4,1.2)
     \put(0.2,0.15){\line(0,1){0.2}}
     \put(0.7,0.15){\line(0,1){0.2}}
     \put(1.2,0.15){\line(0,1){0.7}}
     \put(0.2,0.35){\line(1,0){1}}
     \put(1,0.8){$\circ$}
     \put(0,-0.2){$\circ$}
     \put(0.5,-0.2){$\circ$}
     \put(1,-0.2){$\bullet$}
     \end{picture}}
\newsavebox{\boxvierpartbrotbbw}
   \savebox{\boxvierpartbrotbbw}
   { \begin{picture}(1.4,1.2)
     \put(0.2,0.15){\line(0,1){0.2}}
     \put(0.7,0.15){\line(0,1){0.2}}
     \put(1.2,0.15){\line(0,1){0.7}}
     \put(0.2,0.35){\line(1,0){1}}
     \put(1,0.8){$\bullet$}
     \put(0,-0.2){$\bullet$}
     \put(0.5,-0.2){$\bullet$}
     \put(1,-0.2){$\circ$}
     \end{picture}}
\newsavebox{\boxvierpartbrotwbb}
   \savebox{\boxvierpartbrotwbb}
   { \begin{picture}(1.4,1.2)
     \put(0.2,0.15){\line(0,1){0.7}}
     \put(0.7,0.15){\line(0,1){0.2}}
     \put(1.2,0.15){\line(0,1){0.2}}
     \put(0.2,0.35){\line(1,0){1}}
     \put(0,0.8){$\bullet$}
     \put(0,-0.2){$\circ$}
     \put(0.5,-0.2){$\bullet$}
     \put(1,-0.2){$\bullet$}
     \end{picture}}     
\newsavebox{\boxvierpartwrotbww}
   \savebox{\boxvierpartwrotbww}
   { \begin{picture}(1.4,1.2)
     \put(0.2,0.15){\line(0,1){0.7}}
     \put(0.7,0.15){\line(0,1){0.2}}
     \put(1.2,0.15){\line(0,1){0.2}}
     \put(0.2,0.35){\line(1,0){1}}
     \put(0,0.8){$\circ$}
     \put(0,-0.2){$\bullet$}
     \put(0.5,-0.2){$\circ$}
     \put(1,-0.2){$\circ$}
     \end{picture}}     

\newsavebox{\boxnestpaarpartbwbb}
   \savebox{\boxnestpaarpartbwbb}
   { \begin{picture}(1.7,1)
     \put(0.2,0.2){\line(0,1){0.8}}
     \put(0.6,0.2){\line(0,1){0.4}}
     \put(1,0.2){\line(0,1){0.4}}
     \put(1.4,0.2){\line(0,1){0.8}}
     \put(0.2,1){\line(1,0){1.2}}
     \put(0.6,0.6){\line(1,0){0.4}}
     \put(0,-0.2){$\bullet$}
     \put(0.4,-0.2){$\circ$}
     \put(0.8,-0.2){$\bullet$}
     \put(1.2,-0.2){$\bullet$}
     \end{picture}}
\newsavebox{\boxnestpaarpartwbbb}
   \savebox{\boxnestpaarpartwbbb}
   { \begin{picture}(1.7,1)
     \put(0.2,0.2){\line(0,1){0.8}}
     \put(0.6,0.2){\line(0,1){0.4}}
     \put(1,0.2){\line(0,1){0.4}}
     \put(1.4,0.2){\line(0,1){0.8}}
     \put(0.2,1){\line(1,0){1.2}}
     \put(0.6,0.6){\line(1,0){0.4}}
     \put(0,-0.2){$\circ$}
     \put(0.4,-0.2){$\bullet$}
     \put(0.8,-0.2){$\bullet$}
     \put(1.2,-0.2){$\bullet$}
     \end{picture}}     

\newsavebox{\boxdreipartwww}
   \savebox{\boxdreipartwww}
   { \begin{picture}(1.2,1)
     \put(0.2,0.2){\line(0,1){0.4}}
     \put(0.6,0.2){\line(0,1){0.4}}
     \put(1,0.2){\line(0,1){0.4}}
     \put(0.2,0.6){\line(1,0){0.8}}
     \put(0,-0.2){$\circ$}
     \put(0.4,-0.2){$\circ$}
     \put(0.8,-0.2){$\circ$}
     \end{picture}}

\newsavebox{\boxsechspartwbwbwb}
   \savebox{\boxsechspartwbwbwb}
   { \begin{picture}(2.5,1)
     \put(0.2,0.2){\line(0,1){0.4}}
     \put(0.6,0.2){\line(0,1){0.4}}
     \put(1,0.2){\line(0,1){0.4}}
     \put(1.4,0.2){\line(0,1){0.4}}
     \put(1.8,0.2){\line(0,1){0.4}}
     \put(2.2,0.2){\line(0,1){0.4}}
     \put(0.2,0.6){\line(1,0){2}}
     \put(0,-0.2){$\circ$}
     \put(0.4,-0.2){$\bullet$}
     \put(0.8,-0.2){$\circ$}
     \put(1.2,-0.2){$\bullet$}
     \put(1.6,-0.2){$\circ$}
     \put(2.0,-0.2){$\bullet$}
     \end{picture}}
 
\newsavebox{\boxcrosspartwbbw}
   \savebox{\boxcrosspartwbbw}
   { \begin{picture}(1,1.4)
     \put(0.25,0.15){\line(1,2){0.4}}
     \put(0.65,0.15){\line(-1,2){0.4}}
     \put(0,0.9){$\circ$}
     \put(0.5,0.9){$\bullet$}
     \put(0,-0.2){$\bullet$}
     \put(0.5,-0.2){$\circ$}
     \end{picture}}
\newsavebox{\boxcrosspartbwwb}
   \savebox{\boxcrosspartbwwb}
   { \begin{picture}(1,1.4)
     \put(0.25,0.15){\line(1,2){0.4}}
     \put(0.65,0.15){\line(-1,2){0.4}}
     \put(0,0.9){$\bullet$}
     \put(0.5,0.9){$\circ$}
     \put(0,-0.2){$\circ$}
     \put(0.5,-0.2){$\bullet$}
     \end{picture}}
\newsavebox{\boxcrosspartwwww}
   \savebox{\boxcrosspartwwww}
   { \begin{picture}(1,1.4)
     \put(0.25,0.15){\line(1,2){0.4}}
     \put(0.65,0.15){\line(-1,2){0.4}}
     \put(0,0.9){$\circ$}
     \put(0.5,0.9){$\circ$}
     \put(0,-0.2){$\circ$}
     \put(0.5,-0.2){$\circ$}
     \end{picture}}
\newsavebox{\boxcrosspartbbbb}
   \savebox{\boxcrosspartbbbb}
   { \begin{picture}(1,1.4)
     \put(0.25,0.15){\line(1,2){0.4}}
     \put(0.65,0.15){\line(-1,2){0.4}}
     \put(0,0.9){$\bullet$}
     \put(0.5,0.9){$\bullet$}
     \put(0,-0.2){$\bullet$}
     \put(0.5,-0.2){$\bullet$}
     \end{picture}}

\newsavebox{\boxhalflibpartwwwwww}
   \savebox{\boxhalflibpartwwwwww}
   { \begin{picture}(1.5,1.4)
     \put(0.3,0.15){\line(1,1){0.8}}
     \put(1.1,0.15){\line(-1,1){0.8}}
     \put(0.7,0.15){\line(0,1){0.8}}     
     \put(0,0.9){$\circ$}
     \put(0.5,0.9){$\circ$}
     \put(1,0.9){$\circ$}
     \put(0,-0.2){$\circ$}
     \put(0.5,-0.2){$\circ$}
     \put(1,-0.2){$\circ$}     
     \end{picture}}

\newsavebox{\boxpositionerd} 
   \savebox{\boxpositionerd}
   { \begin{picture}(2.6,1.5)
     \put(0.9,0.2){\line(0,1){0.9}}
     \put(2.3,0.2){\line(0,1){0.9}}
     \put(0.9,1.1){\line(1,0){1.4}}
     \put(0.05,-0.05){$^\uparrow$}
     \put(0.2,0.4){\tiny$^{\otimes d}$}
     \put(1.35,-0.05){$^\uparrow$}
     \put(1.5,0.4){\tiny$^{\otimes d}$}
     \put(0,-0.2){$\circ$}
     \put(0.7,-0.2){$\circ$}
     \put(1.3,-0.2){$\bullet$}
     \put(2.1,-0.2){$\bullet$}
     \end{picture}}
\newsavebox{\boxpositionerrpluseins} 
   \savebox{\boxpositionerrpluseins}
   { \begin{picture}(3.9,1.5)
     \put(1.6,0.2){\line(0,1){0.9}}
     \put(3.6,0.2){\line(0,1){0.9}}
     \put(1.6,1.1){\line(1,0){2}}
     \put(0.05,-0.05){$^\uparrow$}
     \put(0.2,0.4){\tiny$^{\otimes r+1}$}
     \put(2.05,-0.05){$^\uparrow$}
     \put(2.2,0.4){\tiny$^{\otimes r-1}$}
     \put(0,-0.2){$\circ$}
     \put(1.4,-0.2){$\bullet$}
     \put(2,-0.2){$\bullet$}
     \put(3.4,-0.2){$\bullet$}
     \end{picture}}
\newsavebox{\boxpositionerrminuseins} 
   \savebox{\boxpositionerrminuseins}
   { \begin{picture}(3.9,1.5)
     \put(0,0.2){\line(0,1){0.9}}
     \put(2,0.2){\line(0,1){0.9}}
     \put(0,1.1){\line(1,0){2}}
     \put(0.45,-0.05){$^\uparrow$}
     \put(0.6,0.4){\tiny$^{\otimes r-1}$}
     \put(-0.2,-0.2){$\circ$}
     \put(0.4,-0.2){$\circ$}
     \put(1.8,-0.2){$\circ$}
     \put(2.45,-0.05){$^\uparrow$}
     \put(2.6,0.4){\tiny$^{\otimes r+1}$}
     \put(2.4,-0.2){$\bullet$}
     \end{picture}}
\newsavebox{\boxpositionertpluseins} 
   \savebox{\boxpositionertpluseins}
   { \begin{picture}(3.9,1.5)
     \put(1.6,0.2){\line(0,1){0.9}}
     \put(3.6,0.2){\line(0,1){0.9}}
     \put(1.6,1.1){\line(1,0){2}}
     \put(0.05,-0.05){$^\uparrow$}
     \put(0.2,0.4){\tiny$^{\otimes t+1}$}
     \put(2.05,-0.05){$^\uparrow$}
     \put(2.2,0.4){\tiny$^{\otimes t-1}$}
     \put(0,-0.2){$\circ$}
     \put(1.4,-0.2){$\bullet$}
     \put(2,-0.2){$\bullet$}
     \put(3.4,-0.2){$\bullet$}
     \end{picture}}
\newsavebox{\boxpositionerspluseins} 
   \savebox{\boxpositionerspluseins}
   { \begin{picture}(3.9,1.5)
     \put(1.6,0.2){\line(0,1){0.9}}
     \put(3.6,0.2){\line(0,1){0.9}}
     \put(1.6,1.1){\line(1,0){2}}
     \put(0.05,-0.05){$^\uparrow$}
     \put(0.2,0.4){\tiny$^{\otimes s+1}$}
     \put(2.05,-0.05){$^\uparrow$}
     \put(2.2,0.4){\tiny$^{\otimes s-1}$}
     \put(0,-0.2){$\circ$}
     \put(1.4,-0.2){$\bullet$}
     \put(2,-0.2){$\bullet$}
     \put(3.4,-0.2){$\bullet$}
     \end{picture}}
\newsavebox{\boxpositionert} 
   \savebox{\boxpositionert}
   { \begin{picture}(3.9,1.5)
     \put(1.6,0.2){\line(0,1){0.9}}
     \put(3.6,0.2){\line(0,1){0.9}}
     \put(1.6,1.1){\line(1,0){2}}
     \put(0.05,-0.05){$^\uparrow$}
     \put(0.2,0.4){\tiny$^{\otimes t}$}
     \put(2.05,-0.05){$^\uparrow$}
     \put(2.2,0.4){\tiny$^{\otimes t-1}$}
     \put(0,-0.2){$\circ$}
     \put(1.4,-0.2){$\bullet$}
     \put(2,-0.2){$\bullet$}
     \put(3.4,-0.2){$\bullet$}
     \end{picture}}
\newsavebox{\boxpositionerdt} 
   \savebox{\boxpositionerdt}
   { \begin{picture}(2.9,1.5)
     \put(1.2,0.2){\line(0,1){0.9}}
     \put(2.9,0.2){\line(0,1){0.9}}
     \put(1.2,1.1){\line(1,0){1.7}}
     \put(0.05,-0.05){$^\uparrow$}
     \put(0.2,0.4){\tiny$^{\otimes dt}$}
     \put(1.65,-0.05){$^\uparrow$}
     \put(1.8,0.4){\tiny$^{\otimes dt}$}
     \put(0,-0.2){$\circ$}
     \put(1,-0.2){$\circ$}
     \put(1.6,-0.2){$\bullet$}
     \put(2.7,-0.2){$\bullet$}
     \end{picture}}
\newsavebox{\boxpositioners} 
   \savebox{\boxpositioners}
   { \begin{picture}(2.6,1.5)
     \put(0.9,0.2){\line(0,1){0.9}}
     \put(2.3,0.2){\line(0,1){0.9}}
     \put(0.9,1.1){\line(1,0){1.4}}
     \put(0.05,-0.05){$^\uparrow$}
     \put(0.2,0.4){\tiny$^{\otimes s}$}
     \put(1.35,-0.05){$^\uparrow$}
     \put(1.5,0.4){\tiny$^{\otimes s}$}
     \put(0,-0.2){$\circ$}
     \put(0.7,-0.2){$\circ$}
     \put(1.3,-0.2){$\bullet$}
     \put(2.1,-0.2){$\bullet$}
     \end{picture}}
\newsavebox{\boxpositionerdinv} 
   \savebox{\boxpositionerdinv}
   { \begin{picture}(2.6,1.5)
     \put(0.9,0.2){\line(0,1){0.9}}
     \put(2.3,0.2){\line(0,1){0.9}}
     \put(0.9,1.1){\line(1,0){1.4}}
     \put(0.05,-0.05){$^\uparrow$}
     \put(0.2,0.4){\tiny$^{\otimes d}$}
     \put(1.35,-0.05){$^\uparrow$}
     \put(1.5,0.4){\tiny$^{\otimes d}$}
     \put(0,-0.2){$\circ$}
     \put(0.7,-0.2){$\bullet$}
     \put(1.3,-0.2){$\bullet$}
     \put(2.1,-0.2){$\circ$}
     \end{picture}}
\newsavebox{\boxpositionersinv} 
   \savebox{\boxpositionersinv}
   { \begin{picture}(2.6,1.5)
     \put(0.9,0.2){\line(0,1){0.9}}
     \put(2.3,0.2){\line(0,1){0.9}}
     \put(0.9,1.1){\line(1,0){1.4}}
     \put(0.05,-0.05){$^\uparrow$}
     \put(0.2,0.4){\tiny$^{\otimes s}$}
     \put(1.35,-0.05){$^\uparrow$}
     \put(1.5,0.4){\tiny$^{\otimes s}$}
     \put(0,-0.2){$\circ$}
     \put(0.7,-0.2){$\bullet$}
     \put(1.3,-0.2){$\bullet$}
     \put(2.1,-0.2){$\circ$}
     \end{picture}}
\newsavebox{\boxpositionerdpluszwei}
   \savebox{\boxpositionerdpluszwei}
   { \begin{picture}(3.4,1.5)
     \put(1.6,0.2){\line(0,1){0.9}}
     \put(3.0,0.2){\line(0,1){0.9}}
     \put(1.6,1.1){\line(1,0){1.4}}
     \put(0.05,-0.05){$^\uparrow$}
     \put(0.2,0.4){\tiny$^{\otimes d+2}$}
     \put(2.05,-0.05){$^\uparrow$}
     \put(2.2,0.4){\tiny$^{\otimes d}$}
     \put(0,-0.2){$\circ$}
     \put(1.4,-0.2){$\bullet$}
     \put(2,-0.2){$\bullet$}
     \put(2.8,-0.2){$\bullet$}
     \end{picture}}
\newsavebox{\boxpositionersminuszwei}
   \savebox{\boxpositionersminuszwei}
   { \begin{picture}(3.4,1.5)
     \put(1,0.2){\line(0,1){0.9}}
     \put(3.0,0.2){\line(0,1){0.9}}
     \put(1,1.1){\line(1,0){2}}
     \put(0.05,-0.05){$^\uparrow$}
     \put(0.2,0.4){\tiny$^{\otimes s}$}
     \put(1.45,-0.05){$^\uparrow$}
     \put(1.6,0.4){\tiny$^{\otimes s-2}$}
     \put(0,-0.2){$\circ$}
     \put(0.8,-0.2){$\bullet$}
     \put(1.4,-0.2){$\bullet$}
     \put(2.8,-0.2){$\bullet$}
     \end{picture}}
\newsavebox{\boxpositionerrnull}
   \savebox{\boxpositionerrnull}
   { \begin{picture}(3.8,1.5)
     \put(1.2,0.2){\line(0,1){0.9}}
     \put(3.4,0.2){\line(0,1){0.9}}
     \put(1.2,1.1){\line(1,0){2.2}}
     \put(0.05,-0.05){$^\uparrow$}
     \put(0.2,0.4){\tiny$^{\otimes r_0}$}
     \put(1.65,-0.05){$^\uparrow$}
     \put(1.8,0.4){\tiny$^{\otimes r_0-2}$}
     \put(0,-0.2){$\circ$}
     \put(1,-0.2){$\bullet$}
     \put(1.6,-0.2){$\bullet$}
     \put(3.2,-0.2){$\bullet$}
     \end{picture}}
\newsavebox{\boxpositionerwbwb}
   \savebox{\boxpositionerwbwb}
   { \begin{picture}(1.8,1)
     \put(0.6,0.2){\line(0,1){0.6}}
     \put(1.4,0.2){\line(0,1){0.6}}
     \put(0.6,0.8){\line(1,0){0.8}}
     \put(0.05,-0.05){$^\uparrow$}
     \put(0.85,-0.05){$^\uparrow$}
     \put(0,-0.2){$\circ$}
     \put(0.4,-0.2){$\bullet$}
     \put(0.8,-0.2){$\circ$}
     \put(1.2,-0.2){$\bullet$}
     \end{picture}}
\newsavebox{\boxpositionerwwbb}
   \savebox{\boxpositionerwwbb}
   { \begin{picture}(1.8,1)
     \put(0.6,0.2){\line(0,1){0.6}}
     \put(1.4,0.2){\line(0,1){0.6}}
     \put(0.6,0.8){\line(1,0){0.8}}
     \put(0.05,-0.05){$^\uparrow$}
     \put(0.85,-0.05){$^\uparrow$}
     \put(0,-0.2){$\circ$}
     \put(0.4,-0.2){$\circ$}
     \put(0.8,-0.2){$\bullet$}
     \put(1.2,-0.2){$\bullet$}
     \end{picture}}
\newsavebox{\boxpositionerrevwbwb}
   \savebox{\boxpositionerrevwbwb}
   { \begin{picture}(1.8,1)
     \put(0.2,0.2){\line(0,1){0.6}}
     \put(1.0,0.2){\line(0,1){0.6}}
     \put(0.2,0.8){\line(1,0){0.8}}
     \put(0.45,-0.05){$^\uparrow$}
     \put(1.25,-0.05){$^\uparrow$}
     \put(0,-0.2){$\circ$}
     \put(0.4,-0.2){$\bullet$}
     \put(0.8,-0.2){$\circ$}
     \put(1.2,-0.2){$\bullet$}
     \end{picture}}
\newsavebox{\boxpositionersalphaw} 
   \savebox{\boxpositionersalphaw}
   { \begin{picture}(2.6,1.5)
     \put(0.9,0.2){\line(0,1){0.9}}
     \put(2.3,0.2){\line(0,1){0.9}}
     \put(0.9,1.1){\line(1,0){1.4}}
     \put(0.05,-0.05){$^\uparrow$}
     \put(1.35,-0.05){$^\uparrow$}
     \put(0.2,0.4){\tiny$^{\otimes s}$}
     \put(1.5,0.4){\tiny$^{\otimes \alpha}$}
     \put(0,-0.2){$\circ$}
     \put(0.7,-0.2){$\circ$}
     \put(1.3,-0.2){$\circ$}
     \put(2.1,-0.2){$\bullet$}
     \end{picture}}
\newsavebox{\boxpositionersalphab} 
   \savebox{\boxpositionersalphab}
   { \begin{picture}(2.6,1.5)
     \put(0.9,0.2){\line(0,1){0.9}}
     \put(2.3,0.2){\line(0,1){0.9}}
     \put(0.9,1.1){\line(1,0){1.4}}
     \put(0.05,-0.05){$^\uparrow$}
     \put(1.35,-0.05){$^\uparrow$}
     \put(0.2,0.4){\tiny$^{\otimes s}$}
     \put(1.5,0.4){\tiny$^{\otimes \alpha}$}
     \put(0,-0.2){$\circ$}
     \put(0.7,-0.2){$\circ$}
     \put(1.3,-0.2){$\bullet$}
     \put(2.1,-0.2){$\bullet$}
     \end{picture}}

\newsavebox{\boxAspace}
   \savebox{\boxAspace}
   { \begin{picture}(0.1,1.5)
     \end{picture}}
\newsavebox{\boxBspace}
   \savebox{\boxBspace}
   { \begin{picture}(0,1.85)
     \end{picture}}

\newcommand{\idpartww}{\usebox{\boxidpartww}}
\newcommand{\idpartwb}{\usebox{\boxidpartwb}}
\newcommand{\idpartbw}{\usebox{\boxidpartbw}}
\newcommand{\idpartbb}{\usebox{\boxidpartbb}}
\newcommand{\idpartsingletonww}{\usebox{\boxidpartsingletonww}}
\newcommand{\idpartsingletonbb}{\usebox{\boxidpartsingletonbb}}
\newcommand{\idpartsingletonbw}{\usebox{\boxidpartsingletonbw}}
\newcommand{\idpartsingletonwb}{\usebox{\boxidpartsingletonwb}}
\newcommand{\paarpartww}{\usebox{\boxpaarpartww}}
\newcommand{\paarpartbw}{\usebox{\boxpaarpartbw}}
\newcommand{\paarpartwb}{\usebox{\boxpaarpartwb}}
\newcommand{\paarpartbb}{\usebox{\boxpaarpartbb}}
\newcommand{\baarpartww}{\usebox{\boxbaarpartww}}
\newcommand{\baarpartbw}{\usebox{\boxbaarpartbw}}
\newcommand{\baarpartwb}{\usebox{\boxbaarpartwb}}
\newcommand{\baarpartbb}{\usebox{\boxbaarpartbb}}
\newcommand{\cutpaarpartww}{\usebox{\boxcutpaarpartww}}
\newcommand{\cutpaarpartbw}{\usebox{\boxcutpaarpartbw}}
\newcommand{\cutpaarpartwb}{\usebox{\boxcutpaarpartwb}}
\newcommand{\cutpaarpartbb}{\usebox{\boxcutpaarpartbb}}
\newcommand{\nestpaarpartbwbb}{\usebox{\boxnestpaarpartbwbb}}
\newcommand{\nestpaarpartwbbb}{\usebox{\boxnestpaarpartwbbb}}
\newcommand{\singletonw}{\usebox{\boxsingletonw}}
\newcommand{\singletonb}{\usebox{\boxsingletonb}}
\newcommand{\downsingletonw}{\usebox{\boxdownsingletonw}}
\newcommand{\downsingletonb}{\usebox{\boxdownsingletonb}}
\newcommand{\vierpartwbwb}{\usebox{\boxvierpartwbwb}}

\newcommand{\vierpartwwbb}{\usebox{\boxvierpartwwbb}}

\newcommand{\vierpartrotwbwb}{\usebox{\boxvierpartrotwbwb}}
\newcommand{\vierpartrotbwwb}{\usebox{\boxvierpartrotbwwb}}
\newcommand{\vierpartrotwbbw}{\usebox{\boxvierpartrotwbbw}}
\newcommand{\vierpartrotbwbw}{\usebox{\boxvierpartrotbwbw}}
\newcommand{\vierpartrotwwww}{\usebox{\boxvierpartrotwwww}}
\newcommand{\vierpartrotbbbb}{\usebox{\boxvierpartrotbbbb}}
\newcommand{\vierpartwrotwwb}{\usebox{\boxvierpartwrotwwb}}
\newcommand{\vierpartwrotbww}{\usebox{\boxvierpartwrotbww}}
\newcommand{\vierpartbrotwbb}{\usebox{\boxvierpartbrotwbb}}
\newcommand{\vierpartbrotbbw}{\usebox{\boxvierpartbrotbbw}}
\newcommand{\dreipartwww}{\usebox{\boxdreipartwww}}
\newcommand{\sechspartwbwbwb}{\usebox{\boxsechspartwbwbwb}}
\newcommand{\crosspartwbbw}{\usebox{\boxcrosspartwbbw}}
\newcommand{\crosspartbwwb}{\usebox{\boxcrosspartbwwb}}
\newcommand{\crosspartwwww}{\usebox{\boxcrosspartwwww}}
\newcommand{\crosspartbbbb}{\usebox{\boxcrosspartbbbb}}
\newcommand{\halflibpartwwwwww}{\usebox{\boxhalflibpartwwwwww}}
\newcommand{\positionerd}{\usebox{\boxpositionerd}}
\newcommand{\positionerrpluseins}{\usebox{\boxpositionerrpluseins}}
\newcommand{\positionerspluseins}{\usebox{\boxpositionerspluseins}}
\newcommand{\positionertpluseins}{\usebox{\boxpositionertpluseins}}

\newcommand{\positioners}{\usebox{\boxpositioners}}
\newcommand{\positionert}{\usebox{\boxpositionert}}

\newcommand{\positionersinv}{\usebox{\boxpositionersinv}}

\newcommand{\positionersminuszwei}{\usebox{\boxpositionersminuszwei}}
\newcommand{\positionerrminuseins}{\usebox{\boxpositionerrminuseins}}

\newcommand{\positionerwbwb}{\usebox{\boxpositionerwbwb}}
\newcommand{\positionerwwbb}{\usebox{\boxpositionerwwbb}}

\newcommand{\positioner}{\usebox{\boxpositioner}}
\newcommand{\Aspace}{\usebox{\boxAspace}}
\newcommand{\Bspace}{\usebox{\boxBspace}}

\newcommand{\subsetdown}{\begin{rotate}{270}$\subset$\end{rotate}}

\newcommand{\categ}[3]{{#1}_{#2}(#3)}
\newcommand{\categg}[2]{{#1}_{#2}}
\newcommand{\glob}{\textnormal{glob}}
\newcommand{\loc}{\textnormal{loc}}
\newcommand{\grp}{\textnormal{grp}}

\begin{document}
\title[Categories of two-colored non-crossing partitions]{The classification of tensor categories of two-colored noncrossing partitions}
\author{Pierre Tarrago and Moritz Weber}
\address{Saarland University, Fachbereich Mathematik, Postfach 151150,
66041 Saarbr\"ucken, Germany}
\email{tarrago@math.uni-sb.de}
\email{weber@math.uni-sb.de}
\date{\today}

\begin{abstract}
Our basic objects are partitions of finite sets of points into disjoint subsets. We investigate  sets of partitions which are closed under taking tensor products, composition and involution, and which contain certain base partitions. These so called categories of partitions are exactly the tensor categories  being used in the theory of Banica and Speicher's orthogonal easy quantum groups. In our approach, we additionally allow a coloring of the points. This serves as the basis for the introduction of unitary easy quantum groups, which is done in a separate article.
The present article however is purely combinatorial. We find all categories of two-colored noncrossing partitions. For doing so, we extract certain parameters with values in the natural numbers specifying the colorization of the categories on a global as well as on a local level. It turns out that there are ten series of categories, each indexed by one or two parameters from the natural numbers, plus two additional categories. This is just the beginning of the classification of categories of two-colored partitions and we point out open problems at the end of the article.
\end{abstract}

\maketitle
\section*{Introduction}

Set theoretical partitions are basic combinatorial objects which appear in many branches of mathematics, see for instance Stanley's book \cite{Stanley} or free probability theory \cite{nicaspeicher}. A partition is a decomposition of a finite set of $k+l$ points into disjoint subsets (the blocks), for $k,l\in\N_0$. In our case, additionally each point is colored either white or black. We represent such a partition pictorially by connecting $k$ upper points with $l$ lower points using strings according to the block pattern. If these strings may be drawn in such a way that they do not cross, we call the partition noncrossing. Here is an example of a noncrossing partition as well as  of a not noncrossing one.
\setlength{\unitlength}{0.5cm}
\begin{center}
\begin{picture}(12,4)
\put(-1,4.35){\partii{1}{1}{2}}
\put(-1,4.35){\partii{1}{3}{4}}
\put(-1,0.35){\uppartii{1}{1}{2}}
\put(-1,0.35){\upparti{2}{3}}
\put(0.05,0){$\circ$}
\put(1.05,0){$\circ$}
\put(2.05,0){$\bullet$}
\put(0.05,3.3){$\bullet$}
\put(1.05,3.3){$\circ$}
\put(2.05,3.3){$\bullet$}
\put(3.05,3.3){$\circ$}
\put(8.3,3.35){\line(1,-3){1}}
\put(7,4.35){\partii{1}{3}{4}}
\put(8.3,0.35){\line(1,3){1}}
\put(7,0.35){\upparti{1}{3}}
\put(8.05,0){$\bullet$}
\put(9.05,0){$\circ$}
\put(10.05,0){$\circ$}
\put(8.05,3.3){$\circ$}
\put(9.05,3.3){$\bullet$}
\put(10.05,3.3){$\bullet$}
\put(11.05,3.3){$\circ$}
\end{picture}
\end{center}

 Given two such partitions, we may form the tensor product (placing them side by side), the composition (placing one above the other), and the involution (reflecting a partition at the horizontal axis). If a set of partitions is closed under these operations and if it contains certain base partitions, it is called a category of partitions. (See Section \ref{SectPartitions}.)

Before speaking about the main results of this article, let us briefly mention our main application of categories of partitions.
In 1987, Woronowicz introduced compact (matrix) quantum groups \cite{woronowicz1987compact}. These are operator algebraic objects generalizing the notion of compact groups and they are most suitable to describe symmetries arising in the noncommutative framework of operator algebras.
By a Tannaka-Krein type result of Woronowicz \cite{woronowicz1988tannaka} compact matrix quantum groups are completely determined by their intertwiner spaces.
In 2009, Banica and Speicher observed that one can define the above mentioned operations on partitions and that they translate one-to-one to natural operations on the intertwiner spaces via some functor. Thus, any category of partitions gives rise to a compact matrix quantum group by modelling its intertwiner space. This lead them to  the definition of orthogonal easy quantum groups \cite{banica2009liberation}, a quite combinatorial class of quantum groups. One of the nice features of these easy quantum groups is, that many operator algebraic or quantum algebraic properties may be seen already in the underlying combinatorics of partitions, see for instance \cite{freslonweber}, \cite{raum2013easy}. 
 Using our approach involving colors, we may define easy quantum groups also in the unitary case. This is done in a separate article \cite{tarragoweberopalg}.

In the present article however, we focus on the purely combinatorial aspects of this theory. Our guiding question is: Which examples of categories of (two-colored) partitions are there? Our main result is the full classification of all categories of two-colored \emph{noncrossing} partitions -- jump to Section \ref{SectMainResult} to read a short summary. Speaking about the technical details of how to tackle the classification, we first divide the categories into four cases, depending on the size of the blocks; more precisely: whether the partitions $\singletonw\otimes\singletonb$ and/or $\vierpartwbwb$ are in the category or not (see Section \ref{SectPartitions} for a definition of these partitions). We then subdivide them again, since categories containing the partition $\paarpartww\otimes\paarpartbb$ behave very differently from those not containing it. In the first case, the globally colorized one,  ``the colors matter only on a global level''. We thus  only have to determine all possible values $c(p)$ of differences between the black points and the white points of any partition $p$ (when being rotated to one line) of a fixed category.
In the second case, the locally colorized one, we also need to take into account certain local parameters: We have to study all possible values $c(p_1)$ of subpartitions $p_1$ sitting between two legs of a partition. We obtain two local parameters depending on whether the points of these two legs are colored by the same color or not. The general strategy for the classification is described in Section \ref{SectStrategy}, and it is applied in each of the four cases (Sections \ref{SectCaseO} to \ref{SectCaseB}). 

The classification is then summarized in Section \ref{SectMainResult}. 
While there are only seven categories of one-colored noncrossing partitions \cite{weber2013classification}, we obtain ten series of categories in the two-colored case, each indexed by one or two parameters from the natural numbers, plus two additional categories. In this sense, the world of unitary easy quantum groups is way richer than the one of orthogonal easy quantum groups. Note that all categories of one-colored partitions (including possibly crossing ones) have recently been found in \cite{raum2013full}. In the two-colored case however, little is known and the present article constitutes only the beginning of a longer investigation. At the end of this article, we also cover the case of categories containing the crossing partition $\crosspartwwww$, see Section \ref{SectGroupCase}. This is the analog of the group case in Woronowicz's theory. Finally, we comment on further aspects and applications of our work in Section \ref{SectConclRem}.

\section*{Acknowledgements}

We thank Teo Banica, Stephen Curran and Roland Speicher for sending us an unpublished draft \cite{speicherunpublished} of their work on the definition and classification of unitary easy quantum groups. Some parts of this article may be found in their draft, too.

The first author was supported by the Universit\'e Franco-Allemande. Both authors were partially funded by the ERC Advanced Grant on Non-Commutative Distributions in Free Probability, held by Roland Speicher.

\section{Categories of two-colored partitions}\label{SectPartitions}

The basics on non-colored partitions presented in this section are well-known to experts in (orthogonal) easy quantum groups. The main ideas may be found in the initial paper \cite{banica2009liberation} on easy quantum groups. However, we need to formulate it for partitions involving a coloring of the points. Attempts in this direction may be found in \cite{freslonweber}, \cite{freslon2014fusion}, \cite{freslon2014partition}, \cite{lemeuxfusion} or earlier in \cite{banicacollinsintcqg} and \cite{banica2008note}.

\subsection{Partitions}

A \emph{(two-colored) partition} is a set theoretical decomposition of $k+l$ points ($k$ of which are ``upper'' and $l$ of which are ``lower'') into disjoint subsets, the \emph{blocks}.
In addition, the points may either be white ($\circ$) or black ($\bullet$). We say that these colors are \emph{inverse} to each other.
The set of all such partitions is denoted by $P^{\twocol}(k,l)$, for $k,l\in\N_0$, and the collection of all $P^{\twocol}(k,l)$ is denoted by $P^{\twocol}$. 
By $\N_0$ we denote the natural numbers $\N_0=\{0,1,2,\ldots\}$.
We usually draw these partitions by placing the $k$ upper points on an upper line and the $l$ lower points on a lower line. We then connect the points by strings according to their decomposition into blocks.
Here are two examples of partitions in $P^{\twocol}(4,3)$. In the second example, the lower right point forms a block of its own and the other blocks are all of size two.

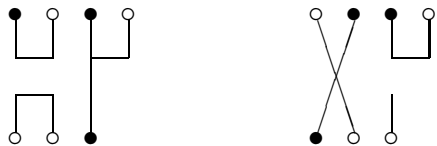
\setlength{\unitlength}{0.5cm}
\begin{center}
\begin{picture}(12,4)
\put(-1,4.35){\partii{1}{1}{2}}
\put(-1,4.35){\partii{1}{3}{4}}
\put(-1,0.35){\uppartii{1}{1}{2}}
\put(-1,0.35){\upparti{2}{3}}
\put(0.05,0){$\circ$}
\put(1.05,0){$\circ$}
\put(2.05,0){$\bullet$}
\put(0.05,3.3){$\bullet$}
\put(1.05,3.3){$\circ$}
\put(2.05,3.3){$\bullet$}
\put(3.05,3.3){$\circ$}
\put(8.3,3.35){\line(1,-3){1}}
\put(7,4.35){\partii{1}{3}{4}}
\put(8.3,0.35){\line(1,3){1}}
\put(7,0.35){\upparti{1}{3}}
\put(8.05,0){$\bullet$}
\put(9.05,0){$\circ$}
\put(10.05,0){$\circ$}
\put(8.05,3.3){$\circ$}
\put(9.05,3.3){$\bullet$}
\put(10.05,3.3){$\bullet$}
\put(11.05,3.3){$\circ$}
\end{picture}
\end{center}

If the connecting strings of a partition $p\in P^{\twocol}(k,l)$ can be drawn in such a way that they do not cross, the partition is called \emph{noncrossing}, and we denote by $NC^{\twocol}(k,l)$ the set of all noncrossing partitions, and $NC^{\twocol}$ for the collection of all $NC^{\twocol}(k,l)$. In the above examples, the first partition is in $NC^{\twocol}$ whereas the second is not.

In the sequel, the following examples of partitions will play a special role.
\begin{itemize}
\item Each of the \emph{identity partitions} $\idpartww,\idpartbb\in P^{\twocol}(1,1)$ connects one upper point with one lower point of the same color. Note that $\idpartwb$ and $\idpartbw$ are \emph{not} identity partitions.
\item The \emph{bicolored pair partitions} $\paarpartbw, \paarpartwb\in P^{\twocol}(0,2)$ connect two lower points of different colors. We also have their horizontally reflected versions $\baarpartbw,\baarpartwb\in P^{\twocol}(2,0)$.
The \emph{unicolored pair partitions} are $\paarpartbb, \paarpartww\in P^{\twocol}(0,2)$\linebreak and $\baarpartbb,\baarpartww\in P^{\twocol}(2,0)$.
\item The \emph{singleton partitions} $\singletonw,\singletonb\in P^{\twocol}(0,1)$ consist of a single lower point respectively. Their reflected versions are denoted by $\downsingletonw, \downsingletonb\in P^{\twocol}(1,0)$.
\item We also have \emph{four block partitions} like $\vierpartwbwb,\vierpartwwbb\in P^{\twocol}(0,4)$ and $\vierpartrotwbwb, \vierpartrotbwwb\in P^{\twocol}(2,2)$. 
\item All preceding examples are partitions consisting of a single block. The \emph{crossing partition} $\crosspartwbbw\in P^{\twocol}(2,2)$ however consists of two blocks. It connects a white upper left point to a white lower right point, as well as a black upper right point to a black lower left point; we also have other colorings like $\crosspartbwwb$ or $\crosspartwwww$. These partitions  are not in $NC^{\twocol}(2,2)$.
\end{itemize}

It is often convenient to associate words to partitions $p\in P^{\twocol}(0,l)$ having no upper points in the sense that each block $V$ is represented by a unique letter $a,b,c,\ldots$. Furthermore, we use the notation $a$ and $a^{-1}$ for points having inverse colors but belonging to the same block. As an example, the partition $\vierpartwwbb$ corresponds to the word $aaa^{-1}a^{-1}$ whereas $\positionerwwbb$ is $abcb^{-1}$. This representation is not unique since we do not specify whether $a$ is white and $a^{-1}$ is black or vice versa. Also, we sometimes use capital letters $X,Y,Z,\ldots$ in order to denote subwords of a partition seen as a word.

\subsection{Operations on partitions}\label{SectOperations}

Let us now turn to operations on the set $P^{\twocol}$.
\begin{itemize}
 \item  The \emph{tensor product} of two partitions $p\in P^{\twocol}(k,l)$ and $q\in P^{\twocol}(k',l')$ is the partition $p\otimes q\in P^{\twocol}(k+k',l+l')$ obtained by horizontal concatenation (writing $p$ and $q$ side by side), i.e. the first $k$ of the $k+k'$ upper points are connected by $p$ to the first $l$ of the $l+l'$ lower points, whereas $q$ connects the remaining $k'$ upper points with the remaining $l'$ lower points.
 \item The \emph{composition} of two partitions $q\in P^{\twocol}(k,l)$ and $p\in P^{\twocol}(l,m)$ is the partition  $pq\in P^{\twocol}(k,m)$ obtained by vertical concatenation (writing $p$ below $q$): First connect $k$ upper points by $q$ to $l$ middle points and then connect these points  by $p$ to $m$ lower points.  This yields a partition, connecting $k$ upper points with $m$ lower points. The $l$ middle points  are removed. 
 
 Note that we can compose two partitions $q\in P^{\twocol}(k,l)$ and $p\in P^{\twocol}(l',m)$ only if
 \begin{itemize}
 \item[(i)] the numbers $l$ and $l'$ coincide,
 \item[(ii)] the colorings match, i.e. the color of the $j$-th lower point of $q$ coincides with the color of the $j$-th upper point of $p$, for all $1\leq j\leq l$.
 \end{itemize} 
 \item The \emph{vertical reflection} of a partition $p\in P^{\twocol}(k,l)$ is given by the reflection  $R_v(p)\in P^{\twocol}(k,l)$ at the vertical axis.
 \item The \emph{horizontal reflection} of a partition $p\in P^{\twocol}(k,l)$ is given by the reflection  $R_h(p)\in P^{\twocol}(l,k)$ at the horizontal axis. We also call it the \emph{involution} of the partition $p$ and denote it by $p^*:=R_h(p)$.
 \item The \emph{inversion of colors} of a partition $p\in P^{\twocol}(k,l)$ is given by the partition  $R_c(p)\in P^{\twocol}(k,l)$ where all colors of the points are inverted.
 \item The \emph{verticolor reflection} of a partition $p$ is given by $\tilde p:=R_vR_c(p)$.
 \item We also have a \emph{rotation} on partitions. Let $p\in P^{\twocol}(k,l)$ be a partition connecting $k$ upper points with $l$ lower points. Shifting the very left upper point to the left of the lower points  and inverting its color gives rise to a partition in $P^{\twocol}(k-1, l+1)$, a \emph{rotated version} of $p$. Note that the point still belongs to the same block after rotation. 
We may also rotate the leftmost lower point to the very left of the upper line (again inverting its color), and we may as well rotate in the right hand side of the lines.
 In particular, for a partition $p\in P^{\twocol}(0,l)$, we may rotate the very left point to the very right and vice versa. Such a rotation on one line does \emph{not} change the colors of the points. 
 \end{itemize}
 
 Here are some examples of these operations.
 
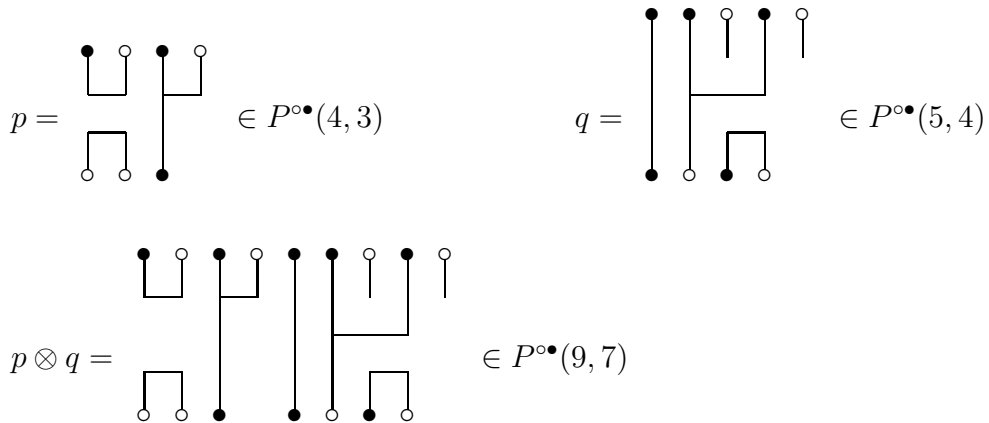
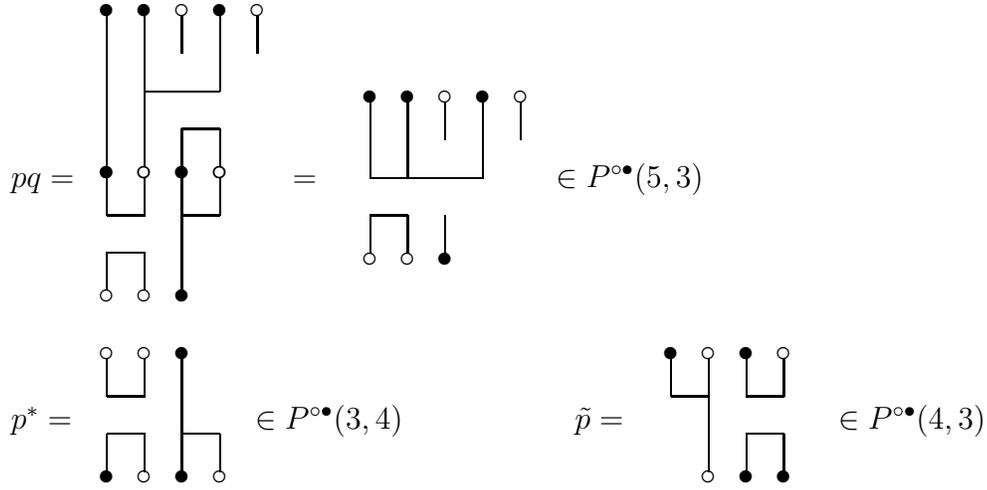
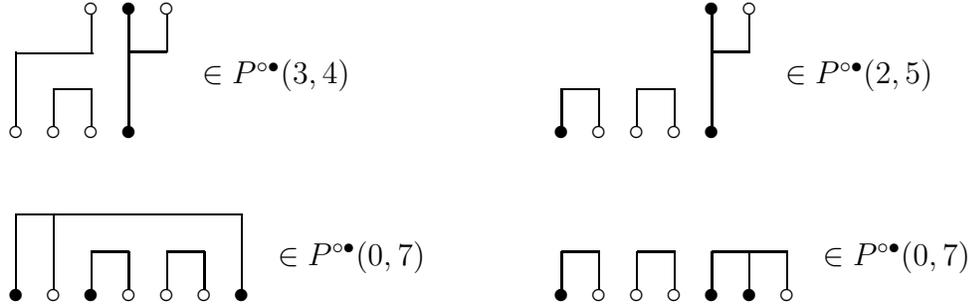
\setlength{\unitlength}{0.5cm}
\newsavebox{\boxpexample}
   \savebox{\boxpexample}
   { \begin{picture}(3,3.5)
     \put(-1,4.35){\partii{1}{1}{2}}
     \put(-1,4.35){\partii{1}{3}{4}}
     \put(-1,0.35){\uppartii{1}{1}{2}}
     \put(-1,0.35){\upparti{2}{3}}
     \put(0.05,0){$\circ$}
     \put(1.05,0){$\circ$}
     \put(2.05,0){$\bullet$}
     \put(0.05,3.3){$\bullet$}
     \put(1.05,3.3){$\circ$}
     \put(2.05,3.3){$\bullet$}
     \put(3.05,3.3){$\circ$}
     \end{picture}}
\newsavebox{\boxqexample}
   \savebox{\boxqexample}
   { \begin{picture}(4,4.5)
     \put(-1,0.35){\upparti{4}{1}}
     \put(-1,0.35){\upparti{2}{2}}
     \put(-1,5.35){\partii{2}{2}{4}}
     \put(-1,5.35){\parti{1}{3}}
     \put(-1,5.35){\parti{1}{5}}
     \put(-1,0.35){\uppartii{1}{3}{4}}
     \put(0.05,0){$\bullet$}
     \put(1.05,0){$\circ$}
     \put(2.05,0){$\bullet$}
     \put(3.05,0){$\circ$}
     \put(0.05,4.3){$\bullet$}
     \put(1.05,4.3){$\bullet$}
     \put(2.05,4.3){$\circ$}
     \put(3.05,4.3){$\bullet$}
     \put(4.05,4.3){$\circ$}
     \end{picture}}
\newsavebox{\boxpexamplelang}
   \savebox{\boxpexamplelang}
   { \begin{picture}(3,4.5)
     \put(-1,5.35){\partii{1}{1}{2}}
     \put(-1,5.35){\partii{1}{3}{4}}
     \put(-1,0.35){\uppartii{1}{1}{2}}
     \put(-1,0.35){\upparti{3}{3}}
     \put(0.05,0){$\circ$}
     \put(1.05,0){$\circ$}
     \put(2.05,0){$\bullet$}
     \put(0.05,4.3){$\bullet$}
     \put(1.05,4.3){$\circ$}
     \put(2.05,4.3){$\bullet$}
     \put(3.05,4.3){$\circ$}
     \end{picture}}
\newsavebox{\boxpq}
   \savebox{\boxpq}
   { \begin{picture}(4,4.5)
     \put(-1,5.35){\partiii{2}{1}{2}{4}}
     \put(-1,5.35){\parti{1}{3}}
     \put(-1,5.35){\parti{1}{5}}
     \put(-1,0.35){\uppartii{1}{1}{2}}
     \put(-1,0.35){\upparti{1}{3}}
     \put(0.05,0){$\circ$}
     \put(1.05,0){$\circ$}
     \put(2.05,0){$\bullet$}
     \put(0.05,4.3){$\bullet$}
     \put(1.05,4.3){$\bullet$}
     \put(2.05,4.3){$\circ$}
     \put(3.05,4.3){$\bullet$}
     \put(4.05,4.3){$\circ$}
     \end{picture}}
\newsavebox{\boxpstern}
   \savebox{\boxpstern}
   { \begin{picture}(3,3.5)
     \put(-1,4.35){\partii{1}{1}{2}}
     \put(-1,0.35){\uppartii{1}{3}{4}}
     \put(-1,0.35){\uppartii{1}{1}{2}}
     \put(-1,0.35){\upparti{3}{3}}
     \put(0.05,3.3){$\circ$}
     \put(1.05,3.3){$\circ$}
     \put(2.05,3.3){$\bullet$}
     \put(0.05,0){$\bullet$}
     \put(1.05,0){$\circ$}
     \put(2.05,0){$\bullet$}
     \put(3.05,0){$\circ$}
     \end{picture}}
\newsavebox{\boxptilde}
   \savebox{\boxptilde}
   { \begin{picture}(3,3.5)
     \put(-1,4.35){\partii{1}{1}{2}}
     \put(-1,4.35){\partii{1}{3}{4}}
     \put(-1,0.35){\uppartii{1}{3}{4}}
     \put(-1,0.35){\upparti{2}{2}}
     \put(1.05,0){$\circ$}
     \put(2.05,0){$\bullet$}
     \put(3.05,0){$\bullet$}
     \put(0.05,3.3){$\bullet$}
     \put(1.05,3.3){$\circ$}
     \put(2.05,3.3){$\bullet$}
     \put(3.05,3.3){$\circ$}
     \end{picture}}
\newsavebox{\boxproteins}
   \savebox{\boxproteins}
   { \begin{picture}(4,3.5)
     \put(-1,4.35){\parti{1}{3}}
     \put(-1,4.35){\partii{1}{4}{5}}
     \put(-1,0.35){\uppartii{1}{2}{3}}
     \put(-1,0.35){\upparti{2}{4}}
     \put(-1,0.35){\upparti{2}{1}}
     \put(0.3,2.3){\line(1,0){2}}
     \put(0.05,0){$\circ$}
     \put(1.05,0){$\circ$}
     \put(2.05,0){$\circ$}
     \put(3.05,0){$\bullet$}
     \put(2.05,3.3){$\circ$}
     \put(3.05,3.3){$\bullet$}
     \put(4.05,3.3){$\circ$}
     \end{picture}}
\newsavebox{\boxprotzwei}
   \savebox{\boxprotzwei}
   { \begin{picture}(5,3.5)
     \put(-1,4.35){\partii{1}{5}{6}}
     \put(-1,0.35){\uppartii{1}{3}{4}}
     \put(-1,0.35){\upparti{2}{5}}
     \put(-1,0.35){\uppartii{1}{1}{2}}
     \put(0.05,0){$\bullet$}
     \put(1.05,0){$\circ$}
     \put(2.05,0){$\circ$}
     \put(3.05,0){$\circ$}
     \put(4.05,0){$\bullet$}
     \put(4.05,3.3){$\bullet$}
     \put(5.05,3.3){$\circ$}
     \end{picture}}
\newsavebox{\boxprotdrei}
   \savebox{\boxprotdrei}
   { \begin{picture}(6,2.5)
     \put(-1,0.35){\uppartiii{2}{1}{2}{7}}
     \put(-1,0.35){\uppartii{1}{3}{4}}
     \put(-1,0.35){\uppartii{1}{5}{6}}
     \put(0.05,0){$\bullet$}
     \put(1.05,0){$\circ$}
     \put(2.05,0){$\bullet$}
     \put(3.05,0){$\circ$}
     \put(4.05,0){$\circ$}     
     \put(5.05,0){$\circ$}     
     \put(6.05,0){$\bullet$}
     \end{picture}}
\newsavebox{\boxprotvier}
   \savebox{\boxprotvier}
   { \begin{picture}(6,1.5)
     \put(-1,0.35){\uppartii{1}{1}{2}}
     \put(-1,0.35){\uppartii{1}{3}{4}}
     \put(-1,0.35){\uppartiii{1}{5}{6}{7}}
     \put(0.05,0){$\bullet$}
     \put(1.05,0){$\circ$}
     \put(2.05,0){$\circ$}     
     \put(3.05,0){$\circ$}     
     \put(4.05,0){$\bullet$}
     \put(5.05,0){$\bullet$}
     \put(6.05,0){$\circ$}
     \end{picture}}
\begin{center}
\begin{picture}(25,5.5)
 \put(0,1.5){$p=$}
 \put(1.5,0){\usebox{\boxpexample}}
 \put(6,1.5){$\in P^{\twocol}(4,3)$}
 \put(15,1.5){$q=$}
 \put(16.5,0){\usebox{\boxqexample}}
 \put(22,1.5){$\in P^{\twocol}(5,4)$}
\end{picture}
\end{center}
\begin{center}
\begin{picture}(25,5.5)
 \put(0,1.5){$p\otimes q=$}
 \put(3,0){\usebox{\boxpexamplelang}}
 \put(7,0){\usebox{\boxqexample}}
 \put(12.5,1.5){$\in P^{\twocol}(9,7)$}
\end{picture}
\end{center}
\begin{center}
\begin{picture}(25,8.7)
 \put(0,3.1){$pq=$}
 \put(2,0){\usebox{\boxpexample}}
 \put(2,3.3){\usebox{\boxqexample}}
 \put(7.5,3.1){$=$}
 \put(9,1){\usebox{\boxpq}}
 \put(14.5,3.1){$\in P^{\twocol}(5,3)$}
\end{picture}
\end{center}
\begin{center}
\begin{picture}(25,4.5)
 \put(0,1.5){$p^*=$}
 \put(2,0){\usebox{\boxpstern}}
 \put(6.5,1.5){$\in P^{\twocol}(3,4)$}
 \put(15,1.5){$\tilde p=$}
 \put(17,0){\usebox{\boxptilde}}
 \put(22,1.5){$\in P^{\twocol}(4,3)$}
\end{picture}
\end{center}
\begin{center}
\begin{picture}(25,6.5)
 \put(0,5){rotated versions of $p$:}
 \put(1,0){\usebox{\boxproteins}}
 \put(6.5,1.5){$\in P^{\twocol}(3,4)$}
 \put(15.5,0){\usebox{\boxprotzwei}}
 \put(22,1.5){$\in P^{\twocol}(2,5)$}
\end{picture}
\end{center}
\begin{center}
\begin{picture}(25,4)
 \put(1,0){\usebox{\boxprotdrei}}
 \put(8.5,1){$\in P^{\twocol}(0,7)$}
 \put(15.5,0){\usebox{\boxprotvier}}
 \put(23,1){$\in P^{\twocol}(0,7)$}
\end{picture}
\end{center}

\subsection{Categories of partitions} \label{SectCateg}

A collection $\CC$ of subsets $\CC(k,l)\subseteq P^{\twocol}(k,l)$ (for all $k,l\in\N_0$) is a \emph{category of partitions}, if it is closed under the tensor product, the composition and the involution, and if it contains the bicolored pair partitions $\paarpartwb$ and  $\paarpartbw$ as well as the identity partitions $\idpartww$ and $\idpartbb$.
Examples of categories of partitions are the set of all partitions $P^{\twocol}$, the set of all pair partitions $P^{\twocol}_2$ (i.e. all blocks have length two), the set of all non-crossing partitions $NC^{\twocol}$, and the set of all non-crossing pair partitions $NC^{\twocol}_2$, as may be verified directly.

If $\CC$ is the smallest category of partitions containing the partitions $p_1,\ldots,p_n$, we write $\CC=\langle p_1,\ldots,p_n\rangle$ and say that $\CC$ is \emph{generated} by $p_1,\ldots,p_n$. Recall that $\paarpartwb, \paarpartbw, \idpartww$ and $\idpartbb$ are always in a category, so we omit to write down these generators.

\begin{lem}\label{PropCategOperations}\label{RemErasePoints}
Let $\CC\subset P^{\twocol}$ be a category of partitions. 
\begin{itemize}
\item[(a)] $\CC$ is closed under rotation and verticolor reflection.
\item[(b)] If $p\in P^{\twocol}(k,l)$ is a partition in $\CC$, we can \emph{erase} two neighbouring points of $p$ if they have different (!) colors, i.e. if the $j$-th and the $(j+1)$-th of the lower points have inverse colors, then the partition $p'\in P^{\twocol}(k,l-2)$ is in $\CC$ which is obtained from $p$ by first connecting the blocks to which the $j$-th and the $(j+1)$-th lower points belong respectively, and then erasing these two points. We may also erase neighbouring points of inverse colors on the upper line.
\item[(c)] Let $p_1\otimes p_2\in\CC$. Then $p_1\otimes\tilde p_1\in\CC$ and $p_2\otimes\tilde p_2\in\CC$. Note that we do not have $p_1\in\CC$ or $p_2\in\CC$ in general.
\item[(d)] Let $p\in\CC(0,l)$ and $q\in\CC(0,m)$. Every partition obtained from placing $q$ between two legs of $p$ is in $\CC$.
\end{itemize}
\end{lem}
\begin{proof}
(a) Firstly, let $p\in P^{\twocol}(k,l)$ be a partition and let the first of the $k$ upper points be black.  Let $r\in P^{\twocol}(k-1,k-1)$ be the tensor product of the identity partitions $\idpartww$ and $\idpartbb$ with the same color pattern as the latter $k-1$ upper points of $p$. The composition $(\idpartww\otimes p)(\paarpartwb\otimes r)$ yields a partition $p'\in P^{\twocol}(k-1,l+1)$ which conincides with the partition obtained from $p$ when rotating the first upper points to the row of lower points. See also \cite[Lem 2.7]{banica2009liberation}. If now $p\in\CC$, then also $p'\in\CC$ since all partitions we used are in the category. Similarly we prove the other cases of rotation. As an example:

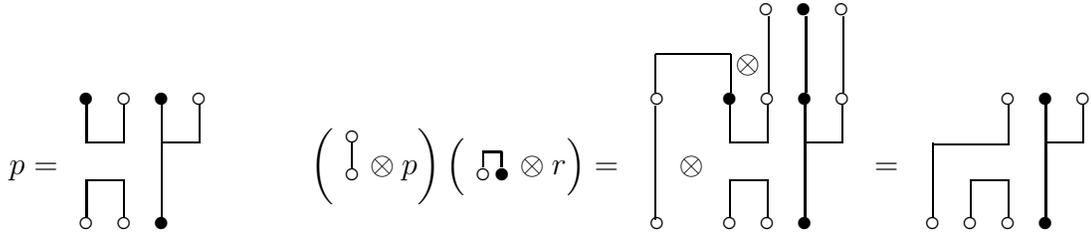
\setlength{\unitlength}{0.5cm}
\begin{center}
\begin{picture}(28,6)
 \put(0,1.5){$p=$}
 \put(1.5,0){\usebox{\boxpexample}}
 \put(8,1.5){$\left(\idpartww\otimes p\right)\left(\paarpartwb\otimes r\right)=$}
 \put(15.9,0.3){\upparti{3}{1}}
 \put(17,0){$\circ$}
 \put(17,3.3){$\circ$}
 \put(17.8,1.5){$\otimes$}
 \put(18.6,0){\usebox{\boxpexample}}
 \put(15.9,3.7){\uppartii{1}{1}{3}}
 \put(19.3,4.2){$\otimes$}
 \put(18.9,3.7){\upparti{2}{1}}
 \put(18.9,3.7){\upparti{2}{2}}
 \put(18.9,3.7){\upparti{2}{3}}
 \put(19.9,5.7){$\circ$}
 \put(20.9,5.7){$\bullet$}
 \put(21.9,5.7){$\circ$}
 \put(23,1.5){$=$}
 \put(24,0){\usebox{\boxproteins}}
\end{picture}
\end{center}

Secondly, if $p\in P^{\twocol}(k,l)$ is in $\CC$, then also $p^*\in P^{\twocol}(l,k)$ is in $\CC$ by the definition of a category. Rotating the $k$ upper points to below and the lower points to the upper line yields $\tilde p$, which is in $\CC$.

(b) Compose $p$ with  $r_1\otimes\baarpartwb\otimes r_2$ or $r_1\otimes\baarpartbw\otimes r_2$,  where $r_1$ and $r_2$ are  suitable tensor products of the identity partitions $\idpartww$ and $\idpartbb$. As an example:

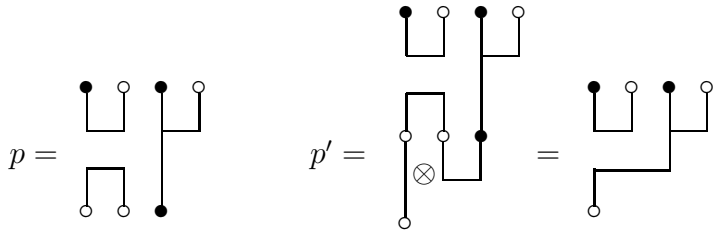
\setlength{\unitlength}{0.5cm}
\newsavebox{\boxpcap}
   \savebox{\boxpcap}
   { \begin{picture}(3,3.5)
     \put(-1,4.35){\partii{1}{1}{2}}
     \put(-1,4.35){\partii{1}{3}{4}}
     \put(-1,4.35){\parti{2}{3}}
     \put(-1,0.35){\upparti{1}{1}}
     \put(0.3,1.3){\line(1,0){2}}
     \put(0.05,0){$\circ$}
     \put(0.05,3.3){$\bullet$}
     \put(1.05,3.3){$\circ$}
     \put(2.05,3.3){$\bullet$}
     \put(3.05,3.3){$\circ$}
     \end{picture}}
\begin{center}
\begin{picture}(28,6)
 \put(0,1.5){$p=$}
 \put(1.5,0){\usebox{\boxpexample}}
 \put(8,1.5){$p'=$}
 \put(10,2){\usebox{\boxpexample}}
 \put(9.25,3.05){\parti{2}{1}}
 \put(10.7,1){$\otimes$}
 \put(9.25,3.05){\partii{1}{2}{3}}
 \put(10.3,-0.3){$\circ$}
 \put(14,1.5){$=$}
 \put(15,0){\usebox{\boxpcap}}
\end{picture}
\end{center}

(c) By (a), we have $\tilde p_2\otimes\tilde p_1\in\CC$ and thus  $p_1\otimes p_2\otimes\tilde p_2\otimes\tilde p_1\in\CC$. Using (b), we infer $p_1\otimes \tilde p_1\in\CC$ and likewise $p_2\otimes\tilde p_2\in\CC$ (using rotation). As an example:

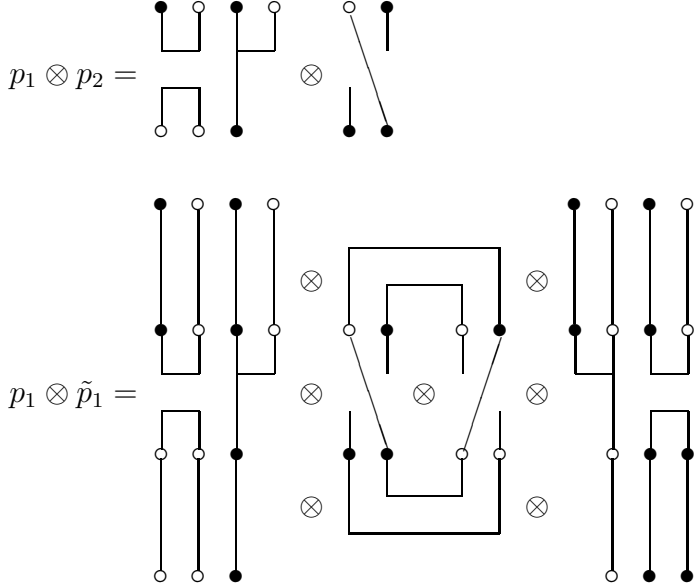
\setlength{\unitlength}{0.5cm}
\newsavebox{\boxppzwei}
   \savebox{\boxppzwei}
   { \begin{picture}(2,3.5)
     \put(-1,4.35){\parti{1}{2}}
     \put(-1,0.35){\upparti{1}{1}}
     \put(0.3,3.3){\line(1,-3){1}}
     \put(0.05,0){$\bullet$}
     \put(1.05,0){$\bullet$}
     \put(0.05,3.3){$\circ$}
     \put(1.05,3.3){$\bullet$}
     \end{picture}}
\newsavebox{\boxppzweitilde}
   \savebox{\boxppzweitilde}
   { \begin{picture}(2,3.5)
     \put(-1,4.35){\parti{1}{1}}
     \put(-1,0.35){\upparti{1}{2}}
     \put(0.3,0.3){\line(1,3){1}}
     \put(0.05,0){$\circ$}
     \put(1.05,0){$\circ$}
     \put(0.05,3.3){$\circ$}
     \put(1.05,3.3){$\bullet$}
     \end{picture}}
\begin{center}
\begin{picture}(28,4)
 \put(0,1.5){$p_1\otimes p_2=$}
 \put(3.5,0){\usebox{\boxpexample}}
 \put(7.7,1.5){$\otimes$}
 \put(8.5,0){\usebox{\boxppzwei}}
\end{picture}
\end{center}
\begin{center}
\begin{picture}(28,11)
 \put(0,4.5){$p_1\otimes \tilde p_1=$}
 \put(3.5,2.9){\usebox{\boxpexample}}
 \put(7.7,4.5){$\otimes$}
 \put(8.5,2.9){\usebox{\boxppzwei}}
 \put(10.7,4.5){$\otimes$}
 \put(11.5,2.9){\usebox{\boxppzweitilde}}
 \put(13.7,4.5){$\otimes$}
 \put(14.5,2.9){\usebox{\boxptilde}}
 \put(2.75,6.6){\upparti{3}{1}}
 \put(2.75,6.6){\upparti{3}{2}}
 \put(2.75,6.6){\upparti{3}{3}}
 \put(2.75,6.6){\upparti{3}{4}}
 \put(2.75,6.6){\uppartii{2}{6}{10}}
 \put(2.75,6.6){\uppartii{1}{7}{9}}
 \put(2.75,6.6){\upparti{3}{12}}
 \put(2.75,6.6){\upparti{3}{13}}
 \put(2.75,6.6){\upparti{3}{14}}
 \put(2.75,6.6){\upparti{3}{15}}
 \put(7.7,7.5){$\otimes$}
 \put(13.7,7.5){$\otimes$}
 \put(7.7,1.5){$\otimes$}
 \put(13.7,1.5){$\otimes$}
 \put(2.75,4){\parti{3}{1}}
 \put(2.75,4){\parti{3}{2}}
 \put(2.75,4){\parti{3}{3}}
 \put(2.75,4){\partii{2}{6}{10}}
 \put(2.75,4){\partii{1}{7}{9}}
 \put(2.75,4){\parti{3}{13}}
 \put(2.75,4){\parti{3}{14}}
 \put(2.75,4){\parti{3}{15}}
 \put(3.8,-0.35){$\circ$}
 \put(4.8,-0.35){$\circ$}
 \put(5.8,-0.35){$\bullet$}
 \put(15.8,-0.35){$\circ$}
 \put(16.8,-0.35){$\bullet$}
 \put(17.8,-0.35){$\bullet$}
 \put(3.8,9.55){$\bullet$}
 \put(4.8,9.55){$\circ$}
 \put(5.8,9.55){$\bullet$}
 \put(6.8,9.55){$\circ$}
 \put(14.8,9.55){$\bullet$}
 \put(15.8,9.55){$\circ$}
 \put(16.8,9.55){$\bullet$}
 \put(17.8,9.55){$\circ$}
\end{picture}
\end{center}

(d) The composition $(r_1\otimes q\otimes r_2)p$ is in $\CC$ for suitable tensor products $r_1$ and $r_2$ of the identity partitions. As an example:

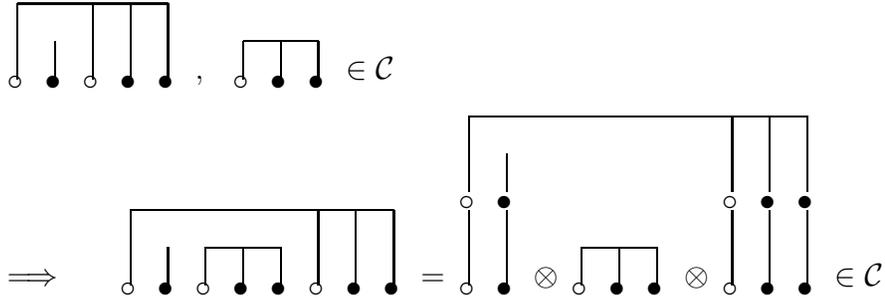
\setlength{\unitlength}{0.5cm}
\begin{center}
\begin{picture}(28,2.5)
 \put(0,0){\uppartiv{2}{1}{3}{4}{5}}
 \put(0,0){\upparti{1}{2}}
 \put(0,0){\uppartiii{1}{7}{8}{9}}
 \put(6,0){,}
 \put(10,0){$\in\CC$}
 \put(1,-0.3){$\circ$}
 \put(2,-0.3){$\bullet$}
 \put(3,-0.3){$\circ$}
 \put(4,-0.3){$\bullet$}
 \put(5,-0.3){$\bullet$}
 \put(7,-0.3){$\circ$}
 \put(8,-0.3){$\bullet$}
 \put(9,-0.3){$\bullet$}
\end{picture}
\end{center}
\begin{center}
\begin{picture}(28,5)
 \put(1,0){$\Longrightarrow$}
 \put(3,0){\uppartiv{2}{1}{6}{7}{8}}
 \put(3,0){\upparti{1}{2}}
 \put(3,0){\uppartiii{1}{3}{4}{5}}
 \put(4,-0.3){$\circ$}
 \put(5,-0.3){$\bullet$}
 \put(6,-0.3){$\circ$}
 \put(7,-0.3){$\bullet$}
 \put(8,-0.3){$\bullet$}
 \put(9,-0.3){$\circ$}
 \put(10,-0.3){$\bullet$}
 \put(11,-0.3){$\bullet$}
 \put(12,0){$=$}
 \put(12,0){\upparti{2}{1}}
 \put(12,0){\upparti{2}{2}}
 \put(12,0){\uppartiii{1}{4}{5}{6}}
 \put(12,0){\upparti{2}{8}}
 \put(12,0){\upparti{2}{9}}
 \put(12,0){\upparti{2}{10}}
 \put(12,2.5){\uppartiv{2}{1}{8}{9}{10}}
 \put(12,2.5){\upparti{1}{2}}
 \put(15,0){$\otimes$}
 \put(19,0){$\otimes$}
 \put(13,-0.3){$\circ$}
 \put(14,-0.3){$\bullet$}
 \put(16,-0.3){$\circ$}
 \put(17,-0.3){$\bullet$}
 \put(18,-0.3){$\bullet$}
 \put(20,-0.3){$\circ$}
 \put(21,-0.3){$\bullet$}
 \put(22,-0.3){$\bullet$}
 \put(13,2){$\circ$}
 \put(14,2){$\bullet$}
 \put(20,2){$\circ$}
 \put(21,2){$\bullet$}
 \put(22,2){$\bullet$}
 \put(23,0){$\in\CC$}
\end{picture}
\end{center}
\end{proof}

\begin{rem}\label{RemComp}
We will often refer to ``composing a partition $q\in P^{\twocol}(k,l)$ with a partition $p\in P^{\twocol}(n,m)$'' where $l\geq n$. By this we mean the composition $(r_1\otimes p\otimes r_2)q$ where $r_i\in P^{\twocol}(a_i,a_i)$ are suitable tensor products of the identity partitions respecting the coloring of the lower points of $q$ and moreover $a_1+n+a_2=l$. In this sense, composing a partition $p\in P^{\twocol}(k,l)$ with $\baarpartwb$ or $\baarpartbw$ yields Lemma \ref{RemErasePoints}(b).
\end{rem}

Tensor product,  composition, involution, and the operations of the preceding lemma are called the \emph{category operations}.

\subsection{Special operations on partitions}

The category operations may be performed in any category of partitions. Other procedures are allowed only if certain key partitions are contained in the category. 

\begin{lem}\label{LemPartRole}
Let $\CC$ be a category of partitions and let $p\in P^{\twocol}(0,l)$ be a partition without upper points.
\begin{itemize}
\item[(a)] If $\paarpartww\otimes\paarpartbb\in\CC$, then $\CC$ is closed under permutation of colors, i.e. if $p\in\CC$, then $p'\in \CC$, where $p'$ is obtained from $p$ by permutation of the colors of the points (without changing the strings connecting the points).
\item[(b)] If $\singletonw\otimes\singletonb\in\CC$, then $\CC$ is closed under permutation of colors of \emph{neighbouring} singletons. Furthermore, we may disconnect any point from a block and turn it into a singleton.
\item[(c)] If $\vierpartwwbb\in\CC$, then $\CC$ is closed under permutation of colors of \emph{neighbouring} points belonging to the same block. Furthermore, we may connect neighbouring blocks.
\item[(d)] If $\vierpartwbwb\in\CC$, we have no permutation of colors in general, and we may only connect neighbouring blocks if they meet at two points with inverse colors.
\item[(e)] If $\positionerwwbb\in\CC$, then $\CC$ is closed under arbitrary positioning of singletons, i.e. if $p\in\CC$, then $p'\in\CC$, where $p'$ is obtained from $p$ by shifting blocks of length one to other positions.
\item[(f)] If $\positionerwbwb\in\CC$, then we may swap a singleton with a neighbour point of inverse color. This procedure inverts both colors. In  other words, if $p=XabY\in\CC$ where $b$ is a singleton and $a$ is a point of color inverse to $b$, then $p'=Xb^{-1}a^{-1}Y\in\CC$.
\end{itemize}
\end{lem}
\begin{proof}
(a) By rotation, the partitions  $\idpartbw\otimes\idpartwb$ and $\idpartwb\otimes\idpartbw$ are in $\CC$. Note that the partitions $\idpartwb$ and $\idpartbw$ themselves are \emph{not} necessarily contained in $\CC$. 
Composing $p$ with $\idpartbw\otimes\idpartwb$ and $\idpartwb\otimes\idpartbw$ (in the sense of Remark \ref{RemComp}) yields a transposition of the colors of the points of $p$. 

(b) Using rotation and the tensor product, we infer   $\idpartsingletonwb\otimes\idpartsingletonbw$ and $\idpartsingletonbw\otimes\idpartsingletonwb$ are in $\CC$. Analoguous to (a), we see that we may permute the colors of neighbouring singletons. 
Furthermore, rotations of $\singletonw\otimes\singletonb$ yield partitions $\idpartsingletonww$ and $\idpartsingletonbb$ in $P^{\twocol}(1,1)$ consisting of two disconnected points of the same color. Composing a partition with these partitions (Remark \ref{RemComp}) yields a partition where some points are disconnected from their blocks (without changing the color).

(c) Again, similar to (a), we infer that $\CC$ is closed under  permutation of colors of neighbouring points belonging to the same block, using $\vierpartrotwbbw,\vierpartrotbwwb\in\CC$. Since we then also have $\vierpartwbwb\in\CC$, the partitions $\vierpartrotwbwb$, $\vierpartrotbwbw$, $\vierpartrotwwww$ and $\vierpartrotbbbb$ are all in $\CC$ by rotation. Composing with them effects that some neighbouring blocks are connected.

(d) We argue as in (c), but we may only use $\vierpartrotwbwb$ and $\vierpartrotbwbw$.

(e) Check that $\singletonw\idpartww\downsingletonw$, $\singletonb\idpartbb\downsingletonb$, $\singletonw\idpartbb\downsingletonw$, $\singletonb\idpartww\downsingletonb$ etc are in $\CC$ using rotation and verticolor reflection.

(f) Use $\singletonw\idpartwb\downsingletonb$ or $\singletonb\idpartbw\downsingletonw$ etc.

Here are some examples concerning the above operations.

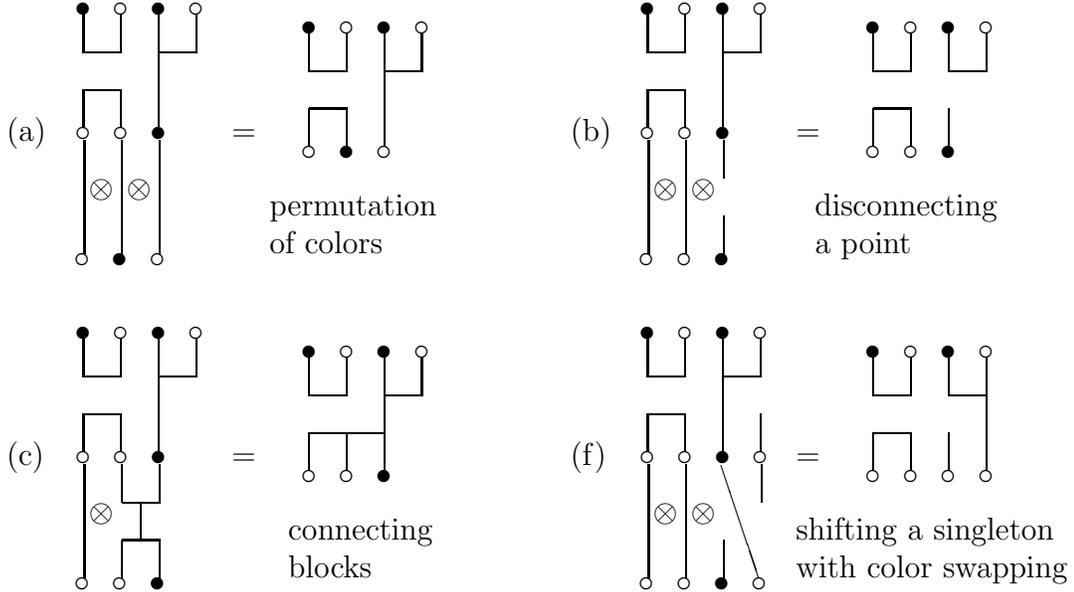
\setlength{\unitlength}{0.5cm}
\newsavebox{\boxppermuted}
   \savebox{\boxppermuted}
   { \begin{picture}(3,3.5)
     \put(-1,4.35){\partii{1}{1}{2}}
     \put(-1,4.35){\partii{1}{3}{4}}
     \put(-1,0.35){\uppartii{1}{1}{2}}
     \put(-1,0.35){\upparti{2}{3}}
     \put(0.05,0){$\circ$}
     \put(1.05,0){$\bullet$}
     \put(2.05,0){$\circ$}
     \put(0.05,3.3){$\bullet$}
     \put(1.05,3.3){$\circ$}
     \put(2.05,3.3){$\bullet$}
     \put(3.05,3.3){$\circ$}
     \end{picture}}
\newsavebox{\boxpdisconnected}
   \savebox{\boxpdisconnected}
   { \begin{picture}(3,3.5)
     \put(-1,4.35){\partii{1}{1}{2}}
     \put(-1,4.35){\partii{1}{3}{4}}
     \put(-1,0.35){\uppartii{1}{1}{2}}
     \put(-1,0.35){\upparti{1}{3}}
     \put(0.05,0){$\circ$}
     \put(1.05,0){$\circ$}
     \put(2.05,0){$\bullet$}
     \put(0.05,3.3){$\bullet$}
     \put(1.05,3.3){$\circ$}
     \put(2.05,3.3){$\bullet$}
     \put(3.05,3.3){$\circ$}
     \end{picture}}
\newsavebox{\boxpconnected}
   \savebox{\boxpconnected}
   { \begin{picture}(3,3.5)
     \put(-1,4.35){\partii{1}{1}{2}}
     \put(-1,4.35){\partii{1}{3}{4}}
     \put(-1,0.35){\uppartiii{1}{1}{2}{3}}
     \put(-1,0.35){\upparti{2}{3}}
     \put(0.05,0){$\circ$}
     \put(1.05,0){$\circ$}
     \put(2.05,0){$\bullet$}
     \put(0.05,3.3){$\bullet$}
     \put(1.05,3.3){$\circ$}
     \put(2.05,3.3){$\bullet$}
     \put(3.05,3.3){$\circ$}
     \end{picture}}
\newsavebox{\boxpsingleton}
   \savebox{\boxpsingleton}
   { \begin{picture}(3,3.5)
     \put(-1,4.35){\partii{1}{1}{2}}
     \put(-1,4.35){\partii{1}{3}{4}}
     \put(-1,0.35){\uppartii{1}{1}{2}}
     \put(-1,0.35){\upparti{2}{3}}
     \put(-1,0.35){\upparti{1}{4}}
     \put(0.05,0){$\circ$}
     \put(1.05,0){$\circ$}
     \put(2.05,0){$\bullet$}
     \put(3.05,0){$\circ$}
     \put(0.05,3.3){$\bullet$}
     \put(1.05,3.3){$\circ$}
     \put(2.05,3.3){$\bullet$}
     \put(3.05,3.3){$\circ$}
     \end{picture}}
\newsavebox{\boxpshifted}
   \savebox{\boxpshifted}
   { \begin{picture}(3,3.5)
     \put(-1,4.35){\partii{1}{1}{2}}
     \put(-1,4.35){\partii{1}{3}{4}}
     \put(-1,0.35){\uppartii{1}{1}{2}}
     \put(-1,0.35){\upparti{1}{3}}
     \put(-1,0.35){\upparti{2}{4}}
     \put(0.05,0){$\circ$}
     \put(1.05,0){$\circ$}
     \put(2.05,0){$\circ$}
     \put(3.05,0){$\circ$}
     \put(0.05,3.3){$\bullet$}
     \put(1.05,3.3){$\circ$}
     \put(2.05,3.3){$\bullet$}
     \put(3.05,3.3){$\circ$}
     \end{picture}}
\begin{center}
\begin{picture}(28,8)
 \put(2,3){\usebox{\boxpexample}}
 \put(1.3,4){\parti{3}{1}}
 \put(1.3,4){\parti{3}{2}}
 \put(1.3,4){\parti{3}{3}}
 \put(2.3,-0.35){$\circ$}
 \put(3.3,-0.35){$\bullet$}
 \put(4.3,-0.35){$\circ$}
 \put(2.7,1.5){$\otimes$}
 \put(3.7,1.5){$\otimes$}
 \put(6.5,3){$=$}
 \put(8,2.5){\usebox{\boxppermuted}}
 \put(7.5,1){permutation}
 \put(7.5,0){of colors}
 \put(0.5,3){(a)}
 \put(17,3){\usebox{\boxpexample}}
 \put(16.3,4){\parti{3}{1}}
 \put(16.3,4){\parti{3}{2}}
 \put(16.3,4){\parti{1}{3}}
 \put(16.3,2){\parti{1}{3}}
 \put(17.3,-0.35){$\circ$}
 \put(18.3,-0.35){$\circ$}
 \put(19.3,-0.35){$\bullet$}
 \put(17.7,1.5){$\otimes$}
 \put(18.7,1.5){$\otimes$}
 \put(21.5,3){$=$}
 \put(23,2.5){\usebox{\boxpdisconnected}}
 \put(22,1){disconnecting}
 \put(22,0){a point}
 \put(15.5,3){(b)}
\end{picture}
\end{center}
\begin{center}
\begin{picture}(28,8)
 \put(2,3){\usebox{\boxpexample}}
 \put(1.3,4){\parti{3}{1}}
 \put(1.3,4){\partii{1}{2}{3}}
 \put(1.3,0){\uppartii{1}{2}{3}}
 \put(1.8,3){\parti{1}{2}}
 \put(2.3,-0.35){$\circ$}
 \put(3.3,-0.35){$\circ$}
 \put(4.3,-0.35){$\bullet$}
 \put(2.7,1.5){$\otimes$}
 \put(6.5,3){$=$}
 \put(8,2.5){\usebox{\boxpconnected}}
 \put(8,1){connecting}
 \put(8,0){blocks}
 \put(0.5,3){(c)}
 \put(17,3){\usebox{\boxpsingleton}}
 \put(16.3,4){\parti{3}{1}}
 \put(16.3,4){\parti{3}{2}}
 \put(16.3,4){\parti{1}{4}}
 \put(16.3,0){\upparti{1}{3}}
 \put(19.5,3){\line(1,-3){1}}
 \put(17.3,-0.35){$\circ$}
 \put(18.3,-0.35){$\circ$}
 \put(19.3,-0.35){$\bullet$}
 \put(20.3,-0.35){$\circ$}
 \put(17.7,1.5){$\otimes$}
 \put(18.7,1.5){$\otimes$}
 \put(21.5,3){$=$}
 \put(23,2.5){\usebox{\boxpshifted}}
 \put(21.5,1){shifting a singleton}
 \put(21.5,0){with color swapping}
 \put(15.5,3){(f)}
\end{picture}
\end{center}
\end{proof}

We formulated the above lemma only for partitions having no upper points, but the statements may be extended to arbitrary partitions $p\in P^{\twocol}(k,l)$. We then have to take into account that the colors are inverted whenever they are rotated from the upper line to the lower line or the converse.

\subsection{The non- (or one-) colored case}

Let us end this section with a comparison to the case of categories of non-colored partitions, which were studied in \cite{banica2009liberation, banicacurranspeicher2010,weber2013classification,raum2014combinatorics,raum2013easy} and in other articles and which were completely classified in \cite{raum2013full}. For the classification in the noncrossing case, see \cite{banica2009liberation} and \cite{weber2013classification}. Recall that there are exactly seven categories, given by:
\begin{align*}
 \langle\singleton\rangle &\qquad\supset &\langle\positioner\rangle &\qquad\supset &\langle\singleton\otimes\singleton\rangle &\qquad\supset &\langle\emptyset\rangle=NC_2\\
\downsubset\; & &\downsubset\; & & & &\downsubset\;\\
\quad\\
 \langle\singleton, \vierpart\rangle=NC &\qquad\supset &\langle\singleton\otimes\singleton, \vierpart\rangle & &\supset &  &\langle\vierpart\rangle
\end{align*}

By $P(k,l)$ we denote the set of \emph{non-colored partitions} where all points have no color. Likewise we use the notations $P$ for all non-colored partitions and $NC$ for all non-colored noncrossing partitions. Categories of non-colored partitions are defined like categories of two-colored partitions when forgetting all colors, see for instance \cite{banica2009liberation} or \cite{raum2013full}.
The key link between non-colored categories and two-colored categories is given by the partition $\paarpartww$ as may be seen in the next proposition. Note that $\idpartwb$ and $\idpartbw$ are rotated and possibly verticolor reflected versions of $\paarpartww$. Composing a partition $p$ with these partitions, we can change the colors of the points of $p$ to every possible color pattern. Hence, categories containing $\paarpartww$ are non-colored categories, in this sense.

To be more precise, let $\Psi:P^{\twocol}\to P$ be the map given by forgetting the colors of a two-colored partition. For a set $\CC\subset P$, we denote by $\Psi^{-1}(\CC)\subset P^{\twocol}$ its preimage under $\Psi$.

\begin{prop}\label{PropOneColored}
\begin{itemize}
\item[(a)] Let $\CC\subset P$ be a category of non-colored partitions. Then $\Psi^{-1}(\CC)\subset P^{\twocol}$ is a category of two-colored partitions containing the unicolored pair partition $\paarpartww$ (or equivalently $\paarpartbb$).
\item[(b)] Let $\CC\subset P^{\twocol}$ be a category of two-colored partitions containing the unicolored pair partition $\paarpartww$ (or equivalently $\paarpartbb$). Then $\Psi(\CC)\subset P$ is a category of non-colored partitions and $\Psi^{-1}(\Psi(\CC))=\CC$.
\end{itemize}
Hence, there is a one-to-one correspondence between categories of non-colored partitions and categories of two-colored partitions containing $\paarpartww$.
\end{prop}
\begin{proof}
(a) It is easy to see from the definition that $\Psi^{-1}(\CC)$ is a category of partitions. Furthermore, $\paarpartww\in\Psi^{-1}(\CC)$ since $\Psi(\paarpartww)=\paarpart\in\CC$.

(b) It is easy to see that $\Psi(\CC)$ is closed under tensor product and involution and that it contains the pair partition $\paarpart$ and the identity partition $\idpart$. The composition is a bit more subtle.
If $p,q\in \Psi(\CC)$, their composition is in $\Psi(\CC)$ only if we can lift $p$ and $q$ to partitions in $\CC$ whose color patterns allow the composition in $P^{\twocol}$. 
But since $\idpartwb$ and $\idpartbw$ are in $\CC$ (by rotation), we can do so: If $p\in\Psi(\CC)$, there is a partition $p_0\in \CC$ such that $\Psi(p_0)=p$. Composing it with tensor products of $\idpartwb$ and $\idpartbw$, we may assume that all points of $p_0$ are white. Now, $\Psi(\CC)$ is closed under composition since $\CC$ is. Similarly, we prove $\Psi^{-1}(\Psi(\CC))\subset\CC$ using $\idpartwb$ and $\idpartbw$; the converse direction is trivial.
\end{proof}

\section{Dividing the categories into cases}

The classification of categories of noncrossing partitions is given by a detailed case study which we will now prepare.

\subsection{The cases $\OOO$, $\HHH$, $\SSS$ and $\BBB$} 

The first division into cases is given by  the sizes of blocks. The next lemma is formulated for arbitrary categories of partitions (not necessarily noncrossing ones).

\begin{lem}\label{LemCases}
Let $\CC\subset P^{\twocol}$ be a category of partitions.
\begin{itemize}
\item[(a)] If $\singletonw\otimes\singletonb\notin\CC$, then all blocks of  partitions $p\in\CC$ have length at least two.
\item[(b)] If $\vierpartwbwb\notin\CC$, then all blocks of partitions $p\in\CC$ have length at most two.
\end{itemize}
\end{lem}
\begin{proof}
(a) Let $p\in\CC$ be a partition containing a block of size one. By rotation and possibly verticolor reflection, we may assume that it is of the form $\singletonw\otimes q$, with no upper points. By Lemma \ref{RemErasePoints}(c) we have $\singletonw\otimes\singletonb\in\CC$.

(b) Let $p\in\CC$ be a partition containing a block of size at least three. By rotation, it is of the form $p=a^{\epsilon_1}X_1a^{\epsilon_2}X_2a^{\epsilon_3}X_3$ with no upper points, where the points $a^{\epsilon_i}$ belong to the same block, and $\epsilon_i\in\{1,-1\}$ depending on the color. The subwords $X_1,X_2$ and $X_3$ are possibly connected to the block on the $a^{\epsilon_i}$. By verticolor reflection, we infer that the following partition is in $\CC$.
\[p\otimes\tilde p=a^{\epsilon_1}X_1a^{\epsilon_2}X_2a^{\epsilon_3}X_3\otimes \tilde{X_3}b^{-\epsilon_3}\tilde{X_2}b^{-\epsilon_2}\tilde{X_1}b^{-\epsilon_1}\]
By \ref{RemErasePoints}(b), we obtain that a partition $q:=a^{\epsilon_1}X_1'a^{\epsilon_2}a^{-\epsilon_2}\tilde{X_1'}a^{-\epsilon_1}$ is in $\CC$. Note that while the blocks on $a^{\epsilon_i}$ and $b^{-\epsilon_i}$ are \emph{not} connected in $p\otimes\tilde p$, the points $a^{\epsilon_i}$ and $a^{-\epsilon_i}$ \emph{are} connected in $q$ due to the procedure as described in Lemma \ref{RemErasePoints}(b). 
Using rotation, we infer that the partion $q':=a^{-\epsilon_1}a^{\epsilon_1}X_1'a^{\epsilon_2}a^{-\epsilon_2}\tilde{X_1'}$ is in $\CC$. Again, tensoring it with its verticolor reflected version and using Lemma \ref{RemErasePoints}(b),  we obtain  $a^{-1}aa^{-1}a\in\CC$ which implies $\vierpartwbwb\in\CC$.

The scheme of the proof is:

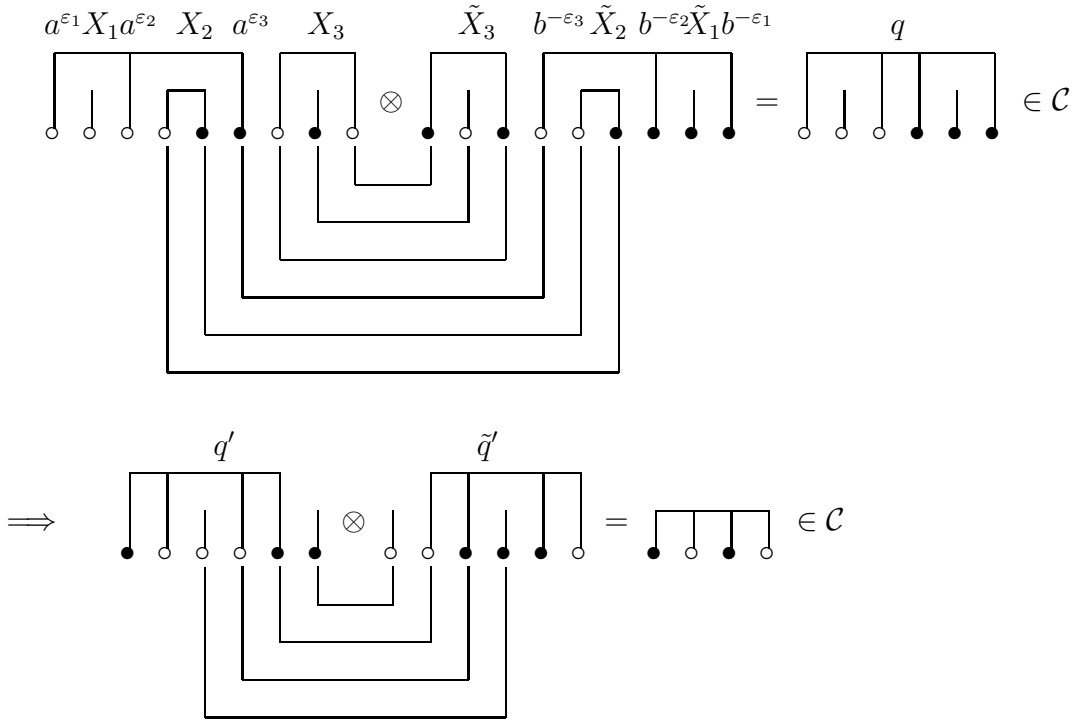
\setlength{\unitlength}{0.5cm}
\begin{center}
\begin{picture}(28,10)
 \put(0,6){\uppartiii{2}{1}{3}{6}}
 \put(0,6){\upparti{1}{2}}
 \put(0,6){\uppartii{1}{4}{5}}
 \put(0,6){\uppartii{2}{7}{9}}
 \put(0,6){\upparti{1}{8}} 
 \put(1,5.65){$\circ$}
 \put(2,5.65){$\circ$}
 \put(3,5.65){$\circ$}
 \put(4,5.65){$\circ$}
 \put(5,5.65){$\bullet$}
 \put(6,5.65){$\bullet$}
 \put(7,5.65){$\circ$}
 \put(8,5.65){$\bullet$}
 \put(9,5.65){$\circ$}
 \put(9.9,6.5){$\otimes$}
 \put(10,6){\uppartiii{2}{4}{7}{9}}
 \put(10,6){\upparti{1}{8}}
 \put(10,6){\uppartii{1}{5}{6}}
 \put(10,6){\uppartii{2}{1}{3}}
 \put(10,6){\upparti{1}{2}} 
 \put(19,5.65){$\bullet$}
 \put(18,5.65){$\bullet$}
 \put(17,5.65){$\bullet$}
 \put(16,5.65){$\bullet$}
 \put(15,5.65){$\circ$}
 \put(14,5.65){$\circ$}
 \put(13,5.65){$\bullet$}
 \put(12,5.65){$\circ$}
 \put(11,5.65){$\bullet$}
 \put(19.9,6.5){$=$}
 \put(20,6){\uppartiv{2}{1}{3}{4}{6}}
 \put(20,6){\upparti{1}{2}}
 \put(20,6){\upparti{1}{5}} 
 \put(21,5.65){$\circ$}
 \put(22,5.65){$\circ$}
 \put(23,5.65){$\circ$}
 \put(24,5.65){$\bullet$}
 \put(25,5.65){$\bullet$}
 \put(26,5.65){$\bullet$}
 \put(0,6.5){\partii{6}{4}{16}}
 \put(0,6.5){\partii{5}{5}{15}}
 \put(0,6.5){\partii{4}{6}{14}}
 \put(0,6.5){\partii{3}{7}{13}}
 \put(0,6.5){\partii{2}{8}{12}}
 \put(0,6.5){\partii{1}{9}{11}}
 \put(1,8.5){$a^{\epsilon_1}$}
 \put(2,8.5){$X_1$}
 \put(3,8.5){$a^{\epsilon_2}$}
 \put(4.5,8.5){$X_2$}
 \put(6,8.5){$a^{\epsilon_3}$}
 \put(8,8.5){$X_3$}
 \put(12,8.5){$\tilde X_3$}
 \put(14,8.5){$b^{-\epsilon_3}$}
 \put(15.5,8.5){$\tilde X_2$}
 \put(16.8,8.5){$b^{-\epsilon_2}$}
 \put(18,8.5){$\tilde X_1$}
 \put(19,8.5){$b^{-\epsilon_1}$}
 \put(23.5,8.5){$q$}
 \put(27,6.5){$\in\CC$}
\end{picture}
\end{center}
\begin{center}
\begin{picture}(28,8.5)
 \put(0,4.5){$\Longrightarrow$}
 \put(2,4){\uppartiv{2}{1}{2}{4}{5}}
 \put(2,4){\upparti{1}{3}}
 \put(2,4){\upparti{1}{6}}
 \put(3,3.65){$\bullet$}
 \put(4,3.65){$\circ$}
 \put(5,3.65){$\circ$}
 \put(6,3.65){$\circ$}
 \put(7,3.65){$\bullet$}
 \put(8,3.65){$\bullet$}
 \put(8.9,4.5){$\otimes$}
 \put(9,4){\uppartiv{2}{2}{3}{5}{6}}
 \put(9,4){\upparti{1}{1}}
 \put(9,4){\upparti{1}{4}}
 \put(10,3.65){$\circ$}
 \put(11,3.65){$\circ$}
 \put(12,3.65){$\bullet$}
 \put(13,3.65){$\bullet$}
 \put(14,3.65){$\bullet$}
 \put(15,3.65){$\circ$}
 \put(15.9,4.5){$=$}
 \put(16,4){\uppartiv{1}{1}{2}{3}{4}}
 \put(17,3.65){$\bullet$}
 \put(18,3.65){$\circ$}
 \put(19,3.65){$\bullet$}
 \put(20,3.65){$\circ$}
 \put(-1,4.5){\partii{4}{6}{14}}
 \put(-1,4.5){\partii{3}{7}{13}}
 \put(-1,4.5){\partii{2}{8}{12}}
 \put(-1,4.5){\partii{1}{9}{11}}
 \put(5.5,6.5){$q'$}
 \put(12.5,6.5){$\tilde q'$}
 \put(21,4.5){$\in\CC$}
\end{picture}
\end{center}
\end{proof}

Note that unlike in the non-colored case, $\singletonw\otimes\singletonb\notin\CC$ does \emph{not} imply that all blocks have even size. Consider for instance $\langle\dreipartwww\rangle$.

\begin{defn}\label{DefCases}
Let $\CC\subset P^{\twocol}$ be a category of partitions. We say that:
\begin{itemize}
\item $\CC$ is in \emph{case $\OOO$}, if $\singletonw\otimes\singletonb\notin\CC$ and $\vierpartwbwb\notin\CC$.
\item $\CC$ is in \emph{case $\BBB$}, if $\singletonw\otimes\singletonb\in\CC$ and $\vierpartwbwb\notin\CC$.
\item $\CC$ is in \emph{case $\HHH$}, if $\singletonw\otimes\singletonb\notin\CC$ and $\vierpartwbwb\in\CC$.
\item $\CC$ is in \emph{case $\SSS$}, if $\singletonw\otimes\singletonb\in\CC$ and $\vierpartwbwb\in\CC$.
\end{itemize}
\end{defn}

The coice of the letters $\OOO, \BBB, \HHH,\SSS$ comes from the non-colored situation, \cite{weber2013classification}.

\subsection{Global and local colorization}

It is convenient to study the categories $\CC\subset NC^{\twocol}$ case by case according to the above definition. According to Lemma \ref{LemPartRole}(a), we divide each of these cases into two subcases: Those categories $\CC$ containing $\paarpartww\otimes\paarpartbb$ behave very differently from those not containing this partition.

\begin{defn}\label{DefGlobalColor}
A category of partitions $\CC\subset P^{\twocol}$ is 
\begin{itemize}
 \item \emph{globally colorized}, if $\paarpartww\otimes\paarpartbb\in\CC$
 \item and \emph{locally colorized} if $\paarpartww\otimes\paarpartbb\notin\CC$.
\end{itemize}
\end{defn}

\subsection{The global parameter $k(\CC)$}

By Lemma \ref{LemPartRole}, we may permute the colors of the points of partitions in globally colorized categories.
Hence the coloring of partitions turns out to be of a global nature -- the difference between the number of white and black points is the only number that matters for the coloring of a partition in such categories.

\begin{defn}
Let $p\in P^{\twocol}$. 
\begin{itemize}
 \item Denote by $c_\circ(p)$ the sum of the number of white points on the lower line of $p$ and the number of black points on the upper line. 
 \item Denote by $c_\bullet(p)$ the sum of the number of black points on the lower line of $p$ and the number of  white points on the upper line.
 \item Define $c: P^{\twocol}\to \Z$ by $c(p):= c_{\circ}(p)-c_{\bullet}(p)$.
\end{itemize}
\end{defn}

We will mainly consider partitions $p\in P^{\twocol}(0,l)$ with no upper points. In this case $c_{\circ}$ is counting the white points whereas $c_{\bullet}$ is counting the black points of a partition.  Recall that rotating black points from the upper line to the lower line turns them into white points. 

\begin{defn}
Let $\CC$ be a category of  partitions. We set $k(\CC)$ as the minimum  of all numbers $c(p)$ such that $c(p)>0$ and $p\in\CC$, if such a  partition exists in $\CC$. Otherwise $k(\CC):=0$. 
The parameter $k(\CC)$ is called the \emph{degree of reflection} of $\CC$. It is the \emph{global parameter} of $\CC$.
\end{defn}

Note that we always find a partition $p$ in $\CC$ such that $c(p)=0$; take for instance $p=\paarpartwb$. In the next lemma, we show that the map $c: P^{\twocol}\to \Z$ behaves well with respect to the category operations. In particular, if there  exists a partition $p\in\CC$ such that $c(p)<0$, then $\tilde p\in\CC$ and $c(\tilde p)=-c(p)>0$.

\begin{lem}\label{LemC}
For the map $c: P^{\twocol}\to \Z$ the following holds true.
\begin{itemize}
\item[(a)] $c(p\otimes q)=c(p)+c(q)$
\item[(b)] $c(pq)=c(p)+c(q)$
\item[(c)] $c(p^*)=-c(p)$
\item[(d)] $c(p')=c(p)$, if $p'$ is obtained from $p$ by rotation.
\item[(e)] $c(\tilde p)=-c(p)$
\end{itemize}
\end{lem}
\begin{proof}
From the definition it is clear that (a), (c), (d) and (e) hold. 
To see the invariance under composition, let $w_1$ be the number of upper white points of $q$, and $b_1$ be the number of upper black points. Let $w_2$ be the number of lower white points of $q$ and likewise $b_2$ for the black points. Since $p$ and $q$ are composable, the numbers $w_2$ and $b_2$ also count the number of upper white and upper black points of $p$, respectively. Finally, let $w_3$ and $b_3$ be the number of lower white and black points of $p$ respectively.
We thus have:
\begin{align*}
c(q)&= c_\circ(q)-c_\bullet(q)=(w_2+b_1)-(w_1+b_2)\\
c(p)&=(w_3+b_2)-(w_2+b_3)\\
c(pq)&=(w_3+b_1)-(w_1+b_3)
\end{align*}
This implies $c(pq)=c(p)+c(q)$.
\end{proof}

The global parameter $k(\CC)$ gives rise to a complete description of all possible numbers $c(p)$ of a category $\CC$.

\begin{prop}\label{LemKZ}
Let $\CC\subset P^{\twocol}$ be a category of partitions and let $k:=k(\CC)\in\N_0$. Then $c(p)\in k\Z$ for all partitions $p\in \CC$.
\end{prop}
\begin{proof}
The statement is obvious for $k=0$ by definition and Lemma \ref{LemC}(e), thus we may assume $k>0$.
Let $p\in\CC$, such that $c(p)\neq 0$. By Lemma \ref{LemC}(e), we may restrict to $c(p)>0$. 
Now, assume that there is a number $m\in\N_0$ such that $km<c(p)<k(m+1)$. By the definition of $k(\CC)$, there is a partition $q\in\CC$ such that $c(q)=k$. Put $r:=\tilde q^{\otimes m}\otimes p$. Then $r\in\CC$ and $c(r)=-mc(q)+c(p)$. Hence $0<c(r)<k$ which contradicts the definition of $k(\CC)$.
\end{proof}

\subsection{The local parameters $d(\CC)$ and $K^\CC(\cutpaarpartbb)$}

We also have some local para-meters. The idea is to determine possible numbers $c(p_1)$ of subpartitions $p_1$ between two legs of a block of a partition $p\in\CC$. The situation when these two legs have the same color behaves quite differently from the case of equally colored legs.

\setlength{\unitlength}{0.5cm}
\begin{center}
\begin{picture}(22,2)
\put(0,0){$p=\quad\ldots \quad\circ \; p_1 \bullet\quad \ldots$}
\put(4.5,0.4){\line(0,1){1}}
\put(6.3,0.4){\line(0,1){1}}
\put(4,1.4){\line(1,0){2.8}}
\put(2.9,1.35){$\ldots$}
\put(6.9,1.35){$\ldots$}
\put(10,0){or}
\put(12,0){$p=\quad\ldots \quad\circ \; p_1 \circ\quad \ldots$}
\put(16.5,0.4){\line(0,1){1}}
\put(18.3,0.4){\line(0,1){1}}
\put(16,1.4){\line(1,0){2.8}}
\put(14.9,1.35){$\ldots$}
\put(18.9,1.35){$\ldots$}
\end{picture}
\end{center}

By rotation, we can reduce it to the following situation.

\begin{defn}
\begin{itemize}
\item Let $p\in NC^{\twocol}(0,l)$ be a partition with no upper points. Assume that $p$ can be decomposed as $p=p_1\otimes p_2$ where $p_1\neq \emptyset$, the partition $p_2$ has at least two points and the first and the last point of $p_2$ belong to the same block. Then we say that $p=p_1\otimes p_2$ is in \emph{nest decomposed form}. 
\item Let $NDF(\cutpaarpartwb)$ be the set of all  partitions $p=p_1\otimes p_2$ in nest decomposed form such that the first  point of $p_2$ is white and the last one is black. By $NDF^{\CC}(\cutpaarpartwb)$
we denote the intersection of $NDF(\cutpaarpartwb)$ and $\CC\subset P^{\twocol}$.
Likewise we use the notations $NDF^{\CC}(\cutpaarpartbw)$, $NDF^{\CC}(\cutpaarpartww)$ and $NDF^{\CC}(\cutpaarpartbb)$ for the three other cases.
\item By $K^{\CC}(\cutpaarpartwb)$ we denote the set of all numbers $c(p_1)\in\Z$ such that $p=p_1\otimes p_2\in NDF^{\CC}(\cutpaarpartwb)$. Likewise we define $K^{\CC}(\cutpaarpartbw)$, $K^{\CC}(\cutpaarpartww)$ and $K^{\CC}(\cutpaarpartbb)$.
\item Let $\CC\subset NC^{\twocol}$ be a category of noncrossing partitions. We define the following \emph{local parameter} $d(\CC)$. If $K^{\CC}(\cutpaarpartwb)$ contains a number $t>0$, we put $d(\CC)$ as the minimum of those numbers. Otherwise $d(\CC):=0$.
\end{itemize}
\end{defn}

\begin{ex}\label{ExNDF}
\begin{itemize}
\item[(a)] The partition $\paarpartww\otimes\paarpartbb$ is in $NDF(\cutpaarpartbb)$ with $c(p_1)=2$, whereas $\paarpartbb\otimes\paarpartww$ is in $NDF(\cutpaarpartww)$ with $c(p_1)=-2$. 
\item[(b)] The partition $p:=\paarpartww\otimes\paarpartww\otimes\paarpartww$ is in nest decomposed form, where $p_1=\paarpartww\otimes\paarpartww$ and $p_2=\paarpartww$. It is in $NDF(\cutpaarpartww)$ with $c(p_1)=4$. The partition $q:=\paarpartww\otimes\paarpartbb\otimes\paarpartww$ is in $NDF(\cutpaarpartww)$ with the decomposition $q_1:=\paarpartww\otimes\paarpartbb$ and $q_2:=\paarpartww$. 
\end{itemize}
\end{ex}

When working on the classfication of noncrossing categories, we will be interested in the sets $K^{\CC}(\cutpaarpartwb)$,  $K^{\CC}(\cutpaarpartbw)$, $K^{\CC}(\cutpaarpartww)$ and $K^{\CC}(\cutpaarpartbb)$ as local parameters. In the remainder of this subsection, we shall prove that we always have $K^{\CC}(\cutpaarpartwb)=K^{\CC}(\cutpaarpartbw)=d\Z$ for $d=d(\CC)$ and $K^{\CC}(\cutpaarpartww)=-K^{\CC}(\cutpaarpartbb)$. Thus, the study of these local parameters boils down to knowing $d(\CC)$ and $K^{\CC}(\cutpaarpartbb)$. Note that $\paarpartwb\otimes\paarpartwb$ is in $NDF^{\CC}(\cutpaarpartwb)$ for all $\CC$, hence  $K^{\CC}(\cutpaarpartwb)$ always contain the zero. As for  $K^{\CC}(\cutpaarpartbb)$ the situation is a bit more complicated -- for instance, it could be empty.

\begin{prop}\label{LemDDivisor}
Let $\CC\subset NC^{\twocol}$ be a category of noncrossing partitions. Then, $K^{\CC}(\cutpaarpartwb)=d\Z$ for $d=d(\CC)$. Furthermore, if  $k(\CC)\neq 0$, then $d(\CC)\neq 0$ and $d(\CC)$ is a divisor of $k(\CC)$.
\end{prop}
\begin{proof}
 We only need to show that $s+t$ and $-s$ are in $K^{\CC}(\cutpaarpartwb)$ whenever $s,t\in K^{\CC}(\cutpaarpartwb)$. Let $p=p_1\otimes p_2$ and $q=q_1\otimes q_2$ be two partitions in $NDF^{\CC}(\cutpaarpartwb)$. Firstly, we have $p_1\otimes q_1\otimes q_2\otimes p_2\in\CC$ using rotation. Using the pair partition $\baarpartbw$ and composition, we can connect the last point of $q_2$ (which is black) with the first point of $p_2$ (which is white). This erases these two points and we obtain a partition $p_1\otimes q_1\otimes r\in\CC$ in nest decomposed form such that the first point of $r$ is white and the last one is black. Thus, $c(p_1)+c(q_1)\in K^{\CC}(\cutpaarpartwb)$. Secondly, we have $\tilde p_2\otimes \tilde p_1\in\CC$, which yields $\tilde p_1\otimes \tilde p_2\in\CC$ by rotation. Now, this is a partition in nest decomposed form, such that the first point of $\tilde p_2$ is white and the last one is black. Thus, $-c(p_1)=c(\tilde p_1)\in  K^{\CC}(\cutpaarpartwb)$. 

We deduce $K^{\CC}(\cutpaarpartwb)=d\Z$ using the Euklidean algorithm like in the proof of Proposition \ref{LemKZ}. Furthermore, note that $k(\CC)\in K^{\CC}(\cutpaarpartwb)$, since $p\otimes \paarpartwb\in NDF^{\CC}(\cutpaarpartwb)$ for all $p\in\CC$. Hence, $d(\CC)\neq 0$ whenever $k(\CC)\neq 0$ and  also $k(\CC)$ is a multiple of $d(\CC)$.
\end{proof}

\begin{defn}\label{DefBs}
By $b_s$ we denote the partition in $P^{\twocol}(0,s)$ consisting of a single block of length $s$ such that all points are white, hence $b_2=\paarpartww$, $b_3=\dreipartwww$, etc.
\end{defn}

\begin{defn}\label{DefNotationNegativ}
For $s<0$ we use the following notation:
\[b_s:=\tilde b_{-s}, \singletonw^{\otimes s}:=\singletonb^{\otimes -s}, \singletonb^{\otimes s}:=\singletonw^{\otimes -s}\]
\end{defn}

The next lemma is of quite technical nature, but it will be needed in this subsection as well as in the remainder of this article several times.

\begin{lem}\label{LemBsInC}
Let $\CC\subset NC^{\twocol}$ be a category of noncrossing partitions.
\begin{itemize}
\item[(a)]  Let $\singletonw\otimes\singletonb\in\CC$. Let $q\in\{\cutpaarpartww,\cutpaarpartwb,\cutpaarpartbw,\cutpaarpartbb\}$ and let $s\in K^{\CC}(q)$. Then  $\singletonw^{\otimes s} a^{\epsilon_1} \singletonw^{\otimes -(s+c(q))} a^{\epsilon_2}\in\CC$, where $a^{\epsilon_1}$ and $a^{\epsilon_2}$ form the pair block $q$. In particular $s\in K^{\CC}(\cutpaarpartwb)$ if and only if $\positioners\in\CC$, and $K^{\CC}(\cutpaarpartwb)= K^{\CC}(\cutpaarpartbw)$.
\item[(b)] Let $p\in\CC$ be a partition such that two neighbouring points have the same color and belong to the same block. Then $\vierpartwwbb\in\CC$ or $\paarpartww\otimes\paarpartbb\in\CC$.
\item[(c)] If $\singletonw\otimes\singletonb\notin\CC$ and $NDF^{\CC}(\cutpaarpartbb)\neq \emptyset$, then $\vierpartwwbb\in\CC$ or $\paarpartww\otimes\paarpartbb\in\CC$. 
\item[(d)] If $\singletonw\otimes\singletonb\notin\CC$ and if $\CC$ contains a partition $p_1\otimes p_2$ such that $c(p_1)\neq 0$, then $\vierpartwwbb\in\CC$ or $\paarpartww\otimes\paarpartbb\in\CC$. 
\end{itemize}
\end{lem}
\begin{proof}
(a) Let $q\in\{\cutpaarpartww,\cutpaarpartwb,\cutpaarpartbw,\cutpaarpartbb\}$ and let $s\in K^{\CC}(q)$.
Let $p=p_1\otimes p_2$ be a partition in $NDF^{\CC}(q)$ such that $c(p_1)=s$. 
By Lemma \ref{LemPartRole}(b), we may assume that $p_1$ consists only of singletons and composing $p$ with $\baarpartwb$ and $\baarpartbw$ (Remark \ref{RemComp}), we may assume that all points of $p_1$ have the same color, thus 
 $p_1=\singletonw^{\otimes s}$ (with the convention of Definition \ref{DefNotationNegativ}). We proceed in the same way for the points between the first and the last point of $p_2$ and we infer that  $\singletonw^{\otimes s} a^{\epsilon_1} \singletonw^{\otimes\alpha} a^{\epsilon_2}$ is in $\CC$ for some $\alpha\in\Z$. Here $a^{\epsilon_1}$ and $a^{\epsilon_2}$  form the pair block $q$. Now, disconnecting these two points using the partitions  $\idpartsingletonww$ and $\idpartsingletonbb$ again (Lemma \ref{LemPartRole}), we deduce $\singletonw^{\otimes s} \otimes \singletonw^{\otimes\alpha + c(q)}\in\CC$. Using verticolor reflection and Lemma \ref{PropCategOperations}(d), this implies:
\[\singletonw^{\otimes s} a^{\epsilon_1} \singletonw^{\otimes\alpha} \singletonw^{\otimes -s} \singletonw^{\otimes-\alpha - c(q)} a^{\epsilon_2}\in\CC\]
Using the pair partitions, we obtain $\singletonw^{\otimes s} a^{\epsilon_1} \singletonw^{\otimes -s - c(q)} a^{\epsilon_2}\in\CC$.

Finally note that we have $\positioners\in\CC$ if and only if $\positionersinv\in\CC$, using verticolor reflection and rotation. This proves $K^{\CC}(\cutpaarpartwb)=K^{\CC}(\cutpaarpartbw)$.

(b) Let $p\in\CC(0,l)$ be a partition of the form $p=aaX$, where $aa$ are two points of the same color belonging to the same block. If this block is of size two, we have $\paarpartww\otimes\paarpartbb\in\CC$ by Lemma \ref{RemErasePoints}. Otherwise, we consider $p\otimes \tilde p\in\CC$ and using the pair partition, we infer $\vierpartwwbb\in\CC$.

(c) If $\CC$ contains a partition $p_1\otimes p_2$ such that $c(p_1)> 0$, we use the pair partition to erase all black points in $p_1$. We obtain a partition $p_1'\otimes p_2\in\CC$ such that all points of $p_1'$ are white. If $\singletonw\otimes\singletonb\notin\CC$, all blocks in $p_1'$ have size at least two (Lemma \ref{LemCases}). Since $p_1'$ is noncrossing, we find two neighbouring points belonging to the same block. Using (b), we are done; likewise for $c(p_1)<0$.

(d) If $NDF^{\CC}(\cutpaarpartbb)\neq\emptyset$, then there is a $p=p_1\otimes p_2\in\CC$ in nest decomposed form such that the first and the last point of $p_2$ are black. If $c(p_1)=0$, we erase $p_1$ using the pair partitions and we obtain $p_2\in\CC$. By rotation, $p_2$ is of the form as in (b).  If $c(p_1)\neq 0$ we use (c).
\end{proof}

We finally prove that our local parameters are given only by $d(\CC)$ and $K^{\CC}(\cutpaarpartbb)$.

\begin{lem}\label{Lem213NEU}
Let $\CC\subset NC^{\twocol}$ and $d=d(\CC)$.
\begin{itemize}
\item[(a)] We always  have:
\[K^{\CC}(\cutpaarpartwb)=K^{\CC}(\cutpaarpartbw)=d\Z \quad\textnormal{and}\quad K^{\CC}(\cutpaarpartww)= -K^{\CC}(\cutpaarpartbb)\]
\item[(b)] In case  $\paarpartww\otimes\paarpartbb$ or $\vierpartwwbb$ are in $\CC$, we have:
\[ K^{\CC}(\cutpaarpartwb)=K^{\CC}(\cutpaarpartbw)=K^{\CC}(\cutpaarpartww)= K^{\CC}(\cutpaarpartbb)=d\Z\]
\item[(c)] Let $\paarpartww\otimes\paarpartbb\notin\CC$ and $\vierpartwwbb\notin\CC$. If $\positionertpluseins\in\CC$ for some $t\in\N_0$, we denote by $r$ the smallest such number $t$. We then have $r\neq 1$ and:
\[K^{\CC}(\cutpaarpartwb)=K^{\CC}(\cutpaarpartbw)=d\Z \quad\textnormal{and}\quad K^{\CC}(\cutpaarpartbb)= -K^{\CC}(\cutpaarpartww)=d\Z+(r+1)\]
\item[(d)] If none of the partitions  $\paarpartww\otimes\paarpartbb$,  $\vierpartwwbb$ and $\positionertpluseins$  is in $\CC$ (for any $t\in\N_0$), then
\[K^{\CC}(\cutpaarpartwb)=K^{\CC}(\cutpaarpartbw)=d\Z \quad\textnormal{and}\quad K^{\CC}(\cutpaarpartbb)= K^{\CC}(\cutpaarpartww)=\emptyset\]
\end{itemize}
\end{lem}
\begin{proof}
(a) If $p_1\otimes p_2\in\CC$ is a partition in nest decomposed form such that the first and the last point of $p_2$ have the same color, then $\tilde p_1\otimes \tilde p_2$ is in $\CC$ proving  $K^{\CC}(\cutpaarpartww)=-K^{\CC}(\cutpaarpartbb)$. The identity $K^{\CC}(\cutpaarpartwb)=K^{\CC}(\cutpaarpartbw)=d\Z$ follows from (b), (c) and (d).

(b) Let $p_1\otimes p_2$ be in nest decomposed form. If $\vierpartwwbb\in\CC$, then the partitions $\vierpartwrotwwb$, $\vierpartwrotbww$, $\vierpartbrotwbb$ and $\vierpartbrotbbw$ are all in $\CC$. Composing $p_1\otimes p_2$ with them (Remark \ref{RemComp}), we can change the first and the last point of $p_2$ arbitrarily. This proves $ K^{\CC}(\cutpaarpartwb)=K^{\CC}(\cutpaarpartbw)=K^{\CC}(\cutpaarpartww)= K^{\CC}(\cutpaarpartbb)$ which is $d\Z$ by Proposition \ref{LemDDivisor}. In the case $\paarpartww\otimes\paarpartbb\in\CC$, use $\paarpartww\otimes\idpartwb$, $\idpartwb\otimes\paarpartww$, $\idpartbw\otimes\paarpartbb$ and $\paarpartbb\otimes\idpartbw$.

(c) From $\singletonw\otimes\positionert\in\CC$, we deduce $\singletonw\otimes\singletonb\in\CC$ using Lemma \ref{PropCategOperations}(c). By Lemma \ref{LemBsInC}, this shows $K^{\CC}(\cutpaarpartwb)= K^{\CC}(\cutpaarpartbw)$. Next, we prove $K^{\CC}(\cutpaarpartbb)=d\Z+(r+1)$. First note that $r\neq 1$ since $\paarpartww\otimes\paarpartbb\notin\CC$. Moreover, $\singletonw^{\otimes r-1}\otimes\singletonb^{\otimes r-1}\in\CC$, since $\singletonw\otimes\singletonb\in\CC$. Let $m\in\Z$. Then $dm\in K^{\CC}(\cutpaarpartwb)$ which by Lemma \ref{LemBsInC} yields:
\setlength{\unitlength}{0.5cm}
\begin{center}
 \begin{picture}(8,1.5)
     \put(-2,0){$p:=$}
     \put(1.6,0.2){\line(0,1){0.9}}
     \put(6.4,0.2){\line(0,1){0.9}}
     \put(1.6,1.1){\line(1,0){4.8}}
     \put(0.05,-0.05){$^\uparrow$}
     \put(0.2,0.4){\tiny$^{\otimes dm}$}
     \put(5.05,-0.05){$^\uparrow$}
     \put(5.2,0.4){\tiny$^{\otimes dm}$}
     \put(2.2,0.4){\tiny$^{\otimes r-1}$}
     \put(3.7,0.4){\tiny$^{\otimes r-1}$}
     \put(0,-0.2){$\circ$}
     \put(1.4,-0.2){$\circ$}
     \put(5,-0.2){$\bullet$}
     \put(6.2,-0.2){$\bullet$}
     \put(2.05,-0.05){$^\uparrow$}
     \put(3.55,-0.05){$^\uparrow$}
     \put(2,-0.2){$\circ$}
     \put(3.5,-0.2){$\bullet$}
     \put(7,0){$\in\CC$}
 \end{picture}
\end{center}
Furthermore, the following partition is in $\CC$ since it is a rotated version of $\positionerrpluseins$:
\setlength{\unitlength}{0.5cm}
\begin{center}
 \begin{picture}(5,2)
     \put(0,0){$q:=$}
     \put(3.6,0.2){\line(0,1){1.2}}
     \put(2.05,-0.05){$^\uparrow$}
     \put(2.2,0.4){\tiny$^{\otimes r+1}$}
     \put(4.45,0.8){$^\downarrow$}
     \put(4.6,0.6){\tiny$^{\otimes r-1}$}
     \put(2,-0.2){$\circ$}
     \put(3.4,-0.2){$\bullet$}
     \put(3.4,1.4){$\circ$}
     \put(4.4,1.4){$\circ$}
 \end{picture}
\end{center}
Now, composing $p$ with $q$ we may shift $r-1$ white singletons (or rather one black singleton, if $r=0$) from the inside of the pair of $p$ to the outside, by which their number increases to $r+1$ white singletons. Furthermore, the color of the first point of the pair changes from white to black. We conclude that $dm+(r+1)$ is in $K^{\CC}(\cutpaarpartbb)$, which proves $d\Z+r+1\subset K^{\CC}(\cutpaarpartbb)$.

Conversely, let $s\in K^{\CC}(\cutpaarpartbb)$. Then $p:=\positionersminuszwei\in\CC$ by Lemma \ref{LemBsInC}. By verticolor reflection, we also have $q:=\positionerrminuseins\in\CC$. Composing $p\otimes q$ with $\baarpartbw$ yields $s-(r+1)\in K^{\CC}(\cutpaarpartwb)=d\Z$.

(d) If there was an element $t\in K^{\CC}(\cutpaarpartbb)$, then $\positionerspluseins\in\CC$ with $s:=t-1$, by Lemma \ref{LemBsInC}; a contradiction. Hence $K^{\CC}(\cutpaarpartbb)=\emptyset$ which implies $K^{\CC}(\cutpaarpartww)=\emptyset$ and $NDF^{\CC}(\cutpaarpartbb)=NDF^{\CC}(\cutpaarpartww)=\emptyset$. In order to prove $K^{\CC}(\cutpaarpartwb)=K^{\CC}(\cutpaarpartbw)$,  the case $\singletonw\otimes\singletonb\in\CC$ is covered by Lemma \ref{LemBsInC}. Now, assume $\singletonw\otimes\singletonb\notin\CC$. 
Let $p_1\otimes p_2\in\CC$ be a partition in nest decomposed form such that the first and the last point of $p_2$ have inverse colors.  Assume $c(p_1)>0$. Using the pair partitions, we may erase all black points in $p_1$ and we obtain a partition $p_1'\otimes p_2$ such that all points in $p_1'$ are white. Since all blocks in $p_1'$ have size at least two by Lemma \ref{LemCases}, we obtain a partition in $NDF^{\CC}(\cutpaarpartww)$ up to rotation which is a contradiction. Likewise we show that $c(p_1)$ cannot be strictly negative. We conclude that $c(p_1)$ is zero and hence $K^{\CC}(\cutpaarpartwb)=\{0\}$ and $K^{\CC}(\cutpaarpartwb)=\{0\}$. Thus $K^{\CC}(\cutpaarpartwb)=K^{\CC}(\cutpaarpartwb)$.
\end{proof}

In the globally colorized case, the local parameters are almost trivial as can be seen in the next lemma.

\begin{lem}\label{LemRDGlobal}\label{LemRDCaseO}
If $\CC\subset NC^{\twocol}$ is globally colorized, then $d(\CC)\in\{1,2\}$ with $d(\CC)=1$ if and only if $\positionerwwbb\in\CC$. Moreover:
\[ K^{\CC}(\cutpaarpartwb)=K^{\CC}(\cutpaarpartbw)=K^{\CC}(\cutpaarpartww)= K^{\CC}(\cutpaarpartbb)=d\Z\]
\end{lem}
\begin{proof}
We have $\paarpartww\otimes\nestpaarpartbwbb\in\CC$ from which we infer $\paarpartww\otimes\nestpaarpartwbbb\in\CC$ by permutation of colors. Thus $2\in K^{\CC}(\cutpaarpartwb)$ and hence $d(\CC)\in\{1,2\}$. If $\positionerwwbb\in\CC$, then $d(\CC)=1$ by definition. Conversely, let $d(\CC)=1$. We thus find a partition $\singletonw\otimes p_2$ in $\CC$. By Lemma \ref{RemErasePoints}, we deduce $\singletonw\otimes\singletonb\in\CC$, which by Lemma \ref{LemBsInC} implies $\positionerwwbb\in\CC$.

The second assertion follows from Lemma \ref{Lem213NEU}(b).
\end{proof}

\begin{lem}\label{LemKD}
Let $\CC\subset NC^{\twocol}$ and put $d=d(\CC)$, $k=k(\CC)$.
\begin{itemize}
\item[(a)] Let $\singletonw\otimes\singletonb\in\CC$. Then $\singletonw^{\otimes d}\otimes p_2\in\CC$ for some partition $p_2$ and hence also $\singletonw^{\otimes d}\otimes\singletonb^{\otimes d}\in\CC$. Moreover, $\singletonw^{\otimes k}\in\CC$ and $\positionerd\in\CC$.
\item[(b)] Let $\vierpartwwbb\in\CC$. Then $b_d\otimes p_2\in\CC$ for some partition $p_2$ and hence also $b_d\otimes \tilde b_d\in\CC$. moreover $b_k\in\CC$.
\item[(c)] If $d(\CC)\neq 0$ or $k(\CC)\neq 0$ or $K^{\CC}(\cutpaarpartbb)\neq\emptyset$ then one of the following partitions is in $\CC$: $\vierpartwwbb$ or $\paarpartww\otimes\paarpartbb$ or $\singletonw\otimes\singletonb$.
\end{itemize}
\end{lem}
\begin{proof}
(a) Since $d\in K^{\CC}(\cutpaarpartwb)$, there is a partition $p_1\otimes p_2$ in nest decomposed form such that $c(p_1)=d$. By Lemma \ref{LemPartRole}(b), we may assume that $p_1$ consists only of singletons, and using $\baarpartwb$ and $\baarpartbw$ (Remark \ref{RemComp}), we may assume that they are all white; hence $p_1=\singletonw^{\otimes d}$. By Lemma \ref{PropCategOperations}(c), we conclude $\singletonw^{\otimes d}\otimes\singletonb^{\otimes d}\in\CC$. As for proving $\singletonw^{\otimes k}\in\CC$, we proceed like above with $p_2=\emptyset$. Finally $\positionerd\in\CC$ follows from Lemma \ref{LemBsInC}.

(b) We consider $p_1\otimes p_2\in\CC$ with $c(p_1)=d$ as in (a), but by Lemma \ref{LemPartRole}(c) we may assume that $p_1$ consists only of one block. The rest of the proof is similar to (a).

(c) If $d(\CC)\neq 0$ or $k(\CC)\neq 0$, we find a partition $p_1\otimes p_2\in\CC$ with $c(p_1)\neq 0$. Using $\baarpartwb$ and $\baarpartbw$ (Remark \ref{RemComp}), we may assume that all points in $p_1$ are of the same color. If  $p_1$ contains a block of length one, then $\singletonw\otimes\singletonb\in \CC$ by Lemma \ref{LemCases}. Otherwise, we have $\vierpartwwbb\in\CC$ or $\paarpartww\otimes\paarpartbb\in\CC$ by Lemma \ref{LemBsInC}(b). 

Furthermore, if $K^{\CC}(\cutpaarpartbb)\neq\emptyset$, then   $\paarpartww\otimes\paarpartbb$,  $\vierpartwwbb$ or $\positionertpluseins$  is in $\CC$, for some $t\in\N_0$, by Lemma \ref{Lem213NEU}. Clearly, $\positionertpluseins\in\CC$ implies $\singletonw\otimes\singletonb\in\CC$ by Lemma \ref{PropCategOperations}.
\end{proof}

\subsection{Summary of the strategy for the classification}\label{SectStrategy}

We now have all tools at hand for the classification of categories $\CC\subset NC^{\twocol}$ of noncrossing partitions. 
The general strategy is as follows. 
\begin{itemize}
 \item We study the cases $\OOO, \HHH, \SSS$ and $\BBB$ (see Definition \ref{DefCases}) step by step subdividing them again into  the local and the global colorization (see Definition \ref{DefGlobalColor}) respectively. 
 \item In each of these cases we first \textbf{determine} all possible global and local \textbf{parameters} $k(\CC)$, $d(\CC)$ and $K^\CC(\cutpaarpartbb)$.
 \item We then \textbf{find characteristic sample partitions} which somehow represent these parameters.
 \item Next, we isolate sets of partitions $M$ depending on the possible values of the parameters and we prove $M\subset\langle p_1,\ldots,p_n\rangle$, where $p_1,\ldots,p_n$ are the sample partitions.
 \item Finally, we prove that \textbf{these are all possible categories} in the considered case. To do so, if $\CC$ has parameters $k$, $d$ and $K^{\CC}(\cutpaarpartbb)$, we have $\langle p_1,\ldots, p_n\rangle\subset \CC$. On the other hand, $\CC\subset M$, which proves $\CC=\langle p_1,\ldots, p_n\rangle=M$.
\end{itemize}

\section{Case $\OOO$}\label{SectCaseO}

Let us first consider the case $\OOO$, i.e. the case of categories $\CC\subset NC^{\twocol}$ of noncrossing partitions such that $\singletonw\otimes\singletonb\notin\CC$ and $\vierpartwbwb\notin\CC$. By Lemma \ref{LemCases}, $\CC$ is a subset of the set $NC_2^{\twocol}$ of all noncrossing pair partitions.

\subsection{Determining the parameters} 

\begin{prop}\label{ParametersO}
 Let $\CC\subset NC^{\twocol}$ be a category of noncrossing partitions in case $\OOO$.
\begin{itemize}
 \item[(a)] If $\CC$ is globally colorized, then $d(\CC)=2$ and $k(\CC)\in 2\N_0$.
 \item[(b)] If $\CC$ is locally colorized, then $d(\CC)=k(\CC)=0$ and we have:
 \[K^\CC(\cutpaarpartwb)=K^\CC(\cutpaarpartbw)=\{0\},\qquad K^\CC(\cutpaarpartww)=K^\CC(\cutpaarpartbb)=\emptyset\]
\end{itemize}
\end{prop}
\begin{proof}
(a) By Lemma \ref{LemRDGlobal}, $d(\CC)=2$ and hence $k(\CC)\in 2\N_0$, using Proposition \ref{LemDDivisor}.

(b) By Lemma \ref{LemBsInC}(d), we have $d(\CC)=k(\CC)=0$. Lemma \ref{Lem213NEU}(d) completes the proof.
\end{proof}

\subsection{Finding partitions realizing the parameters}\label{SubsectRealPartOOO}

\begin{lem}\label{ParametersCaseO}
 Let $\CC\subset NC^{\twocol}$ be a globally colorized category of noncrossing partitions in case $\OOO$ such that $k=k(\CC)\neq 0$. Then $\paarpartww^{\otimes \frac{k}{2}}\in\CC$.
\end{lem}
\begin{proof}
We find a partition $p\in\CC$ having no upper points such that $c(p)=k$, i.e. there are $x+k$ white points and $x$ black points in $p$. Now, the partition $\paarpartwb^{\otimes \frac{k}{2}}\otimes p$ is in $\CC$. By Lemma \ref{LemPartRole}, we may permute the colors of this partition, and we infer that $\paarpartww^{\otimes \frac{k}{2}}\otimes p'$ is in $\CC$ for some partition $p'$ with $c(p')=0$. Using the pair partitions $\paarpartwb$ and $\paarpartbw$ to erase $p'$, we infer that $\paarpartww^{\otimes \frac{k}{2}}\in\CC$.
\end{proof}

\subsection{Description of natural categories in case $\OOO$}
The preceding subsection shows that there are natural categories $\langle\paarpartww^{\otimes \frac{k}{2}},\paarpartww\otimes\paarpartbb\rangle$ in case $\OOO$. Let us describe these categories and also the ``empty category'' $\langle\emptyset\rangle=\langle\paarpartwb,\paarpartbw\rangle$. (Recall that the partitions $\paarpartwb$ and $\paarpartbw$ are always contained in a category and that we omit to write them down explicitely as generators.)

\begin{prop}\label{DescriptionCaseOGlobal}
We have the following natural categories of partitions in case $\OOO$.
\begin{itemize}
 \item[(a)] The category $\categg{\OOO}{\loc}:=\langle\emptyset\rangle$ consists of all noncrossing pair partitions such that each block connects a white point with a black point, when the partition is rotated such that it has no upper points.
 \item[(b)] Let $k\in2\N_0$. Then $\categ{\OOO}{\glob}{k}:=\langle\paarpartww^{\otimes \frac{k}{2}},\paarpartww\otimes\paarpartbb\rangle$ coincides with $\{p\in NC_2^{\twocol}\;|\; c(p)\in k\Z\}$.
Here, $\paarpartww^{\otimes \frac{k}{2}}=\emptyset$ if $k=0$. 
\end{itemize}
In particular, all these categories are pairwise different.
\end{prop}
\begin{proof}
(a) We may construct all partitions $p$ from the assertion using $\paarpartwb$, $\paarpartbw$ and the category operations due to a simple inductive argument: Assume that $p$ has $m+1$ blocks. Since $p$ is noncrossing, it contains at least one block $\paarpartwb$ or $\paarpartbw$ on two consecutive points. Removing it yields a partition which is in $\langle\emptyset\rangle$ by induction hypothesis. Putting it back (Lemma \ref{PropCategOperations}(d)), we infer $p\in\langle\emptyset\rangle$. Conversely, the set of all noncrossing pair partitions with the block rule of the assertion forms a category of partitions, hence containing $\langle\emptyset\rangle$.

(b) Let $p\in NC_2^{\twocol}(0,l)$ be a partition with no upper points such that $c(p)=km\geq 0$, for some $m\in\N_0$. Let $p'\in NC_2^{\twocol}(0,l)$ be the partition obtained from $p$ by replacing each unicolored block $\paarpartww$ or $\paarpartbb$ by $\paarpartwb$. Then, $p'$ is a partition in $\langle\emptyset\rangle\subset\langle\paarpartww^{\otimes \frac{k}{2}},\paarpartww\otimes\paarpartbb\rangle$ by (a) and $c(p')=0$. This implies that $p'\otimes \paarpartww^{\otimes \frac{km}{2}}$ is in $\langle\paarpartww^{\otimes \frac{k}{2}},\paarpartww\otimes\paarpartbb\rangle$, too, with $c(p'\otimes \paarpartww^{\otimes \frac{km}{2}})=c(p)$ by Lemma \ref{LemC}. Hence, permutation of colors yields that $p\otimes \paarpartwb^{\otimes \frac{km}{2}}$ is in  $\langle\paarpartww^{\otimes \frac{k}{2}},\paarpartww\otimes\paarpartbb\rangle$. Using Lemma \ref{PropCategOperations}(b), we infer  $p\in\langle\paarpartww^{\otimes \frac{k}{2}},\paarpartww\otimes\paarpartbb\rangle$. 

Conversely, the set $\{p\in NC_2^{\twocol}\;|\; c(p)\in k\Z\}$ is a category of partitions due to Lemma \ref{LemC} containing $\paarpartww^{\otimes \frac{k}{2}}$ and $\paarpartww\otimes\paarpartbb$.  
\end{proof}

\subsection{Classification in the case $\OOO$}
We are now ready to prove our first classification theorem.

\begin{thm}\label{ThmCaseO}
Let $\CC\subset NC^{\twocol}$ be a category of noncrossing partitions in case $\OOO$. Then $\CC$ coincides with one of the following categories.
\begin{itemize}
\item[(a)] If $\CC$ is globally colorized, then $\CC= \categ{\OOO}{\glob}{k}=\langle\paarpartww^{\otimes \frac{k}{2}},\paarpartww\otimes\paarpartbb\rangle$ for $k=k(\CC)\in 2\N_0$.
\item[(b)] If $\CC$ is locally colorized, then $\CC=\categg{\OOO}{\loc}=\langle\emptyset\rangle$.
\end{itemize}
\end{thm}
\begin{proof}
(a) By Propositions \ref{LemKZ} and  \ref{DescriptionCaseOGlobal}, we have that $\CC$ is contained in $\langle\paarpartww^{\otimes \frac{k}{2}},\paarpartww\otimes\paarpartbb\rangle$. Conversely, $\paarpartww\otimes\paarpartbb\in\CC$ by Definition \ref{DefGlobalColor} and $\paarpartww^{\otimes \frac{k}{2}}\in\CC$ by Lemma \ref{ParametersCaseO}.

(b) Let $\CC$ be locally colorized and let $p\in\CC$ be a partition with no upper points. Then, each block of $p$ connects a white point to a black point, since $NDF^{\CC}(\cutpaarpartww)=NDF^{\CC}(\cutpaarpartbb)=\emptyset$ by Proposition \ref{ParametersO}. By Proposition \ref{DescriptionCaseOGlobal},  $\CC$ is contained in $\langle\emptyset\rangle$, hence they coincide.
\end{proof}

\begin{rem}
For $k=2$, the category $\langle\paarpartww^{\otimes \frac{k}{2}},\paarpartww\otimes\paarpartbb\rangle$ coincides with the non-colored category of partitions $\langle\paarpart\rangle$ in the sense of Proposition \ref{PropOneColored}.
\end{rem}

\section{Case $\HHH$}\label{SectCaseH}

We now turn to the case $\HHH$, i.e. to categories $\CC\subset NC^{\twocol}$ such that $\singletonw\otimes\singletonb\notin\CC$ but $\vierpartwbwb\in\CC$. By Lemma \ref{LemCases},  no blocks of size one occur in any partition considered in this section. Recall, that due to Lemma \ref{LemPartRole}(d), we may connect neighbouring blocks of partitions in $\CC$, if the blocks meet at two points with inverse colors. 

\subsection{Determining the parameters}

\begin{prop}\label{LemRDCaseH}
Let $\CC\subset NC^{\twocol}$ be a category of noncrossing partitions in case $\HHH$. 
\begin{itemize}
 \item[(a)] If $\CC$ is globally colorized, then $d(\CC)=2$ and $k(\CC)\in 2\N_0$.
 \item[(b)] If $\CC$ is locally colorized, then 
 \begin{itemize}
  \item[(i)] either $k(\CC)=d(\CC)=0$ and we have:
    \[K^\CC(\cutpaarpartwb)=K^\CC(\cutpaarpartbw)=0\qquad\textnormal{and}\qquad K^\CC(\cutpaarpartww)=K^\CC(\cutpaarpartbb)=\emptyset\]
  \item[(ii)] or $k(\CC),d(\CC)\in \N_0\backslash \{1,2\}$ and we have:
  \[K^\CC(\cutpaarpartwb)=K^\CC(\cutpaarpartbw)=K^\CC(\cutpaarpartww)=K^\CC(\cutpaarpartbb)=d\Z\]
 \end{itemize}
  Moreover, we have $K^{\CC}(\cutpaarpartbb)\neq\emptyset$ if and only if $\vierpartwwbb\in\CC$. 
\end{itemize}
\end{prop}
\begin{proof}
(a) Use Lemma \ref{Lem213NEU} and Proposition \ref{LemDDivisor}.

(b) If $\vierpartwwbb\notin\CC$, use Lemma \ref{LemBsInC}(d) and Lemma \ref{Lem213NEU}(d) to prove (i). Otherwise, Lemma \ref{Lem213NEU}(b) yields (ii) in combination with Lemma \ref{LemKD}(b) in order to exclude $k(\CC),d(\CC)\in \{1,2\}$. 
\end{proof}

\subsection{Finding partitions realizing the parameters}

\begin{lem}\label{LemSamplePartitionsH}
Let $\CC\subset NC^{\twocol}$ be a category of noncrossing partitions in case $\HHH$. 
\begin{itemize}
 \item[(a)] If $k=k(\CC)\neq 0$, then $b_k\in\CC$.
 \item[(b)] If $d=d(\CC)\neq 0$, then $b_d\otimes \tilde{b_d}\in\CC$.
\end{itemize}
\end{lem}
\begin{proof}
If $k\neq 0$ or $d\neq 0$, we have $\vierpartwwbb\in\CC$ by Proposition \ref{LemRDCaseH}, which by Lemma \ref{LemKD}(b) yields the result.
\end{proof}

\subsection{Description of natural categories}

Motivated by Lemma \ref{LemSamplePartitionsH}, we want to describe the categories $\langle b_k, b_d\otimes \tilde{b_d},\vierpartwbwb\rangle$. Note that for $k\geq 2$ or $d\geq 2$, we may always construct the partition $\vierpartwwbb$ inside the category (Lemma \ref{LemBsInC}). Due to Proposition \ref{LemRDCaseH}, this is a natural generator indeed, so we add it in the following lemma also for the cases $k=d=0$ and treat the case $\langle\vierpartwbwb\rangle$ separately. 

\begin{prop}\label{LemBaseCategoryH}\label{DescriptionCaseHLocal}
 We have the following natural categories in case $\HHH$.
\begin{itemize}
 \item[(a)] The category $\categg{\HHH'}{\loc}:=\langle\vierpartwbwb\rangle$ consists of all noncrossing partitions such that each block is of even length connecting white and black points in an alternating way, when the partition is rotated such that it has no upper points.
 \item[(b)] Let $k,d\in\N_0\backslash\{1\}$ be such that $d$ is a divisor of $k$, if $k\neq 0$, and denote $b_0:=\emptyset$. Denote by $\categ{\HHH}{\loc}{k,d}$ the set of 
 all partitions $p\in NC^{\twocol}$ such that
 \begin{itemize}
 \item[(i)] all blocks have length at least two,
 \item[(ii)] $c(p)\in k\Z$,
 \item[(iii)] if $p_1\otimes p_2$ is any rotated version of $p$ in nest decomposed form, then $c(p_1)\in d\Z$.
 \end{itemize}
We have $\categ{\HHH}{\loc}{k,d}\subset \langle b_k, b_d\otimes \tilde{b_d},\vierpartwwbb,\vierpartwbwb\rangle$.
\end{itemize}
\end{prop}
\begin{proof}
(a) By Lemma \ref{LemPartRole}(d) we may deduce $\sechspartwbwbwb\in\langle\vierpartwbwb\rangle$ from 
 $\vierpartwbwb\otimes\paarpartwb\in\langle\vierpartwbwb\rangle$. Iteratively, we may construct all one block partitions of even length connecting white and black points in an alternating way, and using the category operations, we may construct all partitions of the assertion. Conversely, the set of all noncrossing partitions with the above block rule forms a category of partitions containing $\vierpartwbwb$.

(b) Denote the category $\langle b_k, b_d\otimes \tilde{b_d},\vierpartwwbb,\vierpartwbwb\rangle$ by $\DD$. Let $p\in \categ{\HHH}{\loc}{k,d}$. Since $\DD$ is closed under rotation, we may assume that $p$ has no upper points.  We prove $p\in \DD$ by induction on the number $m$ of blocks of $p$.  
For $m=1$, note that $\iota_l\otimes b_k^{\otimes t}\in \mathcal D$ for all $t\geq 0$ by (a), where $\iota_l\in P^{\twocol}(0,2l)$ consists of a single block on $2l$ points with alternating colors. Using Lemma \ref{LemPartRole}(c), we may connect all blocks of $\iota_l\otimes b_k^{\otimes t}$ and permute its points. We infer that all one block partitions  with $c(p)\in k\Z$ are in $\mathcal D$.

Let $m>1$. By rotation and since $p$ is noncrossing, $p$ is of the form $p=p_1\otimes p_2$ such that $p_2$ consists only of one block. At this point, one would like to apply the induction hypothesis on $p_1$, but we do not have $c(p_1)\in k\Z$ in general. Hence, we use a modified version $p_1'$ instead, constructed as follows.

  Since we are in case $\HHH$, the partition $p_2$ has length at least two and thus $p$ is in nest decomposed form.
Thus $c(p_1)\in d\Z$ and hence also $c(p_2)=c(p)-c(p_1)\in d\Z$. Assume $c(p_1)=ds$ with some $s\geq 0$ (use verticolor reflection in case $s<0$). Let $p_1'$ be the partition obtained from $p_1\otimes \tilde{b_d}^{\otimes s}$ by connecting all points of $\tilde{b_d}^{\otimes s}$ to the last point of $p_1$. Then $p_1'$ is a partition with $m-1$ blocks and $c(p_1')=c(p_1)-ds=0$. 
Furthermore, any nest decomposed form $q_1'\otimes q_2'$ of rotations of $p_1'$ satisfies $c(q_1')\in d\Z$. The proof of this statement is explained by the following scheme:

\setlength{\unitlength}{0.5cm}
\newsavebox{\boxpeins} 
   \savebox{\boxpeins}
   { \begin{picture}(5,1.2)
     \put(0,0){\line(0,1){1.2}}
     \put(5,0){\line(0,1){1.2}}
     \put(0,0){\line(1,0){5}}
     \put(0,1.2){\line(1,0){5}}    
     \put(2.2,0.3){$p_1$}
     \end{picture}}
\newsavebox{\boxpzwei} 
   \savebox{\boxpzwei}
   { \begin{picture}(2,1.2)
     \put(0,0){\line(0,1){1.2}}
     \put(2,0){\line(0,1){1.2}}
     \put(0,0){\line(1,0){2}}
     \put(0,1.2){\line(1,0){2}}    
     \put(0.7,0.3){$p_2$}
     \end{picture}}     
\newsavebox{\boxbd} 
   \savebox{\boxbd}
   { \begin{picture}(2.5,1.2)
     \put(0,0){\line(0,1){1.2}}
     \put(2.5,0){\line(0,1){1.2}}
     \put(0,0){\line(1,0){2.5}}
     \put(0,1.2){\line(1,0){2.5}}    
     \put(0.95,0.2){$\tilde b_d$}
     \end{picture}}     
\newsavebox{\boxconnect} 
   \savebox{\boxconnect}
   { \begin{picture}(0.7,0.3)
     \put(0,0){\line(0,1){0.3}}
     \put(0.7,0){\line(0,1){0.3}}
     \put(0,0){\line(1,0){0.7}}
     \end{picture}}     
\newsavebox{\boxqeins} 
   \savebox{\boxqeins}
   { \begin{picture}(1.5,1.2)
     \put(0,0){\line(0,1){1.2}}
     \put(1.5,0){\line(0,1){1.2}}
     \put(0,0){\line(1,0){1.5}}
     \put(0,1.2){\line(1,0){1.5}}    
     \put(0.3,0.3){$q_1'$}
     \end{picture}}     
\newsavebox{\boxqzweia} 
   \savebox{\boxqzweia}
   { \begin{picture}(1,1.2)
     \put(0,0){\line(0,1){1.2}}
     \put(1,0){\line(0,1){1.2}}
     \put(0,0){\line(1,0){1}}
     \put(0,1.2){\line(1,0){1}}    
     \put(0.1,0.3){$q_2'$}
     \end{picture}}  
\newsavebox{\boxqzweib} 
   \savebox{\boxqzweib}
   { \begin{picture}(14,1.2)
     \put(0,0){\line(0,1){1.2}}
     \put(14,0){\line(0,1){1.2}}
     \put(0,0){\line(1,0){14}}
     \put(0,1.2){\line(1,0){14}}    
     \put(7,0.3){$q_2'$}
     \end{picture}} 
\newsavebox{\boxqzweiaa} 
   \savebox{\boxqzweiaa}
   { \begin{picture}(1,1.2)
     \put(0,0){\line(0,1){1.2}}
     \put(1,0){\line(0,1){1.2}}
     \put(0,0){\line(1,0){1}}
     \put(0,1.2){\line(1,0){1}}    
     \put(0.1,0.3){$q_2''$}
     \end{picture}}  
\newsavebox{\boxqzweibb} 
   \savebox{\boxqzweibb}
   { \begin{picture}(2,1.2)
     \put(0,0){\line(0,1){1.2}}
     \put(2,0){\line(0,1){1.2}}
     \put(0,0){\line(1,0){2}}
     \put(0,1.2){\line(1,0){2}}    
     \put(0.7,0.3){$q_2''$}
     \end{picture}} 
\newsavebox{\boxconnectlang} 
   \savebox{\boxconnectlang}
   { \begin{picture}(2.2,0.3)
     \put(0,0){\line(0,1){0.3}}
     \put(2.2,0){\line(0,1){0.3}}
     \put(0,0){\line(1,0){2.2}}
     \end{picture}}     
\begin{center}
 \begin{picture}(26,1.2)
     \put(0,0){$p=$}
     \put(2,-0.3){\usebox{\boxpeins}}
     \put(8,0){$\otimes$}
     \put(9,-0.3){\usebox{\boxpzwei}}
     \put(20,0){(up to rotation; NDF)}
 \end{picture}
\end{center}
\begin{center}
 \begin{picture}(26,2)
     \put(0,0.8){$p_1'=$}
     \put(2,0.5){\usebox{\boxpeins}}
     \put(7.5,0.5){\usebox{\boxbd}}
     \put(10.5,0.5){\usebox{\boxbd}}
     \put(14.5,0.8){$\ldots$}
     \put(16.5,0.5){\usebox{\boxbd}}
     \put(6.9,0){\usebox{\boxconnect}}
     \put(9.9,0){\usebox{\boxconnect}}
     \put(12.9,0){\usebox{\boxconnect}}
     \put(15.9,0){\usebox{\boxconnect}}
 \end{picture}
\end{center}
\begin{center}
 \begin{picture}(26,1.5)
     \put(2,0.5){\usebox{\boxqzweia}}   
     \put(3.2,0.5){\usebox{\boxqeins}}
     \put(5,0.5){\usebox{\boxqzweib}} 
     \put(2.9,0){\usebox{\boxconnectlang}}    
     \put(20,0.5){(possibly $q_1'\longleftrightarrow q_2'$)}       
 \end{picture}
\end{center}
\begin{center}
 \begin{picture}(26,1.5)
     \put(2,0.5){\usebox{\boxqzweiaa}}   
     \put(3.2,0.5){\usebox{\boxqeins}}
     \put(5,0.5){\usebox{\boxqzweibb}} 
     \put(2.9,0){\usebox{\boxconnectlang}}    
     \put(8,0.8){$\otimes$}
     \put(9,0.5){\usebox{\boxpzwei}}     
     \put(20,0.5){($q_2''=q_2'$ without $\tilde b_d^{\otimes s}$)}       
 \end{picture}
\end{center}
\begin{center}
 \begin{picture}(26,1)
     \put(0,0){$\Longrightarrow$ up to rotation a NDF of $p$, hence $c(q_1')\in d\Z$ and $c(q_2')=c(q_2'')-ds\in d\Z$}
 \end{picture}
\end{center}

We conclude that  $p_1'$ is in $\categ{\HHH}{\loc}{k,d}$ and by induction hypothesis, $p_1'\in\mathcal D$. Composing it with $r\otimes (\tilde{b_d}^*)^{\otimes s}\otimes (\tilde{b_d})^{\otimes s}$, where $r$ is a suitable tensor product of the identity partitions, yields $p_1\otimes \tilde{b_d}^{\otimes s}\in\mathcal D$. Similary $b_d^{\otimes s} \otimes p_2\in\mathcal D$, since the partition $p_2'$ obtained from connecting all blocks of $b_d^{\otimes s} \otimes p_2$ is a one block partition with $c(p_2')=c(p)$. We conclude that $p_1\otimes \tilde{b_d}^{\otimes s}\otimes b_d^{\otimes s} \otimes p_2\in\mathcal D$ and using Lemma \ref{PropCategOperations}(b), we obtain $p=p_1\otimes p_2\in\mathcal D$.
\end{proof}

\subsection{Classification in the case $\HHH$}

\begin{thm}\label{ThmCaseH}
Let $\CC\subset NC^{\twocol}$ be a category of noncrossing partitions in case $\HHH$. Then $\CC$ coincides with one of the following categories.
\begin{itemize}
\item[(i)] If $\CC$ is globally colorized, then $\CC= \categ{\HHH}{\glob}{k}=\langle b_k,\vierpartwbwb,\paarpartww\otimes\paarpartbb\rangle$ for $k=k(\CC)\in 2\N_0$. Here $\categ{\HHH}{\glob}{k}:=\categ{\HHH}{\loc}{k,2}$.
\item[(ii)] If $\CC$ is locally colorized, and $K^{\CC}(\cutpaarpartbb)=\emptyset$, then $\CC=\categg{\HHH'}{\loc}=\langle\vierpartwbwb\rangle$. If $K^{\CC}(\cutpaarpartbb)\neq\emptyset$, we have $\CC=\langle b_k,b_d\otimes\tilde b_d,\vierpartwwbb,\vierpartwbwb\rangle$ for $k=k(\CC)\in\N_0\backslash\{1,2\}$ and $d=d(\CC)\in\N_0\backslash\{1,2\}$. We use the notation $b_0=\emptyset$.
\end{itemize}
\end{thm}
\begin{proof}
(i) Using Lemma \ref{LemSamplePartitionsH}, we know  $\langle b_k,\vierpartwbwb,\paarpartww\otimes\paarpartbb\rangle\subset \CC$.
For the converse inclusion, let $p\in\CC$. 
Then, $c(p)\in k\Z$ by Proposition \ref{LemKZ}. 
Furthermore, $d(\CC)=2$ by Proposition \ref{LemRDCaseH}.
Thus, if $p_1\otimes p_2$ is any rotated version of $p$ in nest decomposed form, then $c(p_1)\in 2\Z$ by Lemma \ref{LemRDGlobal}.
 By Proposition \ref{DescriptionCaseHLocal}, we infer $p\in\categ{\HHH}{\loc}{k,2}\subset\langle b_k, b_2\otimes\tilde b_2,\vierpartwwbb,\vierpartwbwb\rangle$. Since $b_2\otimes\tilde b_2=\paarpartww\otimes\paarpartbb$, we infer 
$\CC= \langle b_k,\vierpartwbwb,\paarpartww\otimes\paarpartbb\rangle$.

(ii) Let $\CC$ be locally colorized and let $k:=k(\CC)$ and $d:=d(\CC)$.

\emph{Case 1.} Let $K^{\CC}(\cutpaarpartbb)=\emptyset$. Since $\CC$ is in case $\HHH$, we have $\langle\vierpartwbwb\rangle\subset \CC$. 
Conversely, let $p\in\CC$ be a partition without upper points. Assume that there is a block of $p$ which does not connect white and black points in an alternating way. Using rotation, we can bring $p$ in nest decomposed form $p=p_1\otimes p_2$ such that the first and the last point of $p_2$ have the same color. 
(In case $p_1=\emptyset$, we replace $p_1$ by $\paarpartwb$.)
This contradicts $K^{\CC}(\cutpaarpartbb)=K^{\CC}(\cutpaarpartww)=\emptyset$.
We conclude that each block of $p\in\CC$ connects white and black points in an alternating way. Furthermore, they are of even length. Otherwise, the first and the last point would have the same color, and again we would find an example of a partition in $NDF^{\CC}(\cutpaarpartbb)$ or in $NDF^{\CC}(\cutpaarpartww)$. By Proposition \ref{LemBaseCategoryH}, we infer $p\in\langle\vierpartwbwb\rangle$ and hence $\CC=\langle\vierpartwbwb\rangle$.

\emph{Case 2.} Let $K^{\CC}(\cutpaarpartbb)\neq \emptyset$. 
Then $\langle b_k,b_d\otimes\tilde b_d,\vierpartwwbb,\vierpartwbwb\rangle\subset\CC$ by  Proposition \ref{LemRDCaseH} and Lemma \ref{LemSamplePartitionsH}.
Conversely, let $p\in\CC$. By Lemma \ref{LemKZ}, we have $c(p)\in k\Z$. Furthermore,  if $p'=p_1\otimes p_2$ is any rotated version of $p$ in nest decomposed form, we have $c(p_1)\in d\Z$ by Proposition \ref{LemRDCaseH}. We thus have $p\in\categ{\HHH}{\loc}{k,d}$, which by Proposition \ref{DescriptionCaseHLocal} yields
$p\in\langle b_k,b_d\otimes\tilde b_d,\vierpartwwbb,\vierpartwbwb\rangle$. This shows $\langle b_k,b_d\otimes\tilde b_d,\vierpartwwbb,\vierpartwbwb\rangle=\CC$.
\end{proof}

\begin{cor}
We have $\categ{\HHH}{\loc}{k,d}=\langle b_k,b_d\otimes\tilde b_d,\vierpartwwbb,\vierpartwbwb\rangle$ in Proposition \ref{DescriptionCaseHLocal}. In particular, all these categories are pairwise different.
\end{cor}
\begin{proof}
In the above theorem, we showed $\langle b_k,b_d\otimes\tilde b_d,\vierpartwwbb,\vierpartwbwb\rangle\subset \CC\subset \categ{\HHH}{\loc}{k,d}$ whenever $\CC$ is a locally colorized category in case $\HHH$ with $k=k(\CC)$, $d=d(\CC)$ and $K^{\CC}(\cutpaarpartbb)\neq \emptyset$. Together with  $\categ{\HHH}{\loc}{k,d}\subset\langle b_k,b_d\otimes\tilde b_d,\vierpartwwbb,\vierpartwbwb\rangle$ of Proposition \ref{DescriptionCaseHLocal}, we have equality here. Moreover, it can easily be seen that the sets $\categ{\HHH}{\loc}{k,d}$ are distinct.
\end{proof}

\begin{rem}
\begin{itemize}
\item[(a)] One can show that the categories  $\langle b_k,\vierpartwbwb,\paarpartww\otimes\paarpartbb\rangle$  are given by the set of all partitions $p\in NC^{\twocol}$ such that $c(p)\in k\Z$ and all blocks of $p$ have even length, for $k\in 2\Z$.
\item[(b)]
 The non-colored case $\langle\vierpart\rangle$ is obtained from $\langle b_k,\vierpartwbwb,\paarpartww\otimes\paarpartbb\rangle$ for $k=2$ in the sense of Proposition \ref{PropOneColored}.
\end{itemize}
\end{rem}

\section{Case $\SSS$}\label{SectCaseS}

We now consider the case $\SSS$, i.e. categories $\CC\subset NC^{\twocol}$ such that $\vierpartwbwb$ and $\singletonw\otimes\singletonb$ are in $\CC$. 

\subsection{Determining the parameters}

\begin{prop}\label{LemRDCaseS}
Let $\CC\subset NC^{\twocol}$ be a category of noncrossing partitions in case $\SSS$. 
\begin{itemize}
\item[(a)] We always have $\positionerwbwb\in\CC$.
\item[(b)] If $\CC$ is locally colorized, then  $\vierpartwwbb\notin\CC$.
\item[(c)] If $\CC$ is globally colorized, then $d(\CC)=1$ and $k(\CC)\in\N_0$. Moreover, $\positionerwwbb\in\CC$.
\item[(d)] If $\CC$ is locally colorized, then $k(\CC),d(\CC)\in\N_0\backslash\{1\}$ and  we have:
\[K^{\CC}(\cutpaarpartwb)=K^{\CC}(\cutpaarpartbw)=d\Z \quad\textnormal{and}\quad K^{\CC}(\cutpaarpartbb)= -K^{\CC}(\cutpaarpartww)=d\Z+1\]
\end{itemize}
\end{prop}
\begin{proof}
(a) By Lemma \ref{LemPartRole}, we may disconnect the white points from $\vierpartwbwb$.

(b) If $\vierpartwwbb\in\CC$, then also $\paarpartww\otimes\singletonb\otimes\singletonb\in\CC$, which implies $\paarpartww\otimes\paarpartbb\in\CC$.

(c) By (a) and using color permutation, we have $\positionerwwbb\in\CC$, thus $d(\CC)=1$, by Lemma \ref{LemRDGlobal}.

(d) Use (a), (b) and Lemma \ref{Lem213NEU}(c). If $d(\CC)=1$ or $k(\CC)=1$, then $\positionerwwbb\in\CC$ by Lemma \ref{LemBsInC}(a).  Using (a) and Lemma \ref{LemPartRole}(e), we infer $\singletonw\otimes\singletonw\otimes\paarpartbb\in\CC$ which implies $\paarpartww\otimes\paarpartbb\in\CC$ by Lemma \ref{RemErasePoints}, a contradiction.
\end{proof}

\subsection{Finding partitions realizing the parameters}

\begin{lem}\label{LemSamplePartitionsS}
 Let $\CC\subset NC^{\twocol}$ be a category in case $\SSS$.
\begin{itemize}
 \item[(a)] If $k=k(\CC)\neq 0$, then $\singletonw^{\otimes k}\in\CC$.
 \item[(b)] If $d=d(\CC)\neq 0$, then $\positionerd\in\CC$.
\end{itemize}
\end{lem}
\begin{proof}
See Lemma \ref{LemKD}(a).
\end{proof}

\subsection{Description of natural categories}

\begin{prop}\label{DescriptionCaseSLocal}
We have the following natural categories in case $\SSS$.
Let $k,d\in\N_0$  such that $d$ is a divisor of $k$, if $k\neq 0$. Denote by $\categ{\SSS}{\loc}{k,d}$ the set of all partitions $p\in NC^{\twocol}$ such that
 \begin{itemize}
 \item[(i)] $c(p)\in k\Z$,
 \item[(ii)] if $p_1\otimes p_2$ is any rotated version of $p$ in nest decomposed form such that the first and the last point of $p_2$
   \begin{itemize}
    \item[$\ldots$] have inverse colors, then $c(p_1)\in d\Z$,
    \item[$\ldots$] both are black, then $c(p_1)\in d\Z+1$,
    \item[$\ldots$] both are white, then $-c(p_1)\in d\Z+1$.
   \end{itemize}
 \end{itemize}
We have $\categ{\SSS}{\loc}{k,d}\subset \langle \singletonw^{\otimes k}, \positionerd,\vierpartwbwb,\singletonw\otimes\singletonb\rangle$.
\end{prop}
\begin{proof} Denote $\langle \singletonw^{\otimes k}, \positionerd,\vierpartwbwb,\singletonw\otimes\singletonb\rangle$ by $\mathcal D$ and let $p\in\categ{\SSS}{\loc}{k,d}$. By rotation, we may always assume that $p$ has no upper points.
 We give a proof by induction on the number $m(p)$ of those blocks of $p$ which have length greater or equal two. 

\emph{Case 1.} Let $m(p)=0$, i.e. $p$ consists only of singletons. 
Since $c(p)\in k\Z$, we have up to permutation of the colors (see Lemma \ref{LemPartRole}(b)) $p=\singletonw^{\otimes kt}\otimes\left(\singletonw\otimes\singletonb\right)^{\otimes w}$ for some number $w$. Hence $p\in\mathcal D$.

\emph{Case 2.} Let $m(p)=1$. Using rotation, $p$ is of the following form:
\[p=a^{\epsilon_1}X_1a^{\epsilon_2}X_2\ldots a^{\epsilon_l}X_l\]
Here, the points $a^{\epsilon_i}$ form a block of length $l\geq 2$, and the $X_i$ are some tensor products of the singletons $\singletonw$ and $\singletonb$. If now all points $a^{\epsilon_i}$ had alternating colors, we could first argue that the partition $a^{\epsilon_1}\ldots a^{\epsilon_l}$ is in $\langle\vierpartwbwb\rangle\subset\mathcal D$ and then insert the tensor products $X_i$ of singletons between the legs using $\singletonw^{\otimes k}$ and $\positionerd$. Unfortunately, the alternating coloring is not always the case. We therefore construct a partition $p'$ involving some ``correction points''. It will be of the form:
\[p'=A_1'X_1'A_2'X_2'\ldots A_l'X_l'\]
The construction of $p'$ is as follows. If $a^{\epsilon_i}$ and $a^{\epsilon_{i+1}}$ have different colors, then up to permutation of the colors, $X_i$ is of the form $X_i=\singletonw^{\otimes dt_i}\otimes \left(\singletonw\otimes\singletonb\right)^{\otimes w_i}$ for some $w_i\in\N_0$, $t_i\in\Z$, by condition (ii) of $\categ{\SSS}{\loc}{k,d}$. We put $A_{i+1}':=a^{\epsilon_{i+1}}$ and $X_i':=X_i$. On the other hand, if $a^{\epsilon_i}$ and $a^{\epsilon_{i+1}}$ both are black, then $X_i=\singletonw^{\otimes dt_i+1}\otimes \left(\singletonw\otimes\singletonb\right)^{\otimes w_i}$ up to permutation, and we put $A_{i+1}':=a^{-\epsilon_{i+1}}a^{\epsilon_{i+1}}$, $X_i':=X_i\otimes\singletonb$, $w_i':=w_i+1$. If $a^{\epsilon_i}$ and $a^{\epsilon_{i+1}}$ both are white, then  $X_i':=X_i\otimes\singletonw$ instead. 
Finally, we put $A_1':=a^{\epsilon_1}$ and $X_l':=X_l$ if $a^{\epsilon_1}$ and $a^{\epsilon_l}$ have inverse colors and $A_1':=a^{-\epsilon_l}a^{\epsilon_1}$, $X_l':=X_l\otimes \singletonb$ or $X_l':=X_l\otimes \singletonw$ otherwise.

Now, the partition $q_1:=A_1'A_2'\ldots A_l'$ consists only of one block of even length with alternating colors, by construction. By Lemma \ref{LemBaseCategoryH}, it is contained in $\langle\vierpartwbwb\rangle\subset\DD$. 

Let $q_2$ be the partition obtained from $q_1$ by inserting subpartitions $(\singletonw\otimes\singletonb)^{\otimes w_i'}$ between $A_i'$ and $A_{i+1}'$, and $(\singletonw\otimes\singletonb)^{\otimes w_l'}$ after $A_l'$. By Lemma \ref{PropCategOperations}(d) we have $q_2\in\DD$. 
Moreover, $c(p')=c(p)$ by construction and $c(p)=\sum_i dt_i\in k\Z$. Let $q_3:=\singletonw^{\otimes c(p)}$, hence
$q_2\otimes q_3$ is in $\DD$. Since $\positionerd\in\DD$, we may use Lemma \ref{LemPartRole} to shift the partitions $\singletonw^{\otimes dt_i}$ (or $\singletonb^{\otimes dt_i}$ resp.) at the right positions, and we 
infer $p'\in\DD$. Using the pair partitions, we finally erase all extra points in $p'$ together with the ``correction singletons'' and we deduce $p\in\DD$.

\emph{Case 3}. Let $m(p)>1$. By rotation, $p$ can be brought in nest decomposed form $p=p_1\otimes p_2$ such that $m(p_2)=1$. Such a decomposition exists since $p$ is noncrossing. 

\emph{Case 3a.} If the first and the last point of $p_2$ have inverse colors, we have $c(p_1)\in d\Z$ by condition (ii). We may assume $c(p_1)\geq 0$, i.e. $c(p_1)=ds$ for some $s\in\N_0$. Then, the partition $p_1':=p_1\otimes \singletonb^{\otimes ds}$ satisfies $c(p_1')=0$ and $m(p_1')=m(p_1)=m(p)-1$. As conditions (i) and (ii) are fulfilled for $p_1'$ (compare the proof of Proposition \ref{LemBaseCategoryH}), we infer $p_1'\in \DD$ by induction hypothesis. By Case 2, we also have $p_2':=\singletonw^{\otimes ds}\otimes p_2\in\DD$, since $c(p_2')=c(p_2)+c(p_1)=c(p)\in k\Z$. Thus, we obtain $p_1'\otimes p_2'\in\DD$ and hence $p=p_1\otimes p_2\in\DD$ by Lemma \ref{PropCategOperations}(b).

\emph{Case 3b.} On the other hand, if both the first and the last point of $p_2$ are black, we consider $p_1':=p_1\otimes \singletonb$. Furthermore, we consider the partition $p_2'$ obtained from $p_2$ by changing the color of its first point from black to white and we insert a black singleton to the right of this first point. Then, $p_1'\otimes p_2'$ is a partition in nest decomposed form fulfilling the conditions (i) and (ii). By Case 3a, we infer $p_1'\otimes p_2'\in\DD$. Since $\positionerwbwb\in\DD$, we may use Lemma \ref{LemPartRole} to infer that $p_1\otimes\singletonb\otimes\singletonw\otimes p_2$ is in $\DD$, which yields $p\in\DD$. We proceed in the same way, if the first and the last point of $p_2$ are white.
\end{proof}

\subsection{Classification in the case $\SSS$}

\begin{thm}\label{ThmCaseS}
Let $\CC\subset NC^{\twocol}$ be a category of noncrossing partitions in case $\SSS$. Then $\CC$ coincides with one of the following categories.
\begin{itemize}
\item[(i)] If $\CC$ is globally colorized, then $\CC= \categ{\SSS}{\glob}{k}=\langle \singletonw^{\otimes k},\vierpartwbwb,\singletonw\otimes\singletonb,\paarpartww\otimes\paarpartbb\rangle$ for $k=k(\CC)\in \N_0$. Here $\categ{\SSS}{\glob}{k}:=\categ{\SSS}{\loc}{k,1}$.
\item[(ii)] If $\CC$ is locally colorized, then $\CC=\langle \singletonw^{\otimes k},\positionerd,\vierpartwbwb,\singletonw\otimes\singletonb\rangle$ for $k=k(\CC)\in\N_0\backslash\{1\}$ and $d=d(\CC)\in\N_0\backslash\{1\}$.
\end{itemize}
\end{thm}
\begin{proof}
(i) By Lemma \ref{LemSamplePartitionsS}, we have $\singletonw^{\otimes k}\in\CC$ for $k=k(\CC)$. Thus $\langle \singletonw^{\otimes k},\vierpartwbwb,\singletonw\otimes\singletonb,\paarpartww\otimes\paarpartbb\rangle\subset\CC$. 
Conversely, let $p\in\CC$. Then $c(p)\in k\Z$ by Proposition \ref{LemKZ}. Furthermore, $d=d(\CC)=1$ by Proposition \ref{LemRDCaseS}, hence $p\in\categ{\SSS}{\loc}{k,1}\subset\langle \singletonw^{\otimes k}, \positionerwwbb,\vierpartwbwb,\singletonw\otimes\singletonb\rangle$ by Proposition \ref{DescriptionCaseSLocal}.
We have $\langle \singletonw^{\otimes k}, \positionerwwbb,\vierpartwbwb,\singletonw\otimes\singletonb\rangle=
\langle \singletonw^{\otimes k},\vierpartwbwb,\singletonw\otimes\singletonb,\paarpartww\otimes\paarpartbb\rangle$.

(ii) Let $\CC$ be locally colorized. Using Lemma \ref{LemSamplePartitionsS}, we infer $\langle \singletonw^{\otimes k},\positionerd,\vierpartwbwb,\singletonw\otimes\singletonb\rangle\subset\CC$ for $k=k(\CC)$ and $d=d(\CC)$. Conversely, let $p\in\CC$. Then $c(p)\in k\Z$ by Proposition \ref{LemKZ}. Let $p_1\otimes p_2$ be a rotated version of $p$ in nest decomposed form. If the first and the last point of $p_2$ have inverse colors, then $c(p_1)\in d\Z$ by Proposition \ref{LemDDivisor}. 
If the first and the last point of $p_2$ both are black, we have to prove  $s-1\in d\Z$ for $s:=c(p_1)$.

Assume that $s\geq 0$. Using Lemma \ref{PropCategOperations}(b)  and Lemma \ref{LemPartRole}, we infer $\singletonw^{\otimes s}\otimes p_2\in\CC$. We have $s\neq 0$, since otherwise $p_2\in\CC$ and rotation would yield a partition such that two neighbouring points have the same color and belong to the same block. By Lemma \ref{LemBsInC}(c) and Proposition \ref{LemRDCaseS} this would lead to a contradiction. 

We thus have $s\geq 1$.
Since $\positionerwbwb\in\CC$ by Proposition \ref{LemRDCaseS}, we may shift one of the white singletons of $p_1'$ to the right hand side of the first point of $p_2$, which inverts the colors of these two points. We infer that the partition $\singletonw^{\otimes s-1}\otimes p_2'$ is in $\CC$, where $p_2'$ is in nest decomposed form such that the first and the last point have inverse colors. By Proposition \ref{LemDDivisor}, we thus have $s-1\in d\Z$.

As for $s<0$, we have $\singletonb^{\otimes -s}\otimes p_2\in\CC$. By Lemma \ref{PropCategOperations}(d), we infer that also $\singletonb^{\otimes -s}\otimes\singletonb\otimes\singletonw\otimes p_2\in\CC$. Again, shifting the white singleton to the right hand side of the first point of $p_2$ yields $\singletonb^{\otimes -s+1}\otimes p_2'\in\CC$ where the first and the last point of $p_2'$ have inverse colors belonging to the same block. Thus, $-s+1\in d\Z$ and hence $s-1\in d\Z$. 

A similar proof shows that $s+1\in d\Z$ if the first and the last point of $p_2$ are white. We thus have $p\in \categ{\SSS}{\loc}{k,d}$ and by Proposition \ref{DescriptionCaseSLocal}, we deduce  $p\in\langle \singletonw^{\otimes k},\positionerd,\vierpartwbwb,\singletonw\otimes\singletonb\rangle$. This shows $\CC=\langle \singletonw^{\otimes k},\positionerd,\vierpartwbwb,\singletonw\otimes\singletonb\rangle$.
\end{proof}

\begin{cor}
We have
$\categ{\SSS}{\loc}{k,d}= \langle \singletonw^{\otimes k}, \positionerd,\vierpartwbwb,\singletonw\otimes\singletonb\rangle$ in Proposition \ref{DescriptionCaseSLocal}. In particular, all these categories are pairwise different.
\end{cor}

\begin{rem}
 The non-colored case $\langle\singleton\otimes\singleton,\vierpart\rangle$ is obtained from $\langle \singletonw^{\otimes k},\vierpartwbwb,\singletonw\otimes\singletonb,\paarpartww\otimes\paarpartbb\rangle$ for $k=2$, whereas $\langle\singleton,\vierpart\rangle$ is the case $k=1$ (see Proposition \ref{PropOneColored}).
\end{rem}

\section{Case $\BBB$}\label{SectCaseB}

Finally, we turn to the case $\BBB$, i.e. to categories $\CC\subset NC^{\twocol}$ such that $\vierpartwbwb\notin\CC$ and $\singletonw\otimes\singletonb\in\CC$. Here, all blocks of partitions $p\in\CC$ have length one or two (Lemma \ref{LemCases}). Like in the non-colored case, this is the most complicated situation, as we can already see when investigating which parameters can occur.

\subsection{Determining the parameters}

\begin{prop}\label{LemRDCaseB}
Let $\CC\subset NC^{\twocol}$ be a category of noncrossing partitions in case $\BBB$. 
\begin{itemize}
\item[(a)] If $\CC$ is globally colorized, then the cases $d(\CC)=1$ and $d(\CC)=2$ can occur.
\item[(b)] If $\CC$ is locally colorized, then
 \begin{itemize}
 \item[(i)] either $K^{\CC}(\cutpaarpartbb)=\emptyset$ and $k(\CC),d(\CC)\in\N_0$,
 \item[(ii)] or we have:
 \[K^{\CC}(\cutpaarpartwb)=K^{\CC}(\cutpaarpartbw)=d\Z \quad\textnormal{and}\quad K^{\CC}(\cutpaarpartbb)= -K^{\CC}(\cutpaarpartww)=d\Z+(r+1)\]
Here,  $r:=r(\CC):=\min\{t\geq 1\;|\; t\in K^{\CC}(\cutpaarpartbb)\}-1$ and $k(\CC)\in\N_0\backslash\{1\}$, $d=d(\CC)\in\N_0\backslash\{1\}$. Furthermore, $r\neq 1$ and $r\in\{0,\frac{d}{2}\}\cap\N_0$.
 \end{itemize}
\end{itemize}
\end{prop}
\begin{proof}
(a) This follows directly from Lemma \ref{LemRDGlobal}.

(b) Let $K^{\CC}(\cutpaarpartbb)\neq\emptyset$. We have  $t+1\in K^{\CC}(\cutpaarpartbb)$ if and only if $\positionertpluseins\in\CC$, by Lemma \ref{LemBsInC}(a). Hence, by Lemma \ref{Lem213NEU}(c) it only remains to prove $r\in\{0,\frac{d}{2}\}\cap\N_0$, $k(\CC)\neq 1$ and $d(\CC)\neq 1$. If $k(\CC)=1$ or $d(\CC)=1$, we would have $\positionerwwbb\in\CC$ by Lemma \ref{LemKD}(a). By Lemma \ref{LemPartRole}(e), this would imply $\singletonw^{\otimes r+1}\otimes\singletonb^{\otimes r-1}\otimes\paarpartbb\in\CC$, a contradiction to $\CC$ being locally colorized.

Finally, let $r\neq 0$. We now prove $r=\frac{d}{2}$. The following two partitions are in $\CC$.
\setlength{\unitlength}{0.5cm}
\begin{center}
 \begin{picture}(16,1.5)
     \put(1.6,0.2){\line(0,1){0.9}}
     \put(4.4,0.2){\line(0,1){0.9}}
     \put(1.6,1.1){\line(1,0){2.8}}
     \put(0.05,-0.05){$^\uparrow$}
     \put(0.2,0.4){\tiny$^{\otimes r+1}$}
     \put(2.05,-0.05){$^\uparrow$}
     \put(2.2,0.4){\tiny$^{\otimes r-1}$}
     \put(0,-0.2){$\circ$}
     \put(1.4,-0.2){$\bullet$}
     \put(2,-0.2){$\bullet$}
     \put(4.2,-0.2){$\bullet$}
     \put(5,0){$\in\CC$}
     \put(7.5,0){\textnormal{and}}
     \put(10.2,0.2){\line(0,1){0.9}}
     \put(12.8,0.2){\line(0,1){0.9}}
     \put(10.2,1.1){\line(1,0){2.6}}
     \put(10.45,-0.05){$^\uparrow$}
     \put(10.6,0.4){\tiny$^{\otimes r+1}$}
     \put(13.05,-0.05){$^\uparrow$}
     \put(13.2,0.4){\tiny$^{\otimes r-1}$}
     \put(10,-0.2){$\circ$}
     \put(10.4,-0.2){$\bullet$}
     \put(12.6,-0.2){$\circ$}
     \put(13,-0.2){$\circ$}
     \put(15,0){$\in\CC$}
 \end{picture}
\end{center}
Forming the tensor product of these two partitions and composing it with  $\baarpartbw$ (Remark \ref{RemComp})  we infer:
\setlength{\unitlength}{0.5cm}
\begin{center}
 \begin{picture}(5,1.5)
     \put(1.4,0.2){\line(0,1){0.9}}
     \put(3.8,0.2){\line(0,1){0.9}}
     \put(1.4,1.1){\line(1,0){2.4}}
     \put(0.05,-0.05){$^\uparrow$}
     \put(0.2,0.4){\tiny$^{\otimes 2r}$}
     \put(1.85,-0.05){$^\uparrow$}
     \put(2,0.4){\tiny$^{\otimes 2r}$}
     \put(0,-0.2){$\circ$}
     \put(1.2,-0.2){$\bullet$}
     \put(1.8,-0.2){$\bullet$}
     \put(3.6,-0.2){$\circ$}
     \put(4.4,0){$\in\CC$}
 \end{picture}
\end{center}
Thus, $2r\in K^{\CC}(\cutpaarpartwb)=d\Z$ and $d\neq 0$. Assume $2r=ds$ for some $s\geq 2$.
 Then $d\leq r$ and $r':=r-d+1$ is a number $0<r'<r+1$. Using the partition $\positionerd\in\CC$ (which is in $\CC$ by Lemma \ref{LemBsInC}), we can shift $d$ of the $r+1$ white singletons of $\positionerrpluseins$ from the outside of the pair to the inside (see Lemma \ref{LemPartRole}(e)). This yields a partition showing that $r'$ is in $K^{\CC}(\cutpaarpartbb)$, in contradiction to the minimality of $r$. We conclude $2r=d$, if $r\neq 0$.
\end{proof}

\subsection{Finding partitions realizing the parameters}

\begin{lem}\label{LemSamplePartitionsB}
Let $\CC\subset NC^{\twocol}$ be a category  in case $\BBB$. 
\begin{itemize}
 \item[(a)] If $k=k(\CC)\neq 0$, then $\singletonw^{\otimes k}\in\CC$.
 \item[(b)] If $d=d(\CC)\neq 0$, then $\positionerd\in\CC$.
 \item[(c)] If $K^{\CC}(\cutpaarpartbb)\neq\emptyset$  then $\positionerrpluseins\in\CC$ for $r=r(\CC)$.
\end{itemize}
\end{lem}
\begin{proof}
Lemma \ref{LemKD}(a) and Lemma \ref{LemBsInC}(a).
\end{proof}

\subsection{Description of natural categories}

\begin{prop}\label{DescriptionCaseBLocal}\label{LemBaseCategoryB}
We have the following natural categories in case $\BBB$.
\begin{itemize}
 \item[(a)] The category $\langle\singletonw\otimes\singletonb\rangle$ consists of all noncrossing partitions $p\in NC^{\twocol}$ such that when $p$ is rotated to a partition having no upper points 
 \begin{itemize}
 \item[(i)] all blocks have size one or two,
 \item[(ii)] the blocks of size two connect a black point and a white point,
 \item[(iii)] the number of black singletons and the number of white singletons between two legs of every pair coincide, and on the global level, too.
 \end{itemize}
 \item[(b)] Let $k,d\in\N_0$ be such that $d$ is a divisor of $k$, if $k\neq 0$. Let $r\in\{0,\frac{d}{2}\}\cap \N_0$ with $r\neq 1$. Denote by $\categ{\BBB'}{\loc}{k,d,r}$ the set of all noncrossing partitions $p\in NC^{\twocol}$ such that
 \begin{itemize}
 \item[(i)] all blocks have size one or two,
 \item[(ii)] $c(p)\in k\Z$,
 \item[(iii)] if $p_1\otimes p_2$ is any rotated version of $p$ in nest decomposed form such that the first and the last point of $p_2$ 
   \begin{itemize}
    \item[$\ldots$] have inverse colors, then $c(p_1)\in d\Z$,
    \item[$\ldots$] both are black, then $c(p_1)\in d\Z+r+1$,
    \item[$\ldots$] both are white, then $-c(p_1)\in d\Z+r+1$.
   \end{itemize}
 \end{itemize}
 We have  $\categ{\BBB'}{\loc}{k,d,r}\subset \langle\singletonw^{\otimes k}, \positionerd,\positionerrpluseins,\singletonw\otimes\singletonb\rangle$.
 \item[(c)] Denote by $\categ{\BBB}{\loc}{k,d}$ the set defined as $\categ{\BBB'}{\loc}{k,d,r}$, but with the additional condition that all blocks of $p$ of size two are of the form $\paarpartwb$ or $\paarpartbw$ when being rotated to one line. We then have $\categ{\BBB}{\loc}{k,d}\subset\langle\singletonw^{\otimes k}, \positionerd,\singletonw\otimes\singletonb\rangle$.
\end{itemize}
\end{prop}
\begin{proof}
(a) Denote the set of all partitions $p\in NC^{\twocol}$ with (i), (ii) and (iii) by $\mathcal E$. It is easy to see that $\mathcal E$ is a category of partitions containing $\singletonw\otimes\singletonb$. So, we only need to prove $p\in\langle\singletonw\otimes\singletonb\rangle$ for all $p\in\mathcal E$. We do so by induction on the number $m$ of blocks of size two of $p$. If $m=0$, then $p$ consists of $l$ white singletons and $l$ black singletons, for some $l\in\N$. Hence, it is of the form $(\singletonw\otimes\singletonb)^{\otimes l}\in\langle\singletonw\otimes\singletonb\rangle$ up to permutation of colors (see also Lemma \ref{LemPartRole}). If $m=1$, the partition $p$ is of the form $p=XaYa^{-1}$ up to rotation, where $X$ and $Y$ are tensor products of singletons and $a$ and $a^{-1}$ form a pair block $\paarpartwb$ or $\paarpartbw$. By assumption, the number of white singletons in $Y$ and the number of black singletons coincide, hence $Y\in\langle\singletonw\otimes\singletonb\rangle$ by case $m=0$. Likewise $X\in\langle\singletonw\otimes\singletonb\rangle$, by the assumption on the global color distribution of the singletons. We infer $p\in\langle\singletonw\otimes\singletonb\rangle$.

If $m>1$, we can write $p=p_1\otimes p_2$ in nest decomposed form up to rotation, where $p_2$ consists of one pair block and some singletons. Then $p_2\in\langle\singletonw\otimes\singletonb\rangle$ by case $m=1$ and $p_1\in\langle\singletonw\otimes\singletonb\rangle$ by  the induction hypothesis. Thus $p=p_1\otimes p_2\in \langle\singletonw\otimes\singletonb\rangle$.

(b) Let $p\in\categ{\BBB'}{\loc}{k,d,r}$.
Denote the category $\langle\singletonw^{\otimes k}, \positionerd,\positionerrpluseins,\singletonw\otimes\singletonb\rangle$ by $\DD$.
We prove $p\in\DD$ by induction on the number $m$ of blocks of $p$ of size two. 

\emph{Case 1.} Let $m=0$. Up to rotation and permutation of the colors, $p$ is of the form $p=\singletonw^{\otimes ks}\otimes(\singletonw\otimes\singletonb)^{\otimes w}$ for some $w\geq 0$ and $c(p)=ks\in k\Z$. Hence $p\in\langle\singletonw^{\otimes k},\singletonw\otimes\singletonb\rangle\subset\DD$.

\emph{Case 2.} Let $m=1$. Up to rotation, $p$ is of the form $p=p_1\otimes a^{\epsilon_1} p_2^0 a^{\epsilon_2}$ where $a^{\epsilon_1}$ and $a^{\epsilon_2}$ form a pair block, and $p_1$ and $p_2^0$ consist only of singletons respectively.

\emph{Case 2a.} If the pair on $a^{\epsilon_1}$ and $a^{\epsilon_2}$ is of the form $\paarpartwb$ or $\paarpartbw$, then $c(p_1)\in d\Z$. Consider $p_1':=p_1\otimes \singletonw^{\otimes -c(p_1)}$. Then $p_1'\in\DD$ by Case 1 since $c(p_1')=0$. Furthermore, ${p_2^0}':=\singletonw^{\otimes c(p_1)}\otimes p_2^0$ is in $\DD$, again by Case 1 because $c({p_2^0}')=c(p_1)+c(p_2^0)=c(p)\in k\Z$. Therefore, the partition $p_1\otimes \singletonw^{\otimes -c(p_1)}\otimes a^{\epsilon_1}\singletonw^{\otimes c(p_1)} p_2^0 a^{\epsilon_2}$ is in $\DD$. Since $c(p_1)\in d\Z$, we may use the partition $\positionerd\in\DD$ to shift $c(p_1)$ singletons from inside the pair to the outside. Thus, $p_1\otimes \singletonw^{\otimes -c(p_1)}\otimes \singletonw^{\otimes c(p_1)}\otimes a^{\epsilon_1}p_2^0 a^{\epsilon_2}\in\DD$ from which we infer $p\in\DD$ using Lemma \ref{PropCategOperations}(b).

\emph{Case 2b.} If the pair on $a^{\epsilon_1}$ and $a^{\epsilon_2}$ is of the form $\paarpartbb$, then $c(p_1)\in d\Z+(r+1)$. Assume $c(p_1)=ds+r+1$ for some $s\geq 0$. Let $p_1'$ be the partition obtained from $p_1$ by removing $c(p_1)$ white singletons. Then $p_1'\in\DD$ by Case 1 since $c(p_1')=0$. Furthermore, let $p_2'=p_2^0\otimes \singletonw^{\otimes ds}\otimes\singletonw^{\otimes r-1}$. Then $p_2'\in\DD$ since $c(p_2')=c(p_2^0)+ds+(r-1)=c(p)\in k\Z$. Finally, the partition $p_1'\otimes \positionerrpluseins$ is in $\mathcal D$ and hence also $p_1'\otimes \singletonw^{\otimes r+1}a^{\epsilon_1}\singletonb^{\otimes r-1}p_2^0\singletonw^{\otimes ds}\singletonw^{\otimes r-1}a^{\epsilon_2}$ by Lemma \ref{PropCategOperations}(d). Using $\positionerd$ we can shift $ds$ white singletons from inside the pair to the outside. Up to permutation of colors of the singletons and by Lemma \ref{PropCategOperations}(b), this yields $p$ which is hence in $\DD$. We proceed in a similar way for $s<0$ and likewise in the case that $a^{\epsilon_1}$ and $a^{\epsilon_2}$ form a pair $\paarpartww$.

\emph{Case 3.} Let $m>1$. Up to rotation, $p$ is in nest decomposed form $p=p_1\otimes p_2$ such that $p_2$ contains only one block of size two. 
Then, $p_1':=p_1\otimes \singletonw^{\otimes c(p_2)}$ is in $\DD$ by induction hypothesis, since $c(p_1')=c(p)$. Likewise $p_2':=\singletonw^{\otimes -c(p_2)}\otimes p_2$ is in $\DD$ by Case 2. Hence $p_1\otimes \singletonw^{\otimes c(p_2)}\otimes\singletonw^{\otimes -c(p_2)}\otimes p_2\in\DD$ from which we deduce $p\in\DD$.

(c) Note that in the proof of (b) we used the partition $\positionerrpluseins$ only when blocks $\paarpartww$ or $\paarpartbb$ where involved. Hence, $p\in \langle\singletonw^{\otimes k}, \positionerd,\singletonw\otimes\singletonb\rangle$ if  all blocks of $p$ of size two are of the form $\paarpartwb$ or $\paarpartbw$.
\end{proof}

\subsection{Classification in the case $\BBB$}

\begin{thm}\label{ThmCaseB}
Let $\CC\subset NC^{\twocol}$ be a category of noncrossing partitions in case $\BBB$. Then $\CC$ coincides with one of the following categories.
\begin{itemize}
\item[(a)] If $\CC$ is globally colorized and 
 \begin{itemize}
 \item[$\ldots$] if $d(\CC)=2$, then $\CC=\categ{\BBB}{\glob}{k}:= \langle \singletonw^{\otimes k}, \singletonw\otimes\singletonb,\paarpartww\otimes\paarpartbb\rangle$ for $ k=k(\CC)\in 2\N_0$, 
 \item[$\ldots$] if $d(\CC)=1$, then $\CC= \categ{\BBB'}{\glob}{k}=\langle \singletonw^{\otimes k},\positionerwwbb,\singletonw\otimes\singletonb,\paarpartww\otimes\paarpartbb\rangle$ for $k=k(\CC)\in \N_0$.  
 Here, $\categ{\BBB'}{\glob}{k}:=\categ{\BBB'}{\loc}{k,1,0}$.
 \end{itemize}
\item[(b)] If $\CC$ is locally colorized and 
 \begin{itemize}
 \item[$\ldots$] if $K^{\CC}(\cutpaarpartbb)=\emptyset$, then $\CC= \langle\singletonw^{\otimes k}, \positionerd,\singletonw\otimes\singletonb\rangle$ for $k=k(\CC)$ and $d=d(\CC)\in\N_0$,
 \item[$\ldots$] if $K^{\CC}(\cutpaarpartbb)\neq\emptyset$, then $\CC= \langle\singletonw^{\otimes k}, \positionerd,\positionerrpluseins,\singletonw\otimes\singletonb\rangle$ for $k=k(\CC)\in\N_0\backslash\{1\}$, $d=d(\CC)\in\N_0\backslash\{1\}$ and $r=r(\CC)\in\{0,\frac{d}{2}\}\cap\N_0$ and $r\neq 1$.
 \end{itemize}
\end{itemize}
\end{thm}
\begin{proof}
(a) We have $\langle \singletonw^{\otimes k}, \singletonw\otimes\singletonb,\paarpartww\otimes\paarpartbb\rangle\subset\CC$  using Lemma \ref{LemSamplePartitionsB}.

\emph{Case 1.} Let $d(\CC)=2$. Let $p\in\CC$ be a partition without upper points. Then $c(p)=ks$ for some $s\in\Z$ by Proposition \ref{LemKZ}. The number of points between two legs of a pair of $p$ is even, because $c(p_1)\in 2\Z$ for all $p_1\otimes p_2$ in nest decomposed form (Lemma \ref{LemRDGlobal}). Consider $p':=p\otimes \singletonw^{\otimes -ks}$.
Let $p''$ be the partition obtained from $p'$ by replacing the colors of the points by the alternating color pattern white-black-white-black-etc. Then, all pair blocks are of the form $\paarpartwb$ or $\paarpartbw$, because there is an even number of points between two legs of a pair. Thus, $p''\in\langle\singletonw\otimes\singletonb\rangle\subset\langle \singletonw^{\otimes k},\singletonw\otimes\singletonb,\paarpartww\otimes\paarpartbb\rangle$ by Proposition \ref{LemBaseCategoryB}. Using permutation of colors, we infer $p'\in\CC$ since $c(p')=c(p'')$. This implies $p'\otimes\singletonw^{\otimes ks}\in\CC$ from which we deduce $p\in\CC$ using Lemma \ref{PropCategOperations}(b).

\emph{Case 2.} If $d(\CC)=1$, we have $\positionerwwbb\in\CC$ by Lemma \ref{LemBsInC}(a). Hence, $\langle \singletonw^{\otimes k},\positionerwwbb,\singletonw\otimes\singletonb,\paarpartww\otimes\paarpartbb\rangle\subset \CC$. 
Conversely, let $p\in\CC$. Then $c(p)\in k\Z$ by Proposition \ref{LemKZ} and thus $p\in \categ{\BBB'}{\loc}{k,1,0}\subset\langle\singletonw^{\otimes k},\positionerwwbb,\positionerwbwb,\singletonw\otimes\singletonb\rangle$ by Proposition \ref{LemBaseCategoryB}. But this category contains $\paarpartww\otimes\paarpartbb$ since we may shift the singletons in $\positionerwbwb$ arbitrarily (using $\positionerwwbb\in\CC$ and Lemma \ref{LemPartRole}(e)), and hence it coincides with $\langle\singletonw^{\otimes k},\positionerwwbb,\singletonw\otimes\singletonb,\paarpartww\otimes\paarpartbb\rangle$.

(b) Let $\CC$ be locally colorized. For $k=k(\CC)$ and $d=d(\CC)$, we have $\singletonw^{\otimes k}\in\CC$ and $\positionerd\in\CC$ by Lemma \ref{LemSamplePartitionsB}.

\emph{Case 1.} Let $K^{\CC}(\cutpaarpartbb)=\emptyset$. Then $\langle\singletonw^{\otimes k}, \positionerd,\singletonw\otimes\singletonb\rangle\subset\CC$. Conversely, let $p\in\CC$. 
Then $c(p)\in k\Z$ by Proposition \ref{LemKZ} and $c(p_1)\in d\Z$ for all $p_1\otimes p_2$ in nest decomposed form by Proposition \ref{LemDDivisor}. Furthermore, all blocks of size two are of the form $\paarpartwb$ or $\paarpartbw$ when being rotated to one line, since $K^{\CC}(\cutpaarpartbb)=K^{\CC}(\cutpaarpartww)=\emptyset$. Thus $p\in \categ{\BBB}{\loc}{k,d}$ which implies $p\in \langle\singletonw^{\otimes k}, \positionerd,\singletonw\otimes\singletonb\rangle$ by Proposition \ref{LemBaseCategoryB}.

\emph{Case 2.} Let $K^{\CC}(\cutpaarpartbb)\neq\emptyset$. Then $\langle\singletonw^{\otimes k}, \positionerd,\positionerrpluseins,\singletonw\otimes\singletonb\rangle\subset\CC$ for $r$ as in Proposition \ref{LemRDCaseB}.  Again, we use Proposition \ref{DescriptionCaseBLocal} to finish the proof.
\end{proof}

\begin{cor}
We have 
$\categ{\BBB'}{\loc}{k,d,r}= \langle\singletonw^{\otimes k}, \positionerd,\positionerrpluseins,\singletonw\otimes\singletonb\rangle$ and 
$\categ{\BBB}{\loc}{k,d}=\langle\singletonw^{\otimes k}, \positionerd,\singletonw\otimes\singletonb\rangle$ in Proposition \ref{DescriptionCaseBLocal}. In particular, these categories are pairwise different.
\end{cor}

\begin{rem}
\begin{itemize}
\item[(a)] If $r=\frac{d}{2}$, then $\positionerrpluseins\in\CC$ implies $\positionerd\in\CC$, see the proof of Proposition \ref{LemRDCaseB}.
\item[(b)] The non-colored case $\langle\singleton\otimes\singleton\rangle$ is obtained from $\langle \singletonw^{\otimes k},\singletonw\otimes\singletonb,\paarpartww\otimes\paarpartbb\rangle$ for $k=2$, whereas  $\langle\singleton\rangle$ is given by $k=1$. The category $\langle\positioner\rangle$ in turn coincides with $\langle \singletonw^{\otimes k},\positionerwwbb,\singletonw\otimes\singletonb,\paarpartww\otimes\paarpartbb\rangle$ for the case $k=2$ (see Proposition \ref{PropOneColored}).
\end{itemize}
\end{rem}

\section{Main result: Summary of the noncrossing case}\label{SectMainResult}

We finally classified all categories $\CC\subset NC^{\twocol}$ of noncrossing (two-colored) partitions. This constitutes the main result of our article. Here is an overview on the results split into the globally colorized case and the locally colorized case. For the convenience of the reader we recall that the definition of a category of partitions may be found in Section \ref{SectCateg}, the cases $\OOO, \HHH, \SSS$ and $\BBB$ are defined in Definition \ref{DefCases}, globally and locally colorization is given in Definition \ref{DefGlobalColor}, the partition $b_k$ is defined in Definition \ref{DefBs} whereas the operation $p\mapsto\tilde p$ is defined in Section \ref{SectOperations}, and the classification theorems are Theorems \ref{ThmCaseO}, \ref{ThmCaseH}, \ref{ThmCaseS}, and \ref{ThmCaseB}. Each of the categories can be described in two ways: firstly by generators and secondly in terms of colorizations, see Section \ref{SectCaseO} to \ref{SectCaseB}.

\begin{thm}\label{ThmClassifGlobal}
Let $\CC\subset NC^{\twocol}$ be a \emph{globally colorized} category of noncrossing partitions. Then it coincides with one of the following categories.
\begin{itemize}
 \item[Case $\OOO$:] $\categ{\OOO}{\glob}{k}=\langle\paarpartww^{\otimes \frac{k}{2}},\paarpartww\otimes\paarpartbb\rangle$ for $k\in 2\N_0$
 \item[Case $\HHH$:] $\categ{\HHH}{\glob}{k}=\langle b_k,\vierpartwbwb,\paarpartww\otimes\paarpartbb\rangle$ for $k\in 2\N_0$
 \item[Case $\SSS$:] $\categ{\SSS}{\glob}{k}=\langle \singletonw^{\otimes k},\vierpartwbwb,\singletonw\otimes\singletonb,\paarpartww\otimes\paarpartbb\rangle$ for $k\in \N_0$
 \item[Case $\BBB$:] $\categ{\BBB}{\glob}{k}=\langle \singletonw^{\otimes k}, \singletonw\otimes\singletonb,\paarpartww\otimes\paarpartbb\rangle$ for $ k\in 2\N_0$ 
 \item[or] $\categ{\BBB'}{\glob}{k}=\langle \singletonw^{\otimes k},\positionerwwbb,\singletonw\otimes\singletonb,\paarpartww\otimes\paarpartbb\rangle$ for $k\in \N_0$
\end{itemize}
\end{thm}

\begin{thm}\label{ThmClassifLocal}
Let $\CC\subset NC^{\twocol}$ be a \emph{locally colorized} category of noncrossing partitions.Then it coincides with one of the following categories.
\begin{itemize}
 \item[Case $\OOO$:] $\categg{\OOO}{\loc}=\langle\emptyset\rangle$
 \item[Case $\HHH$:] $\categg{\HHH'}{\loc}=\langle\vierpartwbwb\rangle$ 
 \item[or] $\categ{\HHH}{\loc}{k,d}=\langle b_k,b_d\otimes\tilde b_d,\vierpartwwbb,\vierpartwbwb\rangle$ for $k,d\in\N_0\backslash\{1,2\}$, $d\vert k$ 
 \item[Case $\SSS$:] $\categ{\SSS}{\loc}{k,d}=\langle \singletonw^{\otimes k},\positionerd,\vierpartwbwb,\singletonw\otimes\singletonb\rangle$ for $k,d\in\N_0\backslash\{1\}$, $d\vert k$
 \item[Case $\BBB$:] $\categ{\BBB}{\loc}{k,d}=\langle\singletonw^{\otimes k}, \positionerd,\singletonw\otimes\singletonb\rangle$ for $k,d\in \N_0$, $d\vert k$
 \item[or] $\categ{\BBB'}{\loc}{k,d,0}=\langle\singletonw^{\otimes k}, \positionerd,\positionerwbwb,\singletonw\otimes\singletonb\rangle$ for $k,d\in \N_0\backslash\{1\}$, $d\vert k$
 \item[or] $\categ{\BBB'}{\loc}{k,d,\frac{d}{2}}=\langle\singletonw^{\otimes k}, \positionerd,\positionerrpluseins,\singletonw\otimes\singletonb\rangle$ for $k\in \N_0\backslash\{1\}$, $d\in 2\N_0\backslash\{0,2\}$, $d\vert k$ and $r=\frac{d}{2}$.
\end{itemize}
\end{thm}

Here is a graphical overview of all categories of two-colored noncrossing partitions. The single framed categories are the locally colorized ones whose inclusions are indicated by single dashed lines (inclusions from top to bottom and from right to left, for fixed parameters $k$ and $d$). Constraints for inclusions are marked in brackets. The double framed categories are the globally colorized ones  with inclusion pattern according to the double dahed lines. The locally colorized categories are contained in the globally colorized ones according to the diagonal chain lines. In our graphic, we also included a cross marking the areas of the cases $\BBB$, $\OOO$, $\SSS$ and $\HHH$.

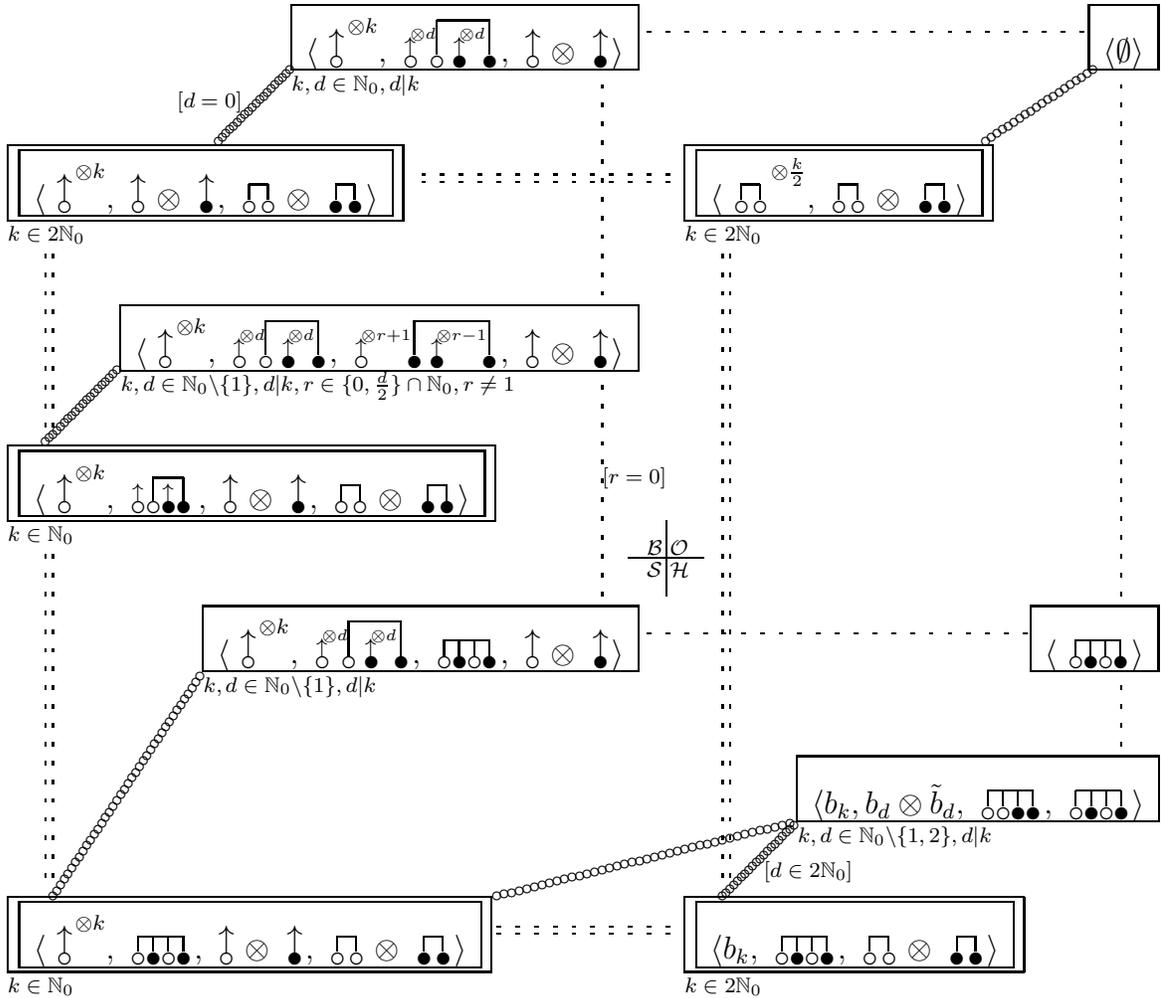
\setlength{\unitlength}{0.5cm}
\begin{center}
\begin{picture}(30,27)
\multiput(15.8,11)(0,0.5){12}{\line(0,1){.1}}
\multiput(15.8,19)(0,0.5){12}{\line(0,1){.1}}
\multiput(29.6,11)(0,0.5){28}{\line(0,1){.1}}
\multiput(29.6,7)(0,0.5){4}{\line(0,1){.1}}
\multiput(17,26)(0.5,0){24}{\line(1,0){.1}}
\multiput(17,10)(0.5,0){21}{\line(1,0){.1}}
\put(15.8,14){\tiny{$[r=0]$}}
\put(28.75,25){\frame{$\Aspace\langle\emptyset\rangle\Aspace$}}
\put(27.2, 9){\frame{$\Aspace\langle\vierpartwbwb\rangle\Aspace$}}
\put(21,5){\frame{$\Aspace\langle b_k,b_d\otimes\tilde b_d,\vierpartwwbb,\vierpartwbwb\rangle\Aspace$}}
 \put(21,4.45){\tiny{$k,d\in\N_0\backslash\{1,2\},d\vert k$}}
\put(7.55,25){\frame{$\Aspace\langle\singletonw^{\otimes k}, \positionerd,\singletonw\otimes\singletonb\rangle\Aspace$}}
 \put(7.55, 24.45){\tiny{$k,d\in \N_0,d\vert k$}}
\put(3,17){\frame{$\Aspace\langle\singletonw^{\otimes k}, \positionerd,\positionerrpluseins,\singletonw\otimes\singletonb\rangle\Aspace$}}
 \put(3,16.45){\tiny{$k,d\in \N_0\backslash\{1\},d\vert k,r\in\{0,\frac{d}{2}\}\cap\N_0,r\neq 1$}}
\put(5.2, 9){\frame{$\Aspace\langle\singletonw^{\otimes k},\positionerd,\vierpartwbwb,\singletonw\otimes\singletonb\rangle\Aspace$}}
 \put(5.2, 8.45){\tiny{$k,d\in\N_0\backslash\{1\},d\vert k$}}
\multiput(1,3.5)(0,0.5){18}{\line(0,1){.1}}
 \multiput(1.2,3.5)(0,0.5){18}{\line(0,1){.1}}
\multiput(1,15.5)(0,0.5){10}{\line(0,1){.1}}
 \multiput(1.2,15.5)(0,0.5){10}{\line(0,1){.1}}
\multiput(19,3.5)(0,0.5){34}{\line(0,1){.1}}
 \multiput(19.2,3.5)(0,0.5){34}{\line(0,1){.1}}
\multiput(11,22)(0.5,0){14}{\line(1,0){.1}}
 \multiput(11,22.2)(0.5,0){14}{\line(1,0){.1}}
\multiput(13,2)(0.5,0){10}{\line(1,0){.1}}
 \multiput(13,2.2)(0.5,0){10}{\line(1,0){.1}}
\put(18,21){\frame{${ }_{ }\Bspace$\frame{$\Aspace\langle
\paarpartww^{\otimes \frac{k}{2}},\paarpartww\otimes\paarpartbb
\rangle\Aspace$}$\Bspace$}}
 \put(18,20.45){\tiny{$k\in 2\N_0$}}
\put(18,1){\frame{${ }_{ }\Bspace$\frame{$\Aspace\langle
b_k,\vierpartwbwb,\paarpartww\otimes\paarpartbb
\rangle\Aspace$}$\Bspace$}}
 \put(18,0.45){\tiny{$k\in 2\N_0$}}
\put(0,21){\frame{${ }_{ }\Bspace$\frame{$\Aspace\langle
\singletonw^{\otimes k}, \singletonw\otimes\singletonb,\paarpartww\otimes\paarpartbb
\rangle\Aspace$}$\Bspace$}}%
 \put(0,20.45){\tiny{$k\in 2\N_0$}}
\put(0,13){\frame{${ }_{ }\Bspace$\frame{$\Aspace\langle
\singletonw^{\otimes k},\positionerwwbb,\singletonw\otimes\singletonb,\paarpartww\otimes\paarpartbb
\rangle\Aspace$}$\Bspace$}}
 \put(0,12.45){\tiny{$k\in \N_0$}}
\put(0,1){\frame{${ }_{ }\Bspace$\frame{$\Aspace\langle
\singletonw^{\otimes k},\vierpartwbwb,\singletonw\otimes\singletonb,\paarpartww\otimes\paarpartbb
\rangle\Aspace$}$\Bspace$}}
 \put(0,0.45){\tiny{$k\in \N_0$}}
\multiput(5.6,23.1)(0.1,0.1){20}{\circle{0.2}}
\multiput(1,15.1)(0.1,0.1){20}{\circle{0.2}}
\multiput(26,23.1)(0.15,0.1){20}{\circle{0.2}}
\multiput(19,3)(0.1,0.1){20}{\circle{0.2}}
\multiput(1.2,3)(0.1,0.15){40}{\circle{0.2}}
\multiput(13,3)(0.2,0.05){40}{\circle{0.2}}
\put(4.5,24){\tiny{$[d=0]$}}
\put(20.1,3.5){\tiny{$[d\in 2\N_0]$}}
\put(16.5,12){\line(1,0){2}}
\put(17.5,11){\line(0,1){2}}
\put(17,11.5){\tiny{$\SSS$}}
\put(17,12.1){\tiny{$\BBB$}}
\put(17.6,11.5){\tiny{$\HHH$}}
\put(17.6,12.1){\tiny{$\OOO$}}
\end{picture}
\end{center}

We also give the corresponding graphic in the orthogonal case ($\paarpartww\in\CC$), for comparison (see Proposition \ref{PropOneColored}).

\setlength{\unitlength}{0.5cm}
\begin{center}
\begin{picture}(30,12)
\multiput(1,3.5)(0,0.5){3}{\line(0,1){.1}}
 \multiput(1.2,3.5)(0,0.5){3}{\line(0,1){.1}}
\multiput(1,7.5)(0,0.5){3}{\line(0,1){.1}}
 \multiput(1.2,7.5)(0,0.5){3}{\line(0,1){.1}}
\multiput(19,3.5)(0,0.5){11}{\line(0,1){.1}}
 \multiput(19.2,3.5)(0,0.5){11}{\line(0,1){.1}}
\multiput(12.7,10)(0.5,0){11}{\line(1,0){.1}}
 \multiput(12.7,10.2)(0.5,0){11}{\line(1,0){.1}}
\multiput(16.8,2)(0.5,0){3}{\line(1,0){.1}}
 \multiput(16.8,2.2)(0.5,0){3}{\line(1,0){.1}}
\put(18,9){\frame{${ }_{ }\Bspace$\frame{$\Aspace
(k=2):\langle\emptyset\rangle\Aspace$}$\Bspace$}}
\put(18,1){\frame{${ }_{ }\Bspace$\frame{$\Aspace
(k=2):\langle\vierpart\rangle\Aspace$}$\Bspace$}}
\put(0,9){\frame{${ }_{ }\Bspace$\frame{$\Aspace
(k=1): \langle\singleton\rangle\quad
(k=2):\langle\singleton\otimes\singleton\rangle\Aspace$}$\Bspace$}}%
\put(0,5){\frame{${ }_{ }\Bspace$\frame{$\Aspace 
(k=2): \langle\legpart\rangle\Aspace$}$\Bspace$}}
\put(0,1){\frame{${ }_{ }\Bspace$\frame{$\Aspace
(k=1): \langle\singleton,\vierpart\rangle\quad
(k=2): \langle\singleton\otimes\singleton,\vierpart\rangle\Aspace$}$\Bspace$}}
\put(16.5,4){\line(1,0){2}}
\put(17.5,3){\line(0,1){2}}
\put(17,3.5){\tiny{$\SSS$}}
\put(17,4.1){\tiny{$\BBB$}}
\put(17.6,3.5){\tiny{$\HHH$}}
\put(17.6,4.1){\tiny{$\OOO$}}
\end{picture}
\end{center}

\begin{rem}\label{RemParameters}
The constraints on the parameters $k$, $d$ and $r$ in the above theorems can be understood by the fact that we have the following equalities.
\begin{itemize}
\item $\categ{\HHH}{\glob}{k}=\categ{\HHH}{\loc}{k,2}$ and $\categ{\HHH}{\glob}{2m+1}=\categ{\SSS}{\glob}{2m+1}$
\item $\categ{\SSS}{\glob}{k}=\categ{\SSS}{\loc}{k,1}=\categ{\HHH}{\loc}{k,1}$
\item $\categ{\BBB}{\glob}{k}=\categ{\BBB'}{\loc}{k,2,1}$ and $\categ{\BBB}{\glob}{2m+1}=\categ{\BBB'}{\glob}{2m+1}$
\item $\categ{\BBB'}{\glob}{k}=\categ{\BBB'}{\loc}{k,1,0}=\categ{\BBB'}{\loc}{k,1,1}$
\end{itemize}
\end{rem}

\begin{rem}
In an unpublished draft, Banica, Curran and Speicher \cite{speicherunpublished} already found several of the above categories. We thank them for sending the draft to us.
\end{rem}

\section{The group case}\label{SectGroupCase}

For  categories $\CC\subset P^{\twocol}$ of two-colored partitions, there are two natural extreme cases. The first is the one of noncrossing partitions, completely classified in the preceding sections. The second is the one containing the crossing partitions $\crosspartwwww$, $\crosspartbbbb$, $\crosspartwbbw$ and $\crosspartbwwb$, which allow us to permute the points of a partition in an arbitrary way (without changing their colors). It is easy to see that one of these four partitions is in a category if and only if all are (by verticolor reflection and rotation).

\begin{defn}\label{DefCrossCase}
A category of two colored partitions $\mathcal{C}$ is in the \emph{group case} if one (and hence all) of the  partitions $\crosspartwwww$, $\crosspartbbbb$, $\crosspartbwwb$ and $\crosspartwbbw$  is in $\mathcal{C}$.
\end{defn} 

The name ``group case'' refers to the situation when a quantum group is associated to a category of partitions (see \cite{tarragoweberopalg}). If $\mathcal C$ is in the group case, the associated quantum group is in fact a group.

The classification of all categories in the group case follows directly from the classification of all categories of noncrossing partitions and the following lemma.

\begin{lem}\label{LemCapCateg}
 Let $\CC$ and $\DD$ be categories of two-colored partitions.
\begin{itemize}
 \item[(a)] Then $\CC\cap \DD$ is again a category of partitions.
 \item[(b)] Let $\CC$ be in the group case and put $\CC_0:=\CC\cap NC^{\twocol}$. Then $\CC=\langle\CC_0,\crosspartwwww\rangle$.
\end{itemize}
\end{lem}
\begin{proof}
(a) This follows directly from the definition of a category.

(b) Let $p\in\CC$. Using the four kinds of crossing partitions of Definition \ref{DefCrossCase}, we may permute the points of $p$ such that we obtain a noncrossing partition $p'$. Since this can be done in $\CC$, we have $p'\in\CC_0\subset\langle\CC_0,\crosspartwwww\rangle$. Thus, we can also reconstruct $p$ in $\langle\CC_0,\crosspartwwww\rangle$ doing all these operations backwards, so $\CC\subset\langle\CC_0,\crosspartwwww\rangle$. We deduce that equality holds.
\end{proof}

For each category of partition $\mathcal{C}$, we put $\mathcal{C}_{\grp}:=\langle \mathcal{C},\crosspartwwww\rangle$. Thus the preceding lemma says that any category of partition in the group case is of the form $\mathcal{C}_{\grp}$ for a category $\CC$ of non-crossing partitions.

\begin{thm}
The categories in the group case are the following.

\begin{itemize}
\item $\categ{\OOO}{\grp,\glob}{k}:=\langle\paarpartww^{\otimes \frac{k}{2}},\paarpartww\otimes\paarpartbb,\crosspartwwww\rangle$ for $k\in 2\N_0$
\item $\categg{\OOO}{\grp,\loc}:=\langle\crosspartwwww\rangle$
\item $\categ{\HHH}{\grp,\glob}{k}:=\langle b_k,\vierpartwbwb,\paarpartww\otimes\paarpartbb,\crosspartwwww\rangle$ for $k\in 2\N_0$
\item $\categ{\HHH}{\grp,\loc}{k,d}:=\langle b_k,b_d\otimes\tilde b_d,\vierpartwbwb,\crosspartwwww\rangle$ for $k,d\in\N_0\backslash\{1,2\}$, $d\vert k$
\item $\categ{\SSS}{\grp,\glob}{k}:=\langle \singletonw^{\otimes k},\vierpartwbwb,\singletonw\otimes\singletonb,\paarpartww\otimes\paarpartbb,\crosspartwwww\rangle$ for $k\in \N_0$
\item $\categ{\BBB}{\grp,\glob}{k}:=\langle\singletonw^{\otimes k}, \singletonw\otimes\singletonb,\paarpartww\otimes\paarpartbb,\crosspartwwww\rangle$ for $k\in 2\N_0$
\item $\categ{\BBB}{\grp,\loc}{k}:=\langle \singletonw^{\otimes k}, \singletonw\otimes\singletonb,\crosspartwwww\rangle$ for $ k\in \N_0$ 
\end{itemize}
\end{thm}
\begin{proof}
The list of categories in the group case is exactly given by all $\langle\CC_0,\crosspartwwww\rangle$, where $\CC_0\subset NC^{\twocol}$ is a category of noncrossing partitions (Lemma \ref{LemCapCateg}). Note that $\positionerd$ is in $\langle\singletonw\otimes\singletonb,\crosspartwwww\rangle$. Moreover:
\begin{align*}
&\langle\categ{\HHH'}{\loc}{k},\crosspartwwww\rangle=\langle\categ{\HHH}{\loc}{0,0},\crosspartwwww\rangle\\
&
\langle\categ{\SSS}{\loc}{k,d},\crosspartwwww\rangle=\langle\categ{\SSS}{\glob}{k},\crosspartwwww\rangle\\
\textnormal{and }
&\langle\categ{\BBB'}{\glob}{k},\crosspartwwww\rangle
=\langle\categ{\BBB}{\glob}{k},\crosspartwwww\rangle
=\langle\categ{\BBB'}{\loc}{k,d,0},\crosspartwwww\rangle
=\langle\categ{\BBB'}{\loc}{k,d,\frac{d}{2}},\crosspartwwww\rangle
\end{align*}
\end{proof}

\section{Concluding remarks}\label{SectConclRem}

\subsection{From categories of partitions to compact quantum groups}

In \cite{tarragoweberopalg}, we define unitary easy quantum groups using categories of two-colored partitions. The first step is to associate a universal $C^*$-algebra to a category of partitions by assigning certain algebraic relations to any partition. This $C^*$-algebra can then be endowed with a comultiplication turning it into a compact matrix quantum group in the sense of Woronowicz \cite{woronowicz1987compact}. Another way of obtaining this quantum group is to turn the category of partitions into a concrete monoidal $W^*$-category in the sense of Woronowicz \cite{woronowicz1988tannaka}. Using his Tannaka-Krein result \cite{woronowicz1988tannaka}, we thus obtain the quantum group via its intertwiner spaces. Quantum groups obtained this way are called unitary easy quantum groups, extending the definition of Banica and Speicher's orthogonal easy quantum groups \cite{banica2009liberation}. If a category consists only of noncrossing partitions, we call the associated quantum group a free easy quantum group. 
See \cite{tarragoweberopalg} for remarks on the use of easy quantum groups for the theory of compact matrix quantum groups, for free probability and for other links.

\subsection{Open problems in the classification of categories of partitions}

In this article, we classified all (two-colored) categories of noncrossing partitions, thus all categories $\langle\emptyset\rangle\subset\CC\subset NC^{\twocol}$. Furthermore, we classified all categories $\langle\crosspartwwww\rangle\subset\CC\subset P^{\twocol}$. Moreover, we know all categories $NC^{\twocol}\subset\CC\subset P^{\twocol}$ from \cite{banica2009liberation} -- there are none besides $NC^{\twocol}$ and $P^{\twocol}$.
However, so far we have no result on determining the categories $\langle\emptyset\rangle\subset\CC\subset\langle\crosspartwwww\rangle$, which would be a natural next step to do in order to complete the classification of all four sides of the following square:
\begin{align*}
NC^{\twocol} &&\supset && \langle\emptyset\rangle\\
\subsetdown\quad && &&\subsetdown\quad\\
P^{\twocol} &&\supset && \langle\crosspartwwww\rangle
\end{align*}
Of course, the long term goal would be also to classify the diagonal of that square, hence  all  categories $\CC\subset P^{\twocol}$.
Here, following the history of classification in the orthogonal case, the strategy could be first to consider the half-liberated case. This amounts to classifying all categories containing some half-liberated partitions $\halflibpartwwwwww$ (with various colorings) but not  the crossing partition $\crosspartwwww$. The category $\langle\halflibpartwwwwww\rangle$ is one such example.

In the orthogonal case, the full classification is achieved along the lines of a division into non-hyperoctahedral categories (not so difficult) and hyperoctahedral ones (difficult). In the unitary case, hyperoctahedral categories should be those containing $\vierpartwbwb$ but not $\singletonw\otimes\singletonb$. It is likely that the classification of non-hyperoctahedral categories is more or less immediately doable, but as for the hyperoctahedral ones, it is unclear whether the methods of \cite{raum2014combinatorics} and \cite{raum2013full} can be applied directly.

\subsection{Using more colors}

Once the step from non-colored partitions to colored partitions is done, the question is: why only two colors? From the combinatorial point of view, it is straightforward to define categories of partitions having $n$ colors and $n$ inverse colors (thus, our two-colored partitions would be the case $n=1$). Such partitions appear for instance in Freslon's work on partition quantum groups \cite{freslon2014partition}. It would be nice to see if we again obtain a much wider variety of noncrossing categories, when passing from $n=1$ to $n>1$.






\bibliographystyle{alpha}
\nocite{*}
\bibliography{bibliographie}

\end{document}